\documentclass[12pt,a4paper,reqno]{amsart} 

\usepackage[T1]{fontenc}
\usepackage[letterpaper, margin=0.8in]{geometry}
\usepackage{amsfonts, amssymb, amsthm, epsf, graphics, bbm, mathabx, verbatim, caption, subcaption, enumerate}
\usepackage{comment}
\usepackage{hyperref}
\usepackage{diagbox}
\usepackage{amsmath}
\usepackage{amsthm}
\usepackage{amssymb}
\usepackage{mathtools}
\usepackage{mathrsfs}
\usepackage{tikz}
\usetikzlibrary{patterns}
\usepackage{float}
\usepackage{enumitem}
\usepackage{pifont}
\usepackage[numbers]{natbib}

\usepackage{pgfplots}
\pgfplotsset{compat=1.18} 
\usepackage{graphicx}
\usepackage{subcaption}

\usepackage{graphicx}

\newcommand{\medcup}{\mathbin{\scalebox{1.4}{$\cup$}}}
\newcommand{\medcap}{\mathbin{\scalebox{1.4}{$\cap$}}}

\newtheorem{theorem}{Theorem}[section]
\newtheorem{lemma}[theorem]{Lemma}
\newtheorem{corollary}[theorem]{Corollary}
\newtheorem{proposition}[theorem]{Proposition}
\newtheorem{definition}[theorem]{Definition}

\theoremstyle{definition}

\newtheorem{example}{Example}
\newtheorem{fact}{Fact}

\newtheorem{remark}{Remark}

\newcommand{\R}{\mathbb{R}}
\newcommand{\Z}{\mathbb{Z}}

\newcommand{\J}{\mathcal{J}}

\newcommand{\dd}{\mathrm{d}}

\renewcommand{\AA}{\mathcal{A}}

\newcommand{\eps}{\varepsilon}

\newcommand{\spt}{\mathrm{spt}}

\newcommand{\norm}[1]{\left\|#1\right\|}

\newcommand{\tnorm}[1]{\vert\mkern-2mu\vert\mkern-2mu\vert#1\vert\mkern-2mu\vert\mkern-2mu\vert}

\title[Weighted mixed-norm estimates for circular averages]{Weighted mixed-norm estimates for circular averages\\ and exceptional set estimates for the wave equation}
\author[Yixuan~Pang~and~Chenjian~Wang]{Yixuan~Pang~and~Chenjian~Wang}
	
\date{\today}

\address{Department of Mathematics, University of Pennsylvania, Philadelphia, PA 19104, United States}
\email{pyixuan@sas.upenn.edu}

\address{Department of Mathematics, University of British Columbia, Vancouver, BC V6T 1Z2, Canada}
\email{chjwang@math.ubc.ca}



\begin{document}

\begin{abstract}
We prove mixed-norm estimates for circular averages with respect to $\alpha$-dimensional fractal measures on $\R^2$, using circle tangency bounds when $\alpha \in (0,1]$ and a $\delta$-discretized slicing lemma for fractals when $\alpha \in (1,2]$. The former estimate is sharp, while the latter improves previous results for $\alpha \in (\frac{3}{2},2]$. These estimates can be viewed as X-ray-type extensions of Wolff's and Bourgain's circular maximal functions. As applications, we obtain new exceptional set estimates for the radial integrability of functions in Lebesgue spaces, as well as for the H\"older regularity in time of solutions to the linear wave equation on $\R^2$. The latter results are the first of their kind.
\end{abstract}

\maketitle

\section{Introduction}\label{sec:introduction}
Let $n\geq 2$ and $\mathbb{S}^{n-1}$ be the unit sphere in $\R^n$. 

Define the \textit{spherical averages}
\begin{equation*}\label{Eq:SphericalAveraging}
    \AA^n f(x,r):=\int_{\mathbb{S}^{n-1}} f(x-ry) \mathrm d\sigma(y), \quad x\in \R^n, r>0.
\end{equation*}
Here $\sigma$ is the normalized measure on $\mathbb{S}^{n-1}$. In this paper, we will prove certain \textit{weighted mixed-norm estimates} for $\AA^2$, the \textit{circular averages}, and use them to provide new exceptional set estimates in integral geometry and partial differential equations. To better motivate this type of estimates, we briefly review some previous results.

\subsection{History and motivation}\label{subsec:History}
\subsubsection{Two classical types of spherical maximal functions}\phantom{.}

Spherical maximal functions such as $f \mapsto \sup_{r>0} |\AA^n f(\cdot,r)|$ have long been a central topic in harmonic analysis. Stein \cite{Stein76} ($n \geq 3$) and Bourgain \cite{Bourgain86} ($n=2$) proved that it is bounded in $L^p(\R^n)$ if and only if $p > \frac{n}{n-1}$. If the supremum is taken over a compact interval such as $I \coloneqq [1,2]$, then we expect \textit{$L^p$ improving estimates} of the form
\begin{align}\label{Eq:local_maximal}
    \norm{\sup_{r\in I} |\AA^n f(\cdot, r)|}_{L^q(\R^n)} \leq C_n \norm{f}_{L^p(\R^n)}
\end{align}
to hold for some $q>p$. Schlag \cite{Schlag97} ($n=2$) and Schlag-Sogge \cite{SS97} ($n\geq3$) obtained the sharp range of $(p,q)$ except for endpoints and boundaries, which was later studied by Lee \cite{Lee03}. In all these works, Fourier analysis plays an important role, while the $n=2$ case also relies on delicate geometric and combinatorial arguments. Notably, Mockenhaupt-Seeger-Sogge \cite{MSS92} gave a new proof of Bourgain's result \cite{Bourgain86} based on their \textit{local smoothing} estimates for the wave equation, and the idea was later used in \cite{SS97} to also recover Schlag's result \cite{Schlag97}. 

Despite the power of Fourier analytical tools, geometric methods offer a different perspective by focusing on the slightly weaker $\delta$-discretized version of (\ref{Eq:local_maximal}), and can sometimes be more effective. Given $0 < \delta < 1$, $x\in\R^n$, and $r \in [1,2]$, let $S_\delta(x,r) \coloneqq \{x\in\R^n: ||x| - r| \leq \delta\}$. Define the \textit{$\delta$-discretized spherical averages}
\begin{equation*}\label{Eq:discretized_SphericalAveraging}
    \AA^n_\delta f(x,r):=\frac{1}{\mathcal L^n(S_\delta(x,r))}\int_{S_\delta(x,r)}f(y)\mathrm dy,
\end{equation*}
where  $\mathcal{L}^n(\cdot)$ is the $n$-dimensional Lebesgue measure. Now  (\ref{Eq:local_maximal}) basically comes down to finding $(p,q)$ such that for any $\eps>0$, we have 
\begin{align}\label{Eq:local_maximal_delta}
    \norm{\sup_{r\in I} |\AA_\delta^n f(\cdot, r)|}_{L^q(\R^n)} \leq C_{n,\eps} \delta^{-\eps} \norm{f}_{L^p(\R^n)}.
\end{align}
Note that in both (\ref{Eq:local_maximal}) and (\ref{Eq:local_maximal_delta}), we may without loss of generality assume $f$ to be nonnegative, thus removing absolute values. Via duality, (\ref{Eq:local_maximal_delta}) can be reduced to a purely geometric problem about overlapping patterns of $\delta$-annuli. The $\delta^{-\eps}$-loss comes from pigeonholing, and can often be removed by further technical treatment. In this way, Schlag \cite{Schlag98} ($n=2$) fully recovered his previous result in \cite{Schlag97}. Recently, Hickman-Jančar \cite{HJ25} ($n \geq 2$) provided a simple geometric proof of (\ref{Eq:local_maximal_delta}) for $p=q \geq \frac{n}{n-1}$, thus essentially recovered Stein's result \cite{Stein76}.

Another influential maximal estimate of the form
\begin{align}\label{Eq:Wolff_maximal}
    \norm{\sup_{x\in \R^n} |\AA_\delta^n f(x,\cdot)|}_{L^q(I)} \leq C_{n,\eps} \delta^{-\eps} \norm{f}_{L^p(\R^n)} \quad (\forall\,\eps>0)
\end{align}
was introduced by Kolasa-Wolff \cite{KW99}. They showed that (\ref{Eq:Wolff_maximal}) cannot hold if $q > p$ (\cite[Proposition 2.1]{KW99}), and the $\delta^{-\eps}$-loss on the right-hand side cannot be removed in general \cite[Proposition 2.2]{KW99}. In other words, there is no $L^p$ improving property, and one cannot expect a bound independent of $\delta$ like (\ref{Eq:local_maximal}) to hold. As for $L^p$ boundedness, sharp results were obtained by Kolasa-Wolff \cite{KW99} ($p \geq 2$) for $n \geq 3$, and by Wolff \cite{Wolff97} ($p \geq 3$) for $n=2$. Again, the $n=2$ case is much more complicated, and the proof in \cite{Wolff97} borrows ideas from computational geometry. Recently, Wolff's result was generalized by Pramanik-Yang-Zahl \cite{PYZ24} to $1$-dimensional fractal families of $C^2$ curves satisfying Sogge’s cinematic curvature condition \cite{Sogge91}. One may consult \cite{PYZ24} for connections with restricted projection in $\R^3$ and the circular Furstenberg problem in $\R^2$, Wang-Zahl \cite{WZ26} for striking applications to the sticky Kakeya problem in $\R^3$, and Zahl \cite{Zahl25} for further multi-parameter developments in $\R^2$. It is worth pointing out that all known proofs of the $n=2$ case rely on circle tangency bounds and are purely geometric. The sharp local smoothing estimate in $\R^2$ Guth-Wang-Zhang \cite{GWZ20} only gives $p \geq 4$ (this was first mentioned in Kolasa-Wolff \cite{KW99} and rigorously proved in Ham-Ko-Lee \cite{HKL22}).

\subsubsection{A unified framework}\phantom{.}

One might wonder if we could unify the formulations for different types of spherical maximal functions. For example, the scale-dependent estimate (\ref{Eq:Wolff_maximal}) for $\AA_\delta^n$ can indeed be written as an estimate for $\AA^n$ with no superficial $\delta^{-\eps}$-loss, as long as we replace $L^p$ with $L_\eps^p$, which is the $L^p$ Sobolev space of order $\eps$ on the right-hand side, see Schlag \cite[Section~4]{Schlag03}). A more significant issue is that different maximal functions involve suprema over parameter spaces that are markedly different. 
For applications mentioned before, one may also want to include generalizations in Pramanik-Yang-Zahl \cite{PYZ24} of fractal type, which was originally only stated in certain ``dual $\delta$-discretized'' form. It turns out that introducing estimates for $\AA^n$ with respect to a fractal weight is both natural and unavoidable. More precisely, for $\alpha\in (0,d]$ ($d\in \mathbb{N}$ will typically be $n$ or $n+1$), we denote by $\mathfrak{C}^d(\alpha)$ the class of nonnegative Borel measure $\mu$ on $\R^{d}$ such that 
\begin{equation*}\label{Eq:FrostmanConstant}
    \langle \mu\rangle_\alpha:=\sup_{x\in \R^d,\rho\in (0,1]}\frac{\mu(B^d(x,\rho))}{\rho^\alpha}
\end{equation*}
is finite. Note that $\mu$ naturally satisfies the following \textit{$\alpha$-dimensional ball condition}: 
\begin{equation*}\label{Eq:BallCondition}
    \mu(B^d(x,\rho))\leq  \langle \mu\rangle_\alpha \rho^\alpha, \quad \forall\, x\in \R^d, \rho\in (0,1] .
\end{equation*}
By the Kolmogorov-Seliverstov-Plessner linearization \cite[Lemma 3.2]{HKL22}, one can cast all the spherical maximal estimates in the unified form ($\alpha\in(0,n+1]$: $\alpha=n$ for (\ref{Eq:local_maximal}), $\alpha=1$ for (\ref{Eq:Wolff_maximal})):
\begin{align}\label{Eq:weighted_form}
    \norm{\AA^n f}_{L^q(\R^{n}\times I;\overline{\nu})} \lesssim \langle\overline{\nu}\rangle_\alpha^{\frac{1}{q}} \tnorm{f},
\end{align}
where $\langle\overline{\nu}\rangle_\alpha^{\frac{1}{q}}$ is natural for homogeneity reasons, and $\tnorm{\cdot}$ can be any reasonable norm (which will typically be $L^p(\R^n)$ or $L_\gamma^p(\R^n)$ for some $\gamma$). The form (\ref{Eq:weighted_form}) is equivalent to certain \textit{fractal/weighted local smoothing estimates} for the half-wave operator, which was first studied by Wolff \cite{Wolff00} (see \cite{CHL17, HKL22, CDK25} for modern expositions). Treating variables $(x,r)$ symmetrically, it is amenable to space-time Lorentz rescaling and many techniques from \textit{weighted restriction} theory. The form (\ref{Eq:weighted_form}) for general $\alpha$ was first considered by Oberlin \cite{Oberlin06, Oberlin07} when studying general sphere packing problems, and systematically revisited using more advanced tools by Ham-Ko-Lee \cite{HKL22}.

\subsubsection{Spherical maximal estimates with respect to fractal weights}\phantom{.}

An interesting special case of (\ref{Eq:weighted_form}) is \textit{weighted spherical maximal estimates} of the form
\begin{align}\label{Eq:weighted_maximal}
    \norm{\sup_{r\in I}|\AA^n f(\cdot,r)|}_{L^q(\nu)} \lesssim \langle\nu\rangle_\alpha^{\frac{1}{q}} \norm{f}_{L^p(\R^n)}
\end{align}
for $\nu \in \mathfrak{C}^n(\alpha)$. Due to counterexamples given by Talagrand \cite{Talagrand80} ($n=2$) and Mitsis \cite[Proposition~3.2]{Mitsis99} ($n\geq 3$), one cannot expect (\ref{Eq:weighted_maximal}) to hold in general when $\alpha \leq 1$. When $\alpha > 1$ and $n \geq 3$, (\ref{Eq:weighted_maximal}) has been widely studied. The first positive result dates back to Mitsis \cite{Mitsis99}, who proved (\ref{Eq:weighted_maximal}) for $p=q=2$, thus resolving the sphere packing problem regarding centers of dimension $\alpha>1$. Motivated by bounding the Hausdorff dimension of blowup sets of $\sup_{r\in I}|\AA^n f(\cdot,r)|$, Iosevich et al. \cite{IKSTU19} considered general $p=q$ cases. All these results were later superseded by general $L^p$ improving estimates of Ham-Ko-Lee \cite[Corollary~3.5]{HKL22}. When $\alpha>1$ and $n=2$, the only known result is \cite[Corollary~1.4]{HKL22}. 

Historically, $L_\gamma^2$-based (i.e., $H^\gamma$-based) weighted maximal estimates for dispersive operators were extensively studied for the sake of determining the sizes of the divergence sets \cite{BBCR11}, which is a natural refinement of Carleson-type problem on almost everywhere convergence. See \cite{LR19, HKL22diverge} and the references therein for results on the half-wave operator closely related to $\AA^n$. See also \cite{Du_Zhang_WeightedRestriction,LR19, EP22} and the references therein for results on the Schr\"odinger operator. However, the estimate (\ref{Eq:weighted_maximal}) is more subtle since a complete understanding of it seems to require $L_\gamma^p$-based weighted local smoothing estimates with optimal regularity $\gamma$, which remain unresolved to date except for some special cases \cite{HKL22}. Even if $n=2$, the sharp (unweighted) local smoothing estimate due to Guth-Wang-Zhang \cite{GWZ20} only resolves the cases $\alpha \in (3 - \sqrt{3}, 3]$ \cite[Corollary~1.4]{HKL22}, leaving $\alpha \in (1, 3-\sqrt{3}]$ wide open. It is worth mentioning that the limiting case $\alpha = 1+$, corresponding to the previously mentioned result of Pramanik-Yang-Zahl \cite{PYZ24}, is also fully resolved \cite[Theorem~1.13]{CDK25}.

Therefore, the $n=2$ case of (\ref{Eq:weighted_maximal}) should be of particular interest, as it could tell how much remains to be explored about the incidence geometry of circles beyond local smoothing techniques. Also, the two extremes $\alpha=2$ and $\alpha=1$ correspond to the two classical types mentioned at the beginning, which are well understood. To be more precise, let $\nu \in \mathfrak{C}^2(\alpha)$. By \textit{Bourgain's circular maximal function estimate}, we mean
\begin{align}\label{Eq:Bourgain_Wolff}
    \norm{\sup_{r\in I} |\AA^2 f(\cdot,r)|}_{L^p(\nu)} \leq C_p 
    \langle \nu \rangle_\alpha^{\frac{1}{p}}\norm{f}_{L^p(\R^2)}
\end{align}
where $\alpha=2$ and $p=2+$ (i.e., $\dd\nu/\dd\mathcal{L}^2 \in L^\infty(\R^2)$), or the essentially equivalent 
\begin{align}\label{Eq:Bourgain_Wolff_delta}
    \norm{\sup_{r\in I} |\AA_\delta^2 f(\cdot,r)|}_{L^p(\nu)} \leq C_\eps \delta^{-\eps}
    \langle \nu \rangle_\alpha^{\frac{1}{p}}\norm{f}_{L^p(\R^2)} \quad (\forall\,\eps>0)
\end{align}
where $\alpha=2$ and $p=2$. On the other hand, by \textit{Wolff's circular maximal function estimate}, we mean (\ref{Eq:Bourgain_Wolff}) where $\alpha=1+$ and $p=3+$, or (\ref{Eq:Bourgain_Wolff_delta}) where $\alpha=1$ and $p=3$. The latter term should be appropriate, because Wolff \cite{Wolff97} pointed out that his arguments work equally well to resolve variants of (\ref{Eq:Wolff_maximal}) ($n=2$) like 
\begin{align*}
    \norm{\sup_{r\in I, x_2 \in \R} |\AA_\delta^2 f(\cdot,x_2,r)|}_{L^3(\R)} \leq C_\eps \delta^{-\eps}
    \langle \nu \rangle_\alpha^{\frac{1}{p}}\norm{f}_{L^3(\R^2)} \quad (\forall\,\eps>0),
\end{align*}
which is clearly a special case of (\ref{Eq:Bourgain_Wolff_delta}) ($\alpha=1$, $p=3$) by linearization, and is in fact not much weaker than it in view of Marstrand's projection theorem in $\R^2$.

\subsubsection{Weighted mixed-norm estimates}\phantom{.}\label{minorsec:weighted_mixed_norm}

Now that we have seen weighted estimates for spherical averages and their usefulness in geometric measure theory and partial differential equations, it is time to turn to the main focus of this paper. Given a set $\Omega \subseteq \R^n$, a measure $\nu \in \mathfrak{C}^n(\alpha)$, and a function $F$ supported on $\R^n \times I$, we define \textit{weighted mixed-norms} of $F$ on $\Omega$ with respect to $\nu$ as
\begin{align*}
    \norm{F}_{L_x^q(\Omega,\nu)(L^s_r(I))}:=\left(\int_{\Omega}\Big(\int_{I}|F|^s\mathrm dr\Big)^\frac{q}{s}\mathrm d\nu(x)\right)^{\frac{1}{q}},
\end{align*}
where $q,s\in [1,\infty]$.\footnote{{Here and below, we may use $L^p_r$ or $L^s_r$ to denote the integration over $r$, this should not be confused too much with the notion of Sobolev space.}} Measurability issues can happen in general, but we do not need to worry about it in all our settings. Roughly speaking, by \textit{weighted mixed-norm estimates} for $\AA^n$ we mean 
\begin{align}\label{Eq:General_MixedNorm}
    \norm{\AA^n f}_{L_x^q(\R^n,\nu)(L^s_r(I))} \lesssim \langle\nu\rangle_\alpha^{\frac{1}{q}}\norm{f}_{L^p(\R^n)},
\end{align}
which reduces to (\ref{Eq:weighted_maximal}) when $s=\infty$. Before everything, let us first elaborate on the backgrounds and motivations.
\begin{enumerate}[label={\rm (\roman*)}]
    \item \textbf{Necessity for $s < \infty$}: The first thing one might ask is why should we care about $s<\infty$, except that such an extension is natural from the point of view of harmonic analysis. The main reason is that in many applications (that will soon be mentioned below) regarding exceptional set estimates, one often cares about $\nu$ of low dimension $\alpha$. It would of course be good if we could still prove $s=\infty$ estimates, which is the strongest (once proved). Unfortunately, for a given $p$, if $\alpha$ is small, it may even be the case that one cannot expect (\ref{Eq:General_MixedNorm}) with $s=\infty$ to hold for any choice of $q$ (the best choice is $q=1$ by H\"older's inequality). A convincing example might be that, by Example~4 in Section~\ref{sec:necessary} ($n=2$), when $\alpha<4-p$, we at least need $s \leq \frac{2p}{4-\alpha-p}<\infty$.
    
    \item  \textbf{Singularity-integrability balance}: Following (i), an instructive intuition is that, when $\alpha$ becomes smaller, $\nu$ becomes more singular and the ranges of $(p,q,s)$ for (\ref{Eq:General_MixedNorm}) must shrink in some sense. If we choose to stick to $s=\infty$, then we are prevented from studying many interesting values of $p$. In this sense, considering general mixed-norm with $L_r^s$ ($s<\infty$) becomes natural, as we need to sacrifice some integrability in $r$ to balance such singularity in $x$. And a comprehensive understanding of (\ref{Eq:General_MixedNorm}) will quantitatively tell us how much do we need to sacrifice.

    \item \textbf{A forecast of our main result}: In particular, when $n=2$, by Talagrand's counterexample mentioned before, Wolff’s circular maximal function estimate of the form (\ref{Eq:Bourgain_Wolff}) fails when $\alpha=1$ and $p=q=3$. Also, even if we take $\nu = \mathcal{L}^2$ ($\alpha=2$), it is well known \cite{Bourgain86} that Bourgain’s circular maximal function estimate of the form (\ref{Eq:Bourgain_Wolff}) fails when $p=q=2$. As we will see in our main result (Theorem~\ref{Thm:MainTheorem_0}), we have (\ref{Eq:General_MixedNorm}) as long as $s<\infty$ in both cases (i.e., sacrifice only a little bit integrability in $r$). In general, the question is if we fix $p,q$ and make $\alpha$ smaller, how small should $s$ be to guarantee the validity of (\ref{Eq:General_MixedNorm})? Our Theorem~\ref{Thm:MainTheorem_0} also gives new results on this.
    
    \item \textbf{The $q=s$ case}: If $q=s$, then (\ref{Eq:General_MixedNorm}) is equivalent to studying (\ref{Eq:weighted_form}) in the special case when $\overline{\nu} = \nu \otimes \mathcal{L}^1|_I$ (with $\alpha$ replaced by $\alpha+1$). By interpolation, this fact allows us to use non-mixed-norm estimates (e.g., \cite{HKL22}) to obtain mixed-norm estimates. In Section~\ref{subsec:compare}, we will see that this may indeed yield new results on (\ref{Eq:General_MixedNorm}) in some cases (Figure~(\subref{fig:middle})). One reason might be that the current tools for (\ref{Eq:General_MixedNorm}) when $q \neq s$ are very limited due to the anisotropic structure, and going back to (\ref{Eq:weighted_form}) makes many rescaling-dependent techniques applicable.
    
    \item \textbf{Necessity for $q \neq s$}: However, it seems that for applications (that will soon be mentioned below) regarding exceptional set estimates, the resulting thresholds only depend on $s$ (for fixed $p$), the procedure in (iv) would not help much, as the strong constraint $q=s$ could make $s$ far from optimal. This justifies the necessity of studying (\ref{Eq:General_MixedNorm}) for general $q \neq s$. Our Theorem~\ref{Thm:MainTheorem_0} also falls into this regime, and we will come back to this point in Section~\ref{subsec:apply}. 

    \item \textbf{Pinned distance sets}: The only previous work we are aware of that deals with general $q\neq s<\infty$ is that of D. Oberlin and R. Oberlin \cite[Theorem~1.3]{Oberlin15} (see also discussions in Section~\ref{subsec:compare}). Their bounds imply exceptional set estimates for pinned distance sets in subcritical cases \cite[Theorem~2.4]{Oberlin15}, via a relation between the dimension of a measure and the $L^p$ (or $L^s$) boundedness of its convolution with a singular potential.

    \item \textbf{Integral geometry}: For a function $f\in L^p(\R^n)$, we clearly have $\AA^n f(x,\cdot) \in L_r^p[1,2]$ for any $x\in\R^n$ by Fubini's theorem. Due to the curvature of spheres, one might expect improved generic radial integrability (i.e., $\AA^n f(x,\cdot) \in L_r^s[1,2]$ for some $s>p$) outside of an exceptional set of $x$. A natural question is to bound the Hausdorff dimension of it. Via some standard real analysis, this exactly comes down to weighted mixed-norm estimates like (\ref{Eq:General_MixedNorm}). Nevertheless, we find no formal statements (for general $s<\infty$) in the literature, so we choose to explicitly give one (Corollary~\ref{Cor:integral_geometry}). 

    \item \textbf{PDEs}:  Assign an initial velocity to an oscillating system, and suppose that an observer standing at a position in the system is detecting fluctuations. A natural question is how ``smooth'' is the wave that the observer sees as time goes on. Physically speaking, most positions should see waves with high regularity, while in rare positions where significant resonance happens frequently, one can only expect to see waves with low regularity. A natural question is how ``rare'' could it be. It turns out that weighted mixed-norm estimates (\ref{Eq:General_MixedNorm}) can be used to study this problem in terms of the Hausdorff dimension. In particular, for the Cauchy problem of the linear wave equation on $\R^2$, by using our Theorem~\ref{Thm:wave_eqn}, we get exceptional set estimates for the H\"older regularity in time, which are the first of their type in the literature. The transition scheme is nontrivial and involves many technicalities (Section~\ref{sec:wave_eqn}). We will provide more historical remarks in Section~\ref{subsec:further_direction}.
\end{enumerate}

\subsection{Main results}\label{subsec:Main_result}

From now on, we abbreviate $\AA^n$ (or $\AA_\delta^n$) as $\AA$ (or $\AA_\delta$) when $n$ 
is clear from the context. Let $B \coloneqq B^n(0,1/4)$.  Our main result is the following \textit{localized} weighted mixed-norm estimate in $\R^2$:
\begin{theorem}\label{Thm:MainTheorem_0} Let $n=2$. 
There is a constant $C>0$ such that
    \begin{equation}\label{Eq:MixedNormEstimate_0}
        \norm{\AA f}_{L
        _x^q(B,\nu)(L^s_r(I))} \leq C \langle\nu\rangle_\alpha^{\frac{1}{q}}\norm{f}_{L^p(\R^n)},
    \end{equation}
    when 
    \begin{enumerate}[label={\rm (\arabic*)}]
        \item \label{Item1OfMainTheorem_0}$p = q = 3,\alpha\in (0,1]$, $1 \leq s < \frac{3}{1-\alpha};$
        \item\label{Item2OfMainTHeorem_0} $p=q=2,\alpha\in (1,2]$, $1\leq s < \frac{2}{2-\alpha}.$
    \end{enumerate}
    Moreover, the result in Case {\rm(1)} is essentially sharp given $p=3$.
\end{theorem}
\begin{remark}
    We will always assume $p,q,s \in [1,\infty]$. We call (\ref{Eq:MixedNormEstimate_0}) the \textit{localized} version of (\ref{Eq:General_MixedNorm}) as we have replaced $\R^n$ with $B$ on the left hand side. However, due to the local property of $\AA$ ($r\in I$), it is easy to show that these two versions are actually equivalent to each other whenever $q \geq p$ (which is the case for us). We will provide details for such kind of equivalence whenever needed, and the reader should now just take (\ref{Eq:MixedNormEstimate_0}) as our main goal.
\end{remark}
\begin{remark}
    In Case (1), when $\alpha=1$ and $s \rightarrow \infty$, we are essentially back to Wolff's circular maximal function. In Case (2), when $\alpha=2$ and $s \rightarrow \infty$, we are essentially back to Bourgain's circular maximal function. As mentioned in the previous subsection, in both cases, we cannot expect (\ref{Eq:MixedNormEstimate_0}) to hold without scale loss if $s = \infty$ (maximal estimates), so it is natural to sacrifice the $s$-endpoints and consider the general weighted mixed-norm estimates. 
\end{remark}

Roughly speaking, Case (1) of Theorem~\ref{Thm:MainTheorem_0} is proved by adapting circle tangency bounds (\cite{Wolff97, PYZ24}) to the mixed-norm setting, while Case (2) is proved by first developing a $\delta$-discretized slicing lemma (Theorem~\ref{Thm:DiscretizedSlicing}) for fractals and then carefully work out an $L^2$ Córdoba-type argument on each slice. Both of the approaches are new in the study of (\ref{Eq:MixedNormEstimate_0}) for $s<\infty$. Theorem~\ref{Thm:DiscretizedSlicing} is pretty general in that it holds in any dimension (not just $n=2$), and does not seem to exist in its form in the literature. It will be proved by applying discretization procedures to Mattila's slicing lemma for singular sliced measures, which were first proved without using Fourier analysis \cite{Mattila81}. In this sense, our proof of Theorem~\ref{Thm:MainTheorem_0} is purely geometric. Working in discretized settings is technically important for us, and we hope that the paradigm here is of independent interest and has the potential to be applied to other problems.

As mentioned in Items (iv) and (v) of Section~\ref{minorsec:weighted_mixed_norm}, Fourier analytic approaches does not seem to work effectively for (\ref{Eq:MixedNormEstimate_0}) due to lack of scaling invariance. Previously, D. Oberlin and R. Oberlin \cite{Oberlin15} also adopted purely geometric methods by using a ``global'' $L^2$ Córdoba-type argument (similar techniques were also used in \cite{Oberlin06, Oberlin06radon, Oberlin07, Oberlin12}). In contrast, our $L^2$ arguments for Case (2) are more ``local'' as they only happen on each slice. It turns out that these two ways do not contain each other in general, and we will systematically compare the results in Section~\ref{subsec:compare}. Here we only point out that the ranges of $s$ in Theorem~\ref{Thm:MainTheorem_0} are better than previous results when $\alpha \in [\frac{1}{2},1]$ or $\alpha \in (\frac{3}{2},2]$, i.e., when $\alpha$ is not too small in both cases.

\begin{remark}
    Our proof should extend to any $3$-dimensional family of \textit{cinematic curves} in $\R^2$ satisfying the \textit{transversality condition} (see \cite{Zahl25}). We mainly focus on circular averages for two reasons besides historical considerations in \ref{subsec:History}. Firstly, they possess rotational symmetry and make the pictures simpler, thus avoiding some technicalities. Secondly, we exactly need them for later applications to wave equations (Theorem~\ref{Thm:wave_eqn}).
\end{remark}

By further interpolating Theorem~\ref{Thm:MainTheorem_0} with the results in \cite{HKL22, Oberlin15} and trivial bounds, it is plausible that one can get many new ranges of the quadruple $(p,q,s,\alpha)$ for (\ref{Eq:MixedNormEstimate_0}) to hold, and so extending the scope of all the applications that we will soon state. However, we stick to Theorem~\ref{Thm:MainTheorem_0} for the following reasons. Firstly, $p=2,3$ correspond to the indices for Wolff's and Bourgain's circular maximal functions respectively, which naturally makes our estimates X-ray--type extensions (see Item (c) in Section~\ref{subsec:add_background}). Secondly, $\alpha\in (0,1],(1,2]$ are necessary for our methods to work. Thirdly, if we fix $(p,\alpha)$ and carry out our methods for general pairs of $(q,s)$, then the endpoint (best) indices we obtain are exactly those in Theorem~\ref{Thm:MainTheorem_0}. In this sense, ``$p=q$'' is really a lucky coincidence that naturally appears. Finally, the entire ranges generated by interpolation are very complicated (boring to write out) and do not provide any valuable insights. In contrast, the current elegant ranges are already enough to demonstrate the efficiency of our methods, and to yield interesting applications.

For $\delta$-discretization, we need to first reduce Theorem~\ref{Thm:MainTheorem_0} to the following Theorem~\ref{Thm:MainTheorem}. Such a reduction can be done by interpolating Theorem~\ref{Thm:MainTheorem} with the ``high frequency decay'' estimates (Section~\ref{sec:high_freq_decay}).

\begin{theorem}\label{Thm:MainTheorem} Let $n=2$.
For all $\eps>0$, there is a constant $C_\eps>0$ such that 
    \begin{equation}\label{Eq:MixedNormEstimate}
        \norm{\AA_\delta f}_{L
        _x^q(B,\nu)(L^s_r(I))}
        \leq C_\eps\delta^{-\eps}\langle\nu\rangle_\alpha^{\frac{1}{q}}\norm{f}_{L^p(\R^n)},
    \end{equation}
    holds for all $0<\delta<1,$ when 
    \begin{enumerate}[label={\rm (\arabic*)}]
        \item \label{Item1OfMainTheorem}$p=q=3,\alpha\in (0,1]$, $1\leq s\leq \frac{3}{1-\alpha};$
        \item\label{Item2OfMainTHeorem} $p=q=2,\alpha\in (1,2]$, $1\leq s\leq \frac{2}{2-\alpha}.$
    \end{enumerate}
    Moreover, the result in Case {\rm(1)} is essentially sharp given $p=3$.
\end{theorem}

\begin{remark}
    It is easy to see that when $1/10 \leq \delta<1$, the inequality (\ref{Eq:MixedNormEstimate}) trivially holds true for any $p,q,s \in [1,\infty]$. So we may without loss of generality always assume that $0<\delta<1/10$. Besides, we may always only consider those small $\eps$, e.g., $\eps<1/10$.
\end{remark}
\begin{fact}\label{Fact:Holder}
    By H\"older's inequality in $L^s(I)$, we see that if (\ref{Eq:MixedNormEstimate_0}) holds for some $s$, then it also holds for smaller $s$. Also, by H\"older's inequality in $L^q(\nu)$ and $\norm{\nu\chi_B} \leq \langle\nu\rangle_\alpha$, we see that if (\ref{Eq:MixedNormEstimate_0}) holds for some $q$, then it also holds for smaller $q$. The same conclusions also hold for (\ref{Eq:MixedNormEstimate}).
\end{fact}

Necessary conditions for both Theorem~\ref{Thm:MainTheorem_0} and \ref{Thm:MainTheorem} will be obtained in Section~\ref{sec:necessary}. Here let us only summarize them as follows:
\begin{proposition}\label{Prop:necessary}
    For \eqref{Eq:MixedNormEstimate_0} and \eqref{Eq:MixedNormEstimate} to hold when $n=2$, we must have
    \begin{align}\label{Eq:necessarys}
        \frac{1}{s} + \frac{\alpha}{q} \geq \frac{1}{p}, \quad 1 + \frac{2}{s} + \frac{\min\{0, \alpha-1\}}{q} \geq \frac{3}{p}, \quad 1 + \frac{1}{s} \geq \frac{2}{p}, \quad 1 + \frac{2}{s} \geq \frac{4-\alpha}{p}, \quad
        \frac{1}{s}\geq \frac{1-\alpha}{p},
    \end{align}
    where the next‑to‑last condition is for $\alpha\in [1,2]$, and the last condition is for $\alpha\in (0,1]$.
\end{proposition}
By taking $p=q=3$ in the first and last conditions above, we immediately see the sharpness in $s$ of Case (1). In fact, the sharpness holds in the stronger sense that if we fix $p=3$, then the most critical endpoint is given by $q=3$ and $s=\frac{3}{1-\alpha}$ (see Figure~(\subref{fig:first}) and (\subref{fig:last}) in Section~\ref{subsec:compare}). However, if we fix $p=2$, then things are more complicated, and we will discuss this systematically in the upcoming Section~\ref{subsec:compare}. In general, it is natural to conjecture that the ranges in (\ref{Eq:necessarys}) are also sufficient, possibly up to boundaries and endpoints.

As a concluding remark, although our main results only deal with the case $n=2$, we will extend the argument to $n \geq 3$ whenever it remains valid and does not incur additional difficulties (Section~\ref{sec:discretization} and ~\ref{subsec:Slicing}). This is because as motivated in Section~\ref{subsec:History}, weighted mixed-norm estimates for spherical averages $\AA^n$ ($n\geq 3$) will also be of interest, and we hope that this will to some extent lay the foundation for future studies.

\subsection{Comparison with previous results}\label{subsec:compare}
Now let us compare our Theorem~\ref{Thm:MainTheorem_0} with existing results in the literature by fixing $p=2$ or $3$ and drawing the range of $(\frac{1}{q}, \frac{1}{s})$. For this purpose, we need to first apply vector-valued interpolation to reduce those results exactly to our setting. One may consult arguments in Section~\ref{sec:high_freq_decay} for a sense of how to make such interpolation rigorous. In what follows, we only present the conclusions.

First consider the case when $p=3$ and $\alpha\in (0,1]$. Let $P \coloneqq (\frac{1}{3}, \frac{1-\alpha}{3})$. Applying \cite[Theorem~1.3]{HKL22} to $\nu\otimes\mathcal{L}^1|_I$, we know that when $\alpha \in (\frac{1}{5},1]$, (\ref{Eq:MixedNormEstimate_0}) holds for $\frac{1}{q} = \frac{1}{s} > \frac{1}{3\alpha+3}$. Let $H \coloneqq (\frac{1}{3\alpha+3}, \frac{1}{3\alpha+3})$. Also, by interpolating \cite[Theorem~2.3~(a)(b)]{Oberlin15} with the trivial bound for $(\frac{1}{p},\frac{1}{q},\frac{1}{s}) = (0,0,0)$, we know that when $\alpha \in (\frac{1}{2}, 1]$, (\ref{Eq:MixedNormEstimate_0}) holds for $\frac{1}{q}>\frac{1}{3\alpha}$ and $\frac{1}{s} = \frac{1}{6}$; and when $\alpha \in (0,\frac{1}{2})$, (\ref{Eq:MixedNormEstimate_0}) holds for $\frac{1}{q} = \frac{1}{3}$ and $\frac{1}{s} > \frac{1-\alpha}{3}$. Let $O_1 \coloneqq (\frac{1}{3\alpha}, \frac{1}{6})$ and $O_2 \coloneqq (\frac{1}{3}, \frac{1-\alpha}{3})$. By using Fact~\ref{Fact:Holder} and interpolating everything with the trivial bound $(\frac{1}{q}, \frac{1}{s}) = (0, \frac{1}{3})$, we get Figure~(\subref{fig:first}), (\subref{fig:last}). Note that $O_1$ always lies in the orange region when $\alpha \in (\frac{1}{2}, 1]$, and $H$ always lies on the segment connecting $(0,\frac{1}{3})$ and $P$ when $\alpha\in (\frac{1}{5},1]$, so they will never dominate. However, since $P=O_2$, our Theorem~\ref{Thm:MainTheorem_0} recovers the previous best result when $\alpha \in (0,\frac{1}{2})$, and extends the critical vertex $(\frac{1}{3}, \frac{1-\alpha}{3})$ to $\alpha \in [\frac{1}{2}, 1]$.

Next consider the case when $p=2$ and $\alpha \in (1,2]$. Let $P \coloneqq (\frac{1}{2}, \frac{2-\alpha}{2})$. Applying \cite[Theorem~1.3]{HKL22} to $\nu\otimes\mathcal{L}^1|_I$, we know that when $\alpha \in (1,2]$, (\ref{Eq:MixedNormEstimate_0}) holds for $\frac{1}{q} = \frac{1}{s} > \frac{1}{4\alpha-2}$. Let $H \coloneqq (\frac{1}{4\alpha-2}, \frac{1}{4\alpha-2})$. 
Also, by interpolating \cite[Theorem~2.3~(a)]{Oberlin15} with the trivial bound for $(\frac{1}{p},\frac{1}{q},\frac{1}{s}) = (0,0,0)$, we know that when $\alpha \in (1, 2]$, (\ref{Eq:MixedNormEstimate_0}) holds for any $(\frac{1}{q},\frac{1}{s})$ lying strictly above the open segment connecting $(0,\frac{1}{2})$ and $O \coloneqq (\frac{1}{2\alpha}, \frac{1}{4})$. By using Fact~\ref{Fact:Holder} and interpolating everything with the trivial bound $(\frac{1}{q}, \frac{1}{s}) = (0, \frac{1}{2})$, we get Figure~(\subref{fig:left}), (\subref{fig:middle}), (\subref{fig:right}). The relative positions of $P,H,O$ depend on $\alpha$: $O$ dominates when $\alpha \in (1,\frac{4}{3}]$, $H$ and $O$ dominate when $\alpha \in (\frac{4}{3}, \frac{3}{2}]$, while $H$ and $P$ dominate when $\alpha \in (\frac{2}{3}, 2]$. 

\begin{remark}
    In Figure~\ref{Fig:ranges}, the admissible ranges are the colored (orange, pink, blue) regions together with the solid lines and filled dots. The dashed lines and open dots on the boundary cannot be reached. The colors of the subregions are assigned according to the historical order: The orange comes from $O$ ($O_1$, $O_2$), the pink is the expansion made by $H$, and the blue represents the new improvement made by our $P$.
\end{remark}

Finally, by taking $p=3$ in Proposition~\ref{Prop:necessary}, we see that our point $P$ determines the sharp ranges in both Figure~(\subref{fig:first}) and (\subref{fig:last}). Also, by taking $p=2$ therein, we can depict the conjectured ranges for $(\frac{1}{q},\frac{1}{s})$ in Figures~(\subref{fig:left}), (\subref{fig:middle}), (\subref{fig:right}). More precisely, they are the upper-right regions enclosed by the gray polygonal lines. The critical conjectured nodes are $Q_1 \coloneqq (\frac{1}{2}, \frac{2-\alpha}{4})$, and $Q_2 \coloneqq (\frac{1}{2\alpha+2}, \frac{1}{2\alpha+2})$. It is interesting to note that $Q_1$ is the midpoint between $(\frac{1}{2},0)$ and $P$, and $Q_2$ lies on the diagonal (as $H$ does). The former fact provides another strong motivation to focus on $p=q=2$. It is also worth pointing out that if the full conjectured ranges for weighted circular averaging estimates of the form (\ref{Eq:weighted_form}) in \cite{HKL22} were known, then by applying it to $\nu \otimes \mathcal{L}^1|_I$ as before, we would essentially get the diagonal points $(\frac{1}{3\alpha+3}, \frac{1}{3\alpha+3})$ for $\alpha \in (0,1]$ and
$(\frac{1}{4\alpha-2}, \frac{1}{4\alpha-2})$ for $\alpha \in (1,2]$. The former matches $H$, which is already known, while the latter is worse than $Q_2$ whenever $\alpha<2$. This somehow indicates that even in the $q=s$ regime (recall Item (iv) in Section~\ref{minorsec:weighted_mixed_norm}), one still needs to exploit the ``product structure'' to get $Q_2$. We do not pursue this direction, since $Q_1$ is always more useful than it in all the known applications (recall Item (v) in Section~\ref{minorsec:weighted_mixed_norm}).

We summarize what we get in the following theorem:
\begin{theorem}\label{Thm:MainTheorem_00}
    Fix $n=2$. If $p=3$ and $\alpha\in(0,1]$, then {\rm (\ref{Eq:MixedNormEstimate_0})} holds for all $(\frac{1}{q}, \frac{1}{s})$ in the admissible ranges illustrated by Figure~{\rm(\subref{fig:first})} and {\rm(\subref{fig:last})}. If $p=2$ and $\alpha\in(1,2]$, then {\rm (\ref{Eq:MixedNormEstimate_0})} holds for all $(\frac{1}{q}, \frac{1}{s})$ in the admissible ranges illustrated by Figure~{\rm(\subref{fig:left})}, {\rm(\subref{fig:middle})} and {\rm(\subref{fig:right})}.
\end{theorem}

\begin{figure}[htbp]
\centering

\begin{subfigure}{0.45\textwidth}
\centering

\begin{tikzpicture}[scale=5,>=stealth]

\pgfmathsetmacro{\myalpha}{2/5}
\pgfmathsetmacro{\xH}{1/(3*\myalpha+3)}
\pgfmathsetmacro{\yH}{1/(3*\myalpha+3)}
\pgfmathsetmacro{\xP}{1/3}
\pgfmathsetmacro{\yP}{(1-\myalpha)/3}

\draw[->] (0,0) -- (1.1,0) node[below] {$\frac{1}{q}$};
\draw[->] (0,0) -- (0,1.1) node[left] {$\frac{1}{s}$};

\foreach \x in {\xH,\xP,1} {
\draw[line width=1pt] (\x,0.5pt) -- (\x,-0.5pt);
}
\foreach \y in {\yH,\yP} {
\draw[line width=1pt] (0.5pt,\y) -- (-0.5pt,\y);
}

\coordinate (H) at (\xH,\yH);
\coordinate (P) at (\xP,\yP);
\fill[orange!20] (0,1) -- (1,1) -- (1,\yP) -- (P) -- (0,1/3) -- cycle;
\draw[black, thick] (0,1/3) -- (0,1) -- (1,1) -- (1,\yP);

\node[circle, fill=black, inner sep=1pt] at (0,1/3) {};
\node[circle, fill=black, inner sep=1pt] at (0,1) {};
\node[circle, fill=black, inner sep=1pt] at (1,1) {};

\node[circle, draw=purple, thick, inner sep=1pt] at (H) {};
\node[circle, draw=blue, thick, inner sep=1pt] at (P) {};
\node[circle, draw=black, thick, inner sep=1pt] at (1,\yP) {};

\node[above,yshift=-1pt] at (H) {\color{purple} $H$}; 
\node[above right, xshift=-2pt, yshift=-2pt] at (P) {{\color{blue}$P$}$=${\color{red} $O_2$}}; 

\node[below, xshift=0pt] at (\xH,0) {\color{purple} \tiny $\frac{1}{3\alpha+3}$};
\node[below] at (\xP,0) {\color{blue} \tiny $\frac{1}{3}$};

\node[left, yshift=1pt] at (0,\yH) {\color{purple} \tiny $1/(3\alpha+3)$};
\node[left, yshift=-2pt] at (0,\yP) {\color{blue} \tiny $(1-\alpha)/3$};

\node[below left, xshift=2pt, yshift=2pt] at (0,0) {\tiny $0$};
\node[below] at (1,0) {\tiny $1$};
\node[left] at (0,1/3) {\tiny $1/3$};
\node[left] at (0,1) {\tiny $1$};

\draw[dashed] (0,0) -- (1,1);
\draw[dashed] (1,0) -- (1,\yP);
\draw[dashed] (H) -- (\xH,0);
\draw[dashed] (H) -- (0,\yH);
\draw[dashed] (0,1/3) -- (1/3,1/3) -- (P) -- (\xP,0);
\draw[dashed] (P) -- (0,\yP);

\draw[dashed, black, thick] (0,1/3) -- (H) -- (P) -- (1,\yP);

\end{tikzpicture}

\caption{$\alpha \in (0, \frac{1}{2})$}
\label{fig:first}
\end{subfigure}
\hfill
\begin{subfigure}{0.45\textwidth}
\centering

\begin{tikzpicture}[scale=5,>=stealth]

\pgfmathsetmacro{\myalpha}{2/3}
\pgfmathsetmacro{\xO}{1/(3*\myalpha)}
\pgfmathsetmacro{\yO}{1/6}
\pgfmathsetmacro{\xH}{1/(3*\myalpha+3)}
\pgfmathsetmacro{\yH}{1/(3*\myalpha+3)}
\pgfmathsetmacro{\xP}{1/3}
\pgfmathsetmacro{\yP}{(1-\myalpha)/3}

\draw[->] (0,0) -- (1.1,0) node[below] {$\frac{1}{q}$};
\draw[->] (0,0) -- (0,1.1) node[left] {$\frac{1}{s}$};

\foreach \x in {\xO,\xH,\xP,1} {
\draw[line width=1pt] (\x,0.5pt) -- (\x,-0.5pt);
}
\foreach \y in {\yO,\yH,\yP} {
\draw[line width=1pt] (0.5pt,\y) -- (-0.5pt,\y);
}

\coordinate (O) at (\xO,\yO);
\coordinate (H) at (\xH,\yH);
\coordinate (P) at (\xP,\yP);
\fill[orange!20] (O) -- (0,1/3) -- (0,1) -- (1,1) -- (1,\yO) -- cycle;
\fill[pink!80] (H) -- (0,1/3) -- (O) -- cycle;
\fill[blue!20] (P) -- (H) -- (O) -- (1,\yO) -- (1,\yP) -- cycle;
\draw[black, thick] (0,1/3) -- (0,1) -- (1,1) -- (1,\yP);

\node[circle, fill=black, inner sep=1pt] at (0,1/3) {};
\node[circle, fill=black, inner sep=1pt] at (0,1) {};
\node[circle, fill=black, inner sep=1pt] at (1,1) {};

\node[circle, draw=red, thick, inner sep=1pt] at (O) {};
\node[circle, draw=purple, thick, inner sep=1pt] at (H) {};
\node[circle, draw=blue, thick, inner sep=1pt] at (P) {};
\node[circle, draw=black, thick, inner sep=1pt] at (1,\yP) {};

\node[above right, xshift=-3pt, yshift=-3pt] at (O) {\color{red} $O_1$}; 
\node[above,yshift=-1pt] at (H) {\color{purple} $H$}; 
\node[above right, xshift=-2pt, yshift=-2pt] at (P) {\color{blue} $P$}; 

\node[below, xshift=0pt] at (\xO,0) {\color{red} \tiny $\frac{1}{3\alpha}$};
\node[below, xshift=0pt] at (\xH,0) {\color{purple} \tiny $\frac{1}{3\alpha+3}$};
\node[below] at (\xP,0) {\color{blue} \tiny $\frac{1}{3}$};

\node[left, yshift=-1pt] at (0,\yO) {\color{red} \tiny $1/6$};
\node[left, yshift=1pt] at (0,\yH) {\color{purple} \tiny $1/(3\alpha+3)$};
\node[left] at (0,\yP) {\color{blue} \tiny $(1-\alpha)/3$};

\node[below left, xshift=2pt, yshift=2pt] at (0,0) {\tiny $0$};
\node[below] at (1,0) {\tiny $1$};
\node[left] at (0,1/3) {\tiny $1/3$};
\node[left] at (0,1) {\tiny $1$};

\draw[dashed] (0,0) -- (1,1);
\draw[dashed] (1,0) -- (1,\yP);
\draw[dashed] (O) -- (\xO,0);
\draw[dashed] (O) -- (0,\yO);
\draw[dashed] (H) -- (\xH,0);
\draw[dashed] (H) -- (0,\yH);
\draw[dashed] (0,1/3) -- (1/3,1/3) -- (P) -- (\xP,0);
\draw[dashed] (P) -- (0,\yP);

\draw[dashed, black, thick] (0,1/3) -- (H) -- (P) -- (1,\yP);

\end{tikzpicture}

\caption{$\alpha \in [\frac{1}{2}, 1]$}
\label{fig:last}
\end{subfigure}

\begin{subfigure}{0.45\textwidth}
\centering

\begin{tikzpicture}[scale=5,>=stealth]

\pgfmathsetmacro{\myalpha}{1.2}
\pgfmathsetmacro{\xO}{1/(2*\myalpha)}
\pgfmathsetmacro{\yO}{1/4}
\pgfmathsetmacro{\xH}{1/(4*\myalpha-2)}
\pgfmathsetmacro{\yH}{1/(4*\myalpha-2)}
\pgfmathsetmacro{\xP}{1/2}
\pgfmathsetmacro{\yP}{(2-\myalpha)/2}
\pgfmathsetmacro{\xQi}{\xP}
\pgfmathsetmacro{\yQi}{\yP/2}
\pgfmathsetmacro{\xQii}{1/(2*\myalpha+2)}
\pgfmathsetmacro{\yQii}{\xQii}

\draw[->] (0,0) -- (1.1,0) node[below] {$\frac{1}{q}$};
\draw[->] (0,0) -- (0,1.1) node[left] {$\frac{1}{s}$};

\foreach \x in {\xO,\xH,\xP,1} {
\draw[line width=1pt] (\x,0.5pt) -- (\x,-0.5pt);
}
\foreach \y in {\yO,\yH,\yP} {
\draw[line width=1pt] (0.5pt,\y) -- (-0.5pt,\y);
}

\coordinate (O) at (\xO,\yO);
\coordinate (H) at (\xH,\yH);
\coordinate (P) at (\xP,\yP);
\coordinate (Qi) at (\xQi,\yQi);
\coordinate (Qii) at (\xQii,\yQii);
\fill[orange!20] (0,1) -- (1,1) -- (1,\yO) -- (O) -- (0,1/2) -- cycle;
\draw[black, thick] (0,1/2) -- (0,1) -- (1,1) -- (1,\yO);

\node[circle, fill=black, inner sep=1pt] at (0,1/2) {};
\node[circle, fill=black, inner sep=1pt] at (0,1) {};
\node[circle, fill=black, inner sep=1pt] at (1,1) {};

\node[circle, draw=red, thick, inner sep=1pt] at (O) {};
\node[circle, draw=purple, thick, inner sep=1pt] at (H) {};
\node[circle, draw=blue, thick, inner sep=1pt] at (P) {};
\node[circle, draw=black, thick, inner sep=1pt] at (1,\yO) {};

\node[circle, draw=gray, thick, inner sep=1pt] at (Qi) {};
\node[circle, draw=gray, thick, inner sep=1pt] at (Qii) {};

\node[above, xshift=2pt] at (O) {\color{red} $O$}; 
\node[above,yshift=-1pt] at (H) {\color{purple} $H$}; 
\node[right, xshift=-2pt] at (P) {\color{blue} $P$}; 
\node[below right, xshift=-3pt, yshift=2pt] at (Qi) {\color{gray} $Q_1$}; 
\node[below, xshift=-2pt, yshift=1pt] at (Qii) {\color{gray} $Q_2$}; 

\node[below] at (\xO,0) {\color{red} \tiny $\frac{1}{2\alpha}$};
\node[below, xshift=-8pt] at (\xH,0) {\color{purple} \tiny $\frac{1}{4\alpha-2}$};
\node[below] at (\xP,0) {\color{blue} \tiny $\frac{1}{2}$};

\node[left] at (0,\yO) {\color{red} \tiny $1/4$};
\node[left,yshift=-2pt] at (0,\yH) {\color{purple} \tiny $1/(4\alpha-2)$};
\node[left, yshift=2pt] at (0,\yP) {\color{blue} \tiny $(2-\alpha)/2$};

\node[below left, xshift=2pt, yshift=2pt] at (0,0) {\tiny $0$};
\node[below] at (1,0) {\tiny $1$};
\node[left] at (0,1/2) {\tiny $1/2$};
\node[left] at (0,1) {\tiny $1$};

\draw[dashed] (1,\yO) -- (1,0);
\draw[dashed] (0,0) -- (1,1);
\draw[dashed] (O) -- (\xO,0);
\draw[dashed] (O) -- (0,\yO);
\draw[dashed] (H) -- (\xH,0);
\draw[dashed] (H) -- (0,\yH);
\draw[dashed] (0,1/2) -- (1/2,1/2) -- (P) -- (\xP,0);
\draw[dashed] (P) -- (0,\yP);

\draw[dashed, black, thick] (0,1/2) -- (O) -- (1,\yO);

\pgfmathsetmacro{\kI}{-(\myalpha-1)/2}
\pgfmathsetmacro{\b}{1/(2*\myalpha+2)}

\draw[gray, thick, domain=\b:1/2] plot (\x, {\kI*\x+1/4});

\draw[gray, dashed, domain=0:\b] plot (\x, {\kI*\x+1/4});

\draw[gray, thick, domain=0:\b] plot (\x, {-\myalpha*\x+1/2});

\draw[gray, dashed, domain=\b:\xO] plot (\x, {-\myalpha*\x+1/2});

\draw[gray, thick, domain=1/2:1] plot (\x, \yP/2);

\end{tikzpicture}

\caption{$\alpha \in (1, \frac{4}{3}]$}
\label{fig:left}
\end{subfigure}
\hfill
\begin{subfigure}{0.45\textwidth}
\centering

\begin{tikzpicture}[scale=5,>=stealth]

\pgfmathsetmacro{\myalpha}{1.4}
\pgfmathsetmacro{\xO}{1/(2*\myalpha)}
\pgfmathsetmacro{\yO}{1/4}
\pgfmathsetmacro{\xH}{1/(4*\myalpha-2)}
\pgfmathsetmacro{\yH}{1/(4*\myalpha-2)}
\pgfmathsetmacro{\xP}{1/2}
\pgfmathsetmacro{\yP}{(2-\myalpha)/2}
\pgfmathsetmacro{\xQi}{\xP}
\pgfmathsetmacro{\yQi}{\yP/2}
\pgfmathsetmacro{\xQii}{1/(2*\myalpha+2)}
\pgfmathsetmacro{\yQii}{\xQii}

\draw[->] (0,0) -- (1.1,0) node[below] {$\frac{1}{q}$};
\draw[->] (0,0) -- (0,1.1) node[left] {$\frac{1}{s}$};

\foreach \x in {\xO,\xH,\xP,1} {
\draw[line width=1pt] (\x,0.5pt) -- (\x,-0.5pt);
}
\foreach \y in {\yO,\yH,\yP} {
\draw[line width=1pt] (0.5pt,\y) -- (-0.5pt,\y);
}

\coordinate (O) at (\xO,\yO);
\coordinate (H) at (\xH,\yH);
\coordinate (P) at (\xP,\yP);
\coordinate (Qi) at (\xQi,\yQi);
\coordinate (Qii) at (\xQii,\yQii);
\fill[orange!20] (0,1) -- (1,1) -- (1,\yO) -- (O) -- (0,1/2) -- cycle;
\fill[pink!80] (O) -- (H) -- (0,1/2) -- cycle;
\draw[black, thick] (0,1/2) -- (0,1) -- (1,1) -- (1,\yO);
\draw[red!60] (0,1/2) -- (O);

\node[circle, fill=black, inner sep=1pt] at (0,1/2) {};
\node[circle, fill=black, inner sep=1pt] at (0,1) {};
\node[circle, fill=black, inner sep=1pt] at (1,1) {};

\node[circle, draw=red, thick, inner sep=1pt] at (O) {};
\node[circle, draw=purple, thick, inner sep=1pt] at (H) {};
\node[circle, draw=blue, thick, inner sep=1pt] at (P) {};
\node[circle, draw=black, thick, inner sep=1pt] at (1,\yO) {};

\node[circle, draw=gray, thick, inner sep=1pt] at (Qi) {};
\node[circle, draw=gray, thick, inner sep=1pt] at (Qii) {};

\node[above, xshift=2pt] at (O) {\color{red} $O$}; 
\node[above,yshift=-1pt] at (H) {\color{purple} $H$}; 
\node[right, xshift=-2pt] at (P) {\color{blue} $P$}; 
\node[below right, xshift=-3pt, yshift=2pt] at (Qi) {\color{gray} $Q_1$}; 
\node[below, xshift=-2pt, yshift=1pt] at (Qii) {\color{gray} $Q_2$}; 

\node[below] at (\xO,0) {\color{red} \tiny $\frac{1}{2\alpha}$};
\node[below, xshift=-5pt] at (\xH,0) {\color{purple} \tiny $\frac{1}{4\alpha-2}$};
\node[below] at (\xP,0) {\color{blue} \tiny $\frac{1}{2}$};

\node[left, yshift=-5pt] at (0,\yO) {\color{red} \tiny $1/4$};
\node[left,yshift=-1pt] at (0,\yH) {\color{purple} \tiny $1/(4\alpha-2)$};
\node[left, yshift=3pt] at (0,\yP) {\color{blue} \tiny $(2-\alpha)/2$};

\node[below left, xshift=2pt, yshift=2pt] at (0,0) {\tiny $0$};
\node[below] at (1,0) {\tiny $1$};
\node[left] at (0,1/2) {\tiny $1/2$};
\node[left] at (0,1) {\tiny $1$};

\draw[dashed] (1,\yO) -- (1,0);
\draw[dashed] (0,0) -- (1,1);

\draw[dashed] (O) -- (\xO,0);
\draw[dashed] (O) -- (0,\yO);
\draw[dashed] (H) -- (\xH,0);
\draw[dashed] (H) -- (0,\yH);
\draw[dashed] (0,1/2) -- (1/2,1/2) -- (P) -- (\xP,0);
\draw[dashed] (P) -- (0,\yP);

\draw[dashed, black, thick] (0,1/2) -- (H) --(O) -- (1,\yO);

\pgfmathsetmacro{\kI}{-(\myalpha-1)/2}
\pgfmathsetmacro{\b}{1/(2*\myalpha+2)}

\draw[gray, thick, domain=\b:1/2] plot (\x, {\kI*\x+1/4});

\draw[gray, dashed, domain=0:\b] plot (\x, {\kI*\x+1/4});

\draw[gray, thick, domain=0:\b] plot (\x, {-\myalpha*\x+1/2});

\draw[gray, dashed, domain=\b:\xO] plot (\x, {-\myalpha*\x+1/2});

\draw[gray, thick, domain=1/2:1] plot (\x, \yP/2);

\end{tikzpicture}

\caption{$\alpha \in (\frac{4}{3},\frac{3}{2}]$}
\label{fig:middle}
\end{subfigure}

\begin{subfigure}{0.45\textwidth}
\centering

\begin{tikzpicture}[scale=5,>=stealth]

\pgfmathsetmacro{\myalpha}{1.7}
\pgfmathsetmacro{\xO}{1/(2*\myalpha)}
\pgfmathsetmacro{\yO}{1/4}
\pgfmathsetmacro{\xH}{1/(4*\myalpha-2)}
\pgfmathsetmacro{\yH}{1/(4*\myalpha-2)}
\pgfmathsetmacro{\xP}{1/2}
\pgfmathsetmacro{\yP}{(2-\myalpha)/2}
\pgfmathsetmacro{\xQi}{\xP}
\pgfmathsetmacro{\yQi}{\yP/2}
\pgfmathsetmacro{\xQii}{1/(2*\myalpha+2)}
\pgfmathsetmacro{\yQii}{\xQii}

\draw[->] (0,0) -- (1.1,0) node[below] {$\frac{1}{q}$};
\draw[->] (0,0) -- (0,1.1) node[left] {$\frac{1}{s}$};

\foreach \x in {\xO,\xH,\xP,1} {
\draw[line width=1pt] (\x,0.5pt) -- (\x,-0.5pt);
}
\foreach \y in {\yO,\yH,\yP} {
\draw[line width=1pt] (0.5pt,\y) -- (-0.5pt,\y);
}

\coordinate (O) at (\xO,\yO);
\coordinate (H) at (\xH,\yH);
\coordinate (P) at (\xP,\yP);
\coordinate (Qi) at (\xQi,\yQi);
\coordinate (Qii) at (\xQii,\yQii);
\fill[orange!20] (0,1) -- (1,1) -- (1,\yO) -- (O) -- (0,1/2) -- cycle;
\fill[pink!80] (H) -- (0,1/2) -- (O) -- (1,\yO) -- (1,\yH) -- cycle;
\fill[blue!20] (P) -- (H) -- (1,\yH) -- (1,\yP) -- cycle;
\draw[black, thick] (0,1/2) -- (0,1) -- (1,1) -- (1,\yP);

\node[circle, fill=black, inner sep=1pt] at (0,1/2) {};
\node[circle, fill=black, inner sep=1pt] at (0,1) {};
\node[circle, fill=black, inner sep=1pt] at (1,1) {};

\node[circle, draw=red, thick, inner sep=1pt] at (O) {};
\node[circle, draw=purple, thick, inner sep=1pt] at (H) {};
\node[circle, draw=blue, thick, inner sep=1pt] at (P) {};
\node[circle, draw=black, thick, inner sep=1pt] at (1,\yP) {};

\node[circle, draw=gray, thick, inner sep=1pt] at (Qi) {};
\node[circle, draw=gray, thick, inner sep=1pt] at (Qii) {};

\node[above, xshift=2pt] at (O) {\color{red} $O$}; 
\node[above,yshift=-1pt] at (H) {\color{purple} $H$}; 
\node[above right, xshift=-2pt, yshift=-2pt] at (P) {\color{blue} $P$}; 
\node[below right, xshift=-3pt, yshift=2pt] at (Qi) {\color{gray} $Q_1$}; 
\node[below, xshift=-2pt, yshift=1pt] at (Qii) {\color{gray} $Q_2$}; 

\node[below] at (\xO,0) {\color{red} \tiny $\frac{1}{2\alpha}$};
\node[below, xshift=-4pt] at (\xH,0) {\color{purple} \tiny $\frac{1}{4\alpha-2}$};
\node[below] at (\xP,0) {\color{blue} \tiny $\frac{1}{2}$};

\node[left, yshift=2pt] at (0,\yO) {\color{red} \tiny $1/4$};
\node[left,yshift=0pt] at (0,\yH) {\color{purple} \tiny $1/(4\alpha-2)$};
\node[left, yshift=0pt] at (0,\yP) {\color{blue} \tiny $(2-\alpha)/2$};

\node[below left, xshift=2pt, yshift=2pt] at (0,0) {\tiny $0$};
\node[below] at (1,0) {\tiny $1$};
\node[left] at (0,1/2) {\tiny $1/2$};
\node[left] at (0,1) {\tiny $1$};

\draw[dashed] (1,\yO) -- (1,0);
\draw[dashed] (0,0) -- (1,1);

\draw[dashed] (O) -- (\xO,0);
\draw[dashed] (O) -- (0,\yO);
\draw[dashed] (H) -- (\xH,0);
\draw[dashed] (H) -- (0,\yH);
\draw[dashed] (0,1/2) -- (1/2,1/2) -- (P) -- (\xP,0);
\draw[dashed] (P) -- (0,\yP);

\draw[dashed, black, thick] (0, 1/2) -- (H) -- (P) -- (1,\yP);

\pgfmathsetmacro{\kI}{-(\myalpha-1)/2}
\pgfmathsetmacro{\b}{1/(2*\myalpha+2)}
\pgfmathsetmacro{\bI}{1/(2*\myalpha-2)}

\draw[gray, thick, domain=\b:1/2] plot (\x, {\kI*\x+1/4});

\draw[gray, dashed, domain=0:\b] plot (\x, {\kI*\x+1/4});

\draw[gray, thick, domain=0:\b] plot (\x, {-\myalpha*\x+1/2});

\draw[gray, dashed, domain=\b:\xO] plot (\x, {-\myalpha*\x+1/2});

\draw[gray, thick, domain=1/2:1] plot (\x, \yP/2);

\end{tikzpicture}

\caption{$\alpha \in (\frac{3}{2}, 2]$}
\label{fig:right}
\end{subfigure}

\caption{(\subref{fig:first}), (\subref{fig:last}) are for $p=3$, while (\subref{fig:left}), (\subref{fig:middle}), (\subref{fig:right}) are for $p=2$.}
\label{Fig:ranges}
\end{figure}

\subsection{Applications}\label{subsec:apply}
We may look at Theorem~\ref{Thm:MainTheorem_0} from several other perspectives. As we will see in Section~\ref{sec:discretization}, it corresponds to the $\delta$-discretized version of the planar incidence problem: For circles taking centers in a fractal set with multiple
radii at each center, how much do they overlap? This is the typical geometric picture for X-ray--type estimates and requires a more delicate understanding of the incidence geometry of circles (see Item (c) in Section~\ref{subsec:add_background}). Theorem~\ref{Thm:MainTheorem_0} can be regarded as a way of quantifying the extent to which the circles overlap, i.e., the non-concentration property of the circles. 

In integral geometry, if $f\in L^p(\R^2)$, then clearly for every $x\in\R^2$, we have $\AA f(x,\cdot) \in L_r^p[1,2]$ by Fubini's theorem. Due to the intrinsic curvature properties of $\AA$, one might expect improved generic integrability (i.e., $\AA f(x,\cdot) \in L_r^s[1,2]$ for some $s>p$) outside of an exceptional set of $x$. A natural question is to bound the Hausdorff dimension of $D(f,s) \coloneqq \{x\in\R^2: \AA f(x,\cdot) \not\in L_r^s[1,2] \}$. This is what we anticipated in Item (vii) of Section~\ref{minorsec:weighted_mixed_norm}. Our Theorem~\ref{Thm:MainTheorem_0} has the following immediate corollary:
\begin{corollary}\label{Cor:integral_geometry}
    If $f \in L_{\rm loc}^{3+}(\R^2)$, then $\dim_H D(f,s) \leq 1-\frac{3}{s}$ for any $s\in (3,\infty]$. If $f \in L_{\rm loc}^{2+}(\R^2)$, then $\dim_H D(f,s) \leq 2-\frac{2}{s}$ for any $s\in (4, \infty]$.  
\end{corollary}
\begin{remark}
    If the conjectured ranges in Proposition~\ref{Prop:necessary} were true for $p=2$, then we would have $\dim_H D(f,s) \leq 2-\frac{4}{s}$ for $f \in L_{\rm loc}^{2+}(\R^2)$ and any $s\in(4,\infty]$.
\end{remark}
Throughout the paper, by ``$g \in L_{\rm loc}^{p_0+}(\R^n)$'', we mean that for any compact domain $\Omega \subseteq \R^n$, there exists some $p>p_0$ depending only on $\Omega$ such that $g \in L^p(\Omega)$. We need slightly better local integrability because the ranges of $s$ in Theorem~\ref{Thm:MainTheorem_0} has open right endpoint. The proof of Corollary~\ref{Cor:integral_geometry} relies on real analysis arguments very similar to those in Section~\ref{subsec:Proof_wave_eqn} for (the upcoming) Theorem~\ref{Thm:wave_eqn}, and so is omitted. 

For a more profound application of Theorem~\ref{Thm:MainTheorem_0}, we consider the Cauchy problem for the linear wave equation with zero initial displacement on $\R^2$:
\[
\begin{cases}
u_{tt}-\Delta_x u=0, \qquad x\in \mathbb{R}^2,\ t\in (0,\infty),\\
u(x,0)=0,\\
u_t(x,0)=h(x).
\end{cases}
\]
By viewing the solution $u(x,t)$ as certain superposition of circles, one might expect Theorem~\ref{Thm:MainTheorem_0}/Theorem~\ref{Thm:MainTheorem_00} to reveal properties of the wave equation. And this is indeed the case: we can bound the Hausdorff dimension of $E(u,\beta) \coloneqq \{x\in\mathbb{R}^2: u(x,\cdot) \not\in C_{\rm loc}^\beta[0,\infty)\}$ (the set of $x$ with poor regularity in $t$). This is what we anticipated in Item (viii) of Section~\ref{minorsec:weighted_mixed_norm}. 
\begin{theorem}\label{Thm:wave_eqn}
If $h \in L_{\rm loc}^{3+}(\R^2)$, then $\dim_H E(u,\beta) \leq 3\beta-\frac{1}{2}$, $\forall\,\beta\in(\frac{1}{6},\frac{1}{2}]$. If $h \in L_{\rm loc}^{2+}(\R^2)$, then $\dim_H E(u,\beta) \leq 2\beta + 1$ for any $\beta\in(\frac{1}{4},\frac{1}{2}]$. 
\end{theorem}
\begin{remark}
    If the conjectured ranges in Proposition~\ref{Prop:necessary} were true for $p=2$, then we would have $\dim_H E(u,\beta) \leq 4\beta$ for $h \in L_{\rm loc}^{2+}(\R^2)$ and any $\beta\in(\frac{1}{4},\frac{1}{2}]$.
\end{remark}

We do not expect Theorem~\ref{Thm:wave_eqn} to be sharp in general, but it appears to be the first exceptional set estimates for time regularity in the literature, although the weighted mixed-norm estimates have been studied in \cite{HKL22, Oberlin15}. In Section~\ref{sec:wave_eqn}, we will develop a general transition scheme to go from weighted mixed-norm estimates to bounds for $\dim_H E(u,\beta)$. See Section~\ref{subsec:further_direction} for more historical remarks.

Again, we emphasize that for both Corollary~\ref{Cor:integral_geometry} and Theorem~\ref{Thm:wave_eqn}, the exceptional set estimates usually do not depend on $q$, but we do want $1/s$ to be as small as possible. Since the $1/s$ in our Theorem~\ref{Thm:MainTheorem_0} is worse than that given by \cite{Oberlin15} (only) when $\alpha\in (1,\frac{3}{2})$ (recall Figure~(\subref{fig:left}) and (\subref{fig:middle})), we have excluded the exceptional set estimates associated to this range. In other words, the current bounds stated in Corollary~\ref{Cor:integral_geometry} and Theorem~\ref{Thm:wave_eqn} are the best known so far. Intuitively, our improvements (blue regions in Figure~(\subref{fig:last})) and (\subref{fig:right}) gives new bounds for $\dim_H D(f,s)$ and $\dim_H E(u,\beta)$ when $s$ and $\beta$ are large. For small $s$ and $\beta$, the results in \cite{Oberlin15} should yield better bounds, but for simplicity, we do not try to make the statements comprehensive, and only focus on those ranges most relevant to our contributions.

\begin{remark}
    A natural question is whether our Theorem~\ref{Thm:MainTheorem_0} will lead to new exceptional set estimates for pinned distance sets, if one plugs it into the framework in \cite{Oberlin15} (recall Item (vi) in Section~\ref{minorsec:weighted_mixed_norm}). Unfortunately, this is not the case. One reason might be that the weighted mixed-norm estimates in \cite{Oberlin15} are applied when $p,q,s$ are close to $1$, while our Theorem~\ref{Thm:MainTheorem_0} focuses on large $p,q,s$. On the other hand, we can indeed get new results for the wave equations. This indicates that weighted mixed-norm estimates proved by different methods may have different advantageous regimes in applications, and it would be interesting to further explore such a phenomenon.
\end{remark}

\subsection{Additional backgrounds}\label{subsec:add_background}
Here we will provide additional remarks to put the weighted mixed-norm estimates (\ref{Eq:General_MixedNorm}) in a coherent big picture, which should make the story more comprehensive (the reader may just skip this part on the first reading):
\begin{enumerate}[label={\rm (\alph*)}]
    \item \textbf{Sobolev framework}:  Working in the Sobolev regime $L_\gamma^p$ is another way to balance singularity (recall (ii)) and also leads to some applications. For example, D. Oberlin-R. Oberlin \cite{Oberlin15} also proved $L_{\frac{1-\alpha}{2}+}^2$--$L^q(\nu)$ estimates when $q \in [1,2)$ and $\alpha\in (0,\frac{n-1}{2})$, which can be applied to obtain exceptional set estimates for pinned distance sets. It is also worth mentioning that estimates of the form $L_\gamma^2$--$L_{x,r}^2(\R^n \times I, \nu\otimes\mathcal{L}^1|_I)$ ($q=s=2$) are closely related to Liu's identity \cite{Liu19} for studying pinned distance sets, see \cite[Section~1.2]{HKL22diverge}. However, the $L^p$ framework indeed has its advantages in some applications, and is also more intuitive as all norms are now purely geometric (i.e., does not involve the Fourier side).
    
    \item \textbf{The $\mu = \mathcal{L}^n$ case}: Sharp (unweighted) mixed-norm estimates for $\AA^2$ were obtained by Li, Lou, and Yu \cite{LLY25}, who also studied more general classes of planar curves. One may also consult \cite{BORSS22} for some (unweighted) mixed-norm estimates for $\AA^n$ in frequency-localized forms, which can be used to study variational bounds for $\AA^n$ (still of mixed-norm--type, $L_r^s$ replaced with variational norms). Local smoothing estimates and interpolation techniques play a central role in both works. In contrast, our approach is very different since the weight $\nu$ causes additional complications, which requires us to excavate more delicate structure beyond local smoothing and ``trivial'' manipulations like interpolation.

    \item \textbf{X-ray analogues}: In the classical Kakeya literature, averaging over lines are often called \textit{X-ray transforms}. Their (unweighted) mixed-norm estimates (initiated by Wolff \cite{Wolff98}($n=3$) and generalized by {\L}aba and Tao \cite{LT01_Rev}($n\geq 4$)) were developed to (partially) reduce the Kakeya problem to its sticky case \cite{KLT00, LT01_GAFA}. After discretization, things come down to analyzing the setting where there are multiple parallel $\delta$-tubes in each direction. Our (\ref{Eq:General_MixedNorm}), after discretization (Section~\ref{sec:discretization}), can be viewed as an analogue under the ``center $\leftrightarrow$ direction'' and ``radius $\leftrightarrow$ translation'' interpretation. In both settings, we want improved dependence on the number of parallel tubes/concentric annuli. This is why we may call it ``X-ray--type extensions of Bourgain's and Wolff's circular maximal estimates''. 
    
    A key difference is that here such an extension is only reasonable in the presence of a singular weight $\nu$, because otherwise we already have the sharp bounds (in contrast, the Kakeya maximal conjecture in higher dimensions is still wide open). Weighted mixed-norm estimates for other geometric averaging operators are also rarely studied, at least compared with unweighted variants of X-ray transforms. Notably, Oberlin \cite{Oberlin06radon, Oberlin12} studied restricted Radon transforms (averaging over hypersurfaces) of form similar to (\ref{Eq:General_MixedNorm}), and connected them to the Furstenberg set problem and exceptional set estimates for orthogonal projections in $\R^2$, as well as the hyperplane packing problem in higher dimensions.
    
    \item \textbf{Time-weighted case}: One can also reverse the heuristics for analogy in (c), i.e., consider X-ray analogues subject to ``center $\leftrightarrow$ translation'' and ``radius $\leftrightarrow$ direction''. In the unweighted case, the corresponding $L_r^s(I)(L_x^q(\R^2))$ norm resembles (\ref{Eq:Wolff_maximal}) ($q=\infty$) and actually matches the form of classcial Strichartz estimates for various dispersive equations. One can also add fractal weights in $r$ (time variable), see e.g., \cite{LLR25strichartz} for estimates for dispersive operators. The motivation mainly comes from the recent intense research on Lebesgue bounds for spherical maximal functions with \textit{restricted dilation sets}, see the references therein. We only emphasize that all these works rely heavily on techniques on the Fourier side.
\end{enumerate}

\subsection{Notation}\label{subsec:notation}
If $X$ is a finite set, we use $\# X$ to denote its cardinality. $A_1\lesssim A_2$ or $A_1 = O(A_2)$ means that there is a constant $C$ independent of small parameter $\delta$, such that $A_1\leq CA_2$. We abbreviate $A_1\leq C_\eps\delta^{-\eps} A_2$ as $A_1\lessapprox A_2$ when $\delta$ and $\eps$ are clear from the context. Moreover, $A_1 \sim A_2$ means $A_1 \lesssim A_2$ and $A_1\gtrsim A_2$, while $A_1 \approx A_2$ means $A_1 \lessapprox A_2$ and $A_1\gtrapprox A_2$. We use $B^N(x,r)$ to represent a closed ball centered at $x$ with radius $r$ in $\mathbb{R}^N$. We abbreviate $B^{N}(x,r)$ as $B(x,r)$ if $\mathbb{R}^{N}$ is clear from the context. For any $p\in [1,\infty]$, let $p'$ denote the conjugate exponent of $p$, i.e., $1/p+1/p'=1$. Denote the support of a Borel measure $\mu$ by $\mathrm{spt}\,\mu$, and the support of a function $f$ by $\spt\,f$. Let $\norm{\mu}$ denote the total variance of $\mu$. Let $\chi_E$ be the indicator function of a set $E$. When the dimension $N$ is clear, we also use $|\cdot|$ to denote the Lebesgue measure $\mathcal L^N(\cdot)$. When $n=2$, we may use $C(x,r)$ in place of $S(x,r)$ to represent the circle centered at $x$ with radius $r$. For $\beta\in(0,1]$ and $R>0$, define $$\norm{F}_{C^\beta[0,R]} \coloneqq \norm{F}_{L^\infty[0,R]} + \sup_{x,y\in[0,R], x\neq y} \frac{|F(x)-F(y)|}{|x-y|^\beta}.$$
Let $C^\beta[0,R] \coloneqq \{F: \norm{F}_{C^\beta[0,R]}<\infty\}$ and $C_{\rm loc}^\beta[0,\infty) \coloneqq \{F: F \in C^\beta[0,R],\,\forall\,R>0\}$.

\subsection{Structure of the paper and sketch of the proof }\label{subsec:SketchOfTheProof}We now summarize the structure of the paper and give a rough sketch of the proof. \vspace{0.5em}\\
\noindent \textbf{Standard reduction.} We will first prove Theorem~\ref{Thm:MainTheorem} via purely geometric approaches. By a factoring argument, it suffices to consider a di discretized version of \eqref{Eq:MixedNormEstimate}. We first give the notion of the $\alpha$-dimension for a discretized set.
\begin{definition}[Katz-Tao $(\delta,\alpha,C)$-set]
      $Y\subseteq B$ is a  \textup{Katz-Tao $(\delta,\alpha,C)$-set}, if 
\begin{equation*}
    |\big(Y\medcap B(x,r)\big)|_\delta\leq C\Big(\frac{r}{\delta}\Big)^\alpha, \quad \forall\, x\in \R^n, r\in [\delta,1],
\end{equation*}
for some $C<\infty$. Here $|Y|_\delta$ is the cardinality of a maximal $\delta$-separated subset of $Y$. Equivalently, it is the cardinality of a minimal $\delta$-net of $X$, with elements in $X$. 
\end{definition}

\begin{proposition}[Discretization of $\nu$] \label{Prop:Nu=1Reduction}
 Denote as $X_\delta$  the union of the $\delta$-cubes whose centers are points in $X$, where $X$ is a Katz-Tao $(\delta,\alpha,C)$-set in $B\medcap \delta\mathbb Z^n$. If for $s,q,p\in [1,\infty]$, we have 
\begin{align}
    \norm{\AA_\delta f}_{L_x^q(X_\delta)(L^s_r(I))}\lessapprox \delta^{\frac{n-\alpha}{q}}\norm{f}_{L^p(\R^n)},\label{Eq:Nu=1}
\end{align}
then \eqref{Eq:MixedNormEstimate} holds for the same $s,q,p$.
\end{proposition}
\begin{remark}
    Here we assume that $X$ is a subset of the $\delta$-lattice in $B$. This is totally for the purpose of well presentation. In fact, $X$ can be  any $\delta$-separated set in $B$. 
\end{remark}
A duality argument further reduces \eqref{Eq:Nu=1} to analyzing the intersection of spherical annuli.

\begin{proposition}[Duality argument]\label{Prop:Discritization}
    Let $X$ be as in Proposition~\ref{Prop:Nu=1Reduction}. For each $x\in X$, Let $R_x$ be a subset of $\{r\in I\medcap \delta\mathbb Z\}$ with $\# R_x = m$. If for conjugate exponents $s',q',p'\in [1,\infty]$ of $s,q,p$, we have
    \begin{equation}\label{Eq:DiscretizedVersionOfMixedNormEstimate}
        \Big\|\sum_{x_i\in X}\sum_{r_j\in R_{x_i}}\chi_{S_{10\delta}(x_i,r_j)}\Big\|_{L^{p'}(\R^n)}\lessapprox (m\delta)^{\frac{1}{s'}}\delta^{-\frac{\alpha}{q}}(\#X)^{\frac{1}{q'}},
    \end{equation}
    where $\chi_E$ is the characteristic function of set $E$.
    Then \eqref{Eq:Nu=1} holds for the same $s,q,p$.
\end{proposition}

\begin{remark}
    In fact, \eqref{Eq:MixedNormEstimate}, \eqref{Eq:Nu=1} and \eqref{Eq:DiscretizedVersionOfMixedNormEstimate} are equivalent to each other. We omit the inverse implications.
\end{remark}
These reductions are routine and work for all $n\geq 2$. Details are presented in Section \ref{sec:discretization}.

To prove Proposition \ref{Prop:Discritization} for desired $s,q,p$ when $n=2$, we take different strategies depending on different ranges of $\alpha$.
\vspace{0.5em}\\
\noindent \textbf{Case $\alpha\in (0,1]$.} This corresponds to Item \eqref{Item1OfMainTheorem} of Theorem \ref{Thm:MainTheorem}. We modify the argument of invoking incidence estimate of circles in Pramanik, Yang and Zahl \cite{PYZ24}. Informally speaking, the incidence bound states the following fact. Let $\mathcal C$ be a set of circles satisfying certain nonconcentration and nondegenerate conditions, and $$\mathcal R\coloneq \{x\in \R^2: \text{there are $\mu$ circles tangent at }x \text{ with the same tangent line} \}$$ 
be the set of ``$\mu$-clamshells'', then $\#\mathcal R\lesssim (\#\mathcal C/\mu)^{3/2}$. The precise statement of this fact and the proof of Item~(\ref{Item:Item1OfBasicGMTLemma}) are presented in Section \ref{sec:case_(0,1]}.
\vspace{0.5em}\\
\noindent \textbf{Case $\alpha\in (1,2]$.} This corresponds to Item \eqref{Item2OfMainTHeorem} of Theorem \ref{Thm:MainTheorem}. Our argument relies on a fractal Fubini/slicing theorem in \cite[Proposition~6.3~and~6.4]{Mattila15}. Informally speaking, this theorem allows us to extract a product structure from the ``$\alpha$-dimensional'' set $X$: $X=$``$X_1\times X_2$'', where $X_1$ is ``$1$-dimensional'' and $X_2$ is ``$(\alpha-1)$-dimensional''. After a harmless change of variables, we can assume $X_1$ is parallel to the $y$-axis. On each slice $\{y_0\}\times X_2, y_0\in X_1$, the centers have roughly the same $y$-coordinate $y_0$. This, when combined with certain ``small cap'' reduction, allows us to avoid the tangency case and prove Item (\ref{Item2OfMainTHeorem}). Details are contained in Section \ref{sec:case_(1,2]}.\vspace{0.5em}\\
\noindent 
\textbf{High frequency decay.} Having proved Theorem~\ref{Thm:MainTheorem}, we can interpolate it with the “high frequency decay” estimates in \cite{HKL22} to remove the $\delta^{-\eps}$-loss, which immediately yields Theorem~\ref{Thm:MainTheorem_0}. Such an interpolation is standard and widely used in the literature, see Section~\ref{sec:high_freq_decay} for more details. Our key observation is that high frequency decay is always available in the mixed-norm setting.
\vspace{0.5em}\\
\noindent 
\textbf{Applications to wave equations.} In Section~\ref{sec:wave_eqn}, we will prove Theorem~\ref{Thm:wave_eqn}. Things will be divided into four subsections. In Section~\ref{subsec:prelim}, we will first provide the key bridge (Proposition~\ref{Prop:Abel_transform}) relating $L_r^s$ to $C_{\rm loc}^\beta$, and justify the well-posedness of the H\"older regularity problem. In Section~\ref{subsec:extend_0}, we will do the main technical manipulation, i.e., extending the radial bounds to $r=0$ via certain rescaling arguments. In Section~\ref{subsec:Proof_wave_eqn}, we will finish the proof of Theorem~\ref{Thm:wave_eqn} by the tools established. In Section~\ref{subsec:further_direction}, we will offer additional historical remarks, as well as possible further directions.\\
\noindent 
\textbf{Necessary conditions.}
In the last Section \ref{sec:necessary} of the paper, we also present a conjectural sharp range of $(p,q,s)$ for the mixed norm estimate to be true, by testing the classical examples.

\subsection{Acknowledgements}
The authors would like to thank Shaoming Guo, Yumeng Ou, Malabika Pramanik, Pablo Shmerkin, and Joshua Zahl for helpful conversations during this project. 
\section{
Discretization and duality argument}\label{sec:discretization}

This section consists of two standard reductions of the original estimate. Before the proofs, we record the following ``locally constant property'' of $\AA_\delta f$.
\begin{lemma}\label{Lem:LocallyConstantProperty}
    If $|(x,r)-(\bar x,\bar r)|\leq \delta$, then 
    \begin{equation*}
        \AA_\delta f(x,r)\leq \AA_{10\delta}f(\bar x,\bar r).
    \end{equation*}
\end{lemma}
\begin{proof}
    After unwinding the definition, it suffices to prove the inclusion 
    \begin{equation}\label{Eq:InclusionRelation}
        S_\delta(x,r)\subseteq S_{10\delta}(\bar x,\bar r),
    \end{equation}
     Rewrite the condition,  $\bar x=x+\mathrm{Err}(x)$, $\bar r=r+\mathrm{Err}(r)$ for some $|\mathrm{Err}(x)|,|\mathrm{Err}(r)|\leq \delta$. Therefore, for any $y\in S_\delta(x,r)$, there is $\mathrm{Err}:|\mathrm{Err}|\leq \delta$ such that
    \begin{equation*}
      y=r\frac{y-x}{|y-x|}+\mathrm{Err}=(\bar r-\mathrm{Err}(r))\frac{y-\bar x+\mathrm{Err}(x)}{|y-\bar x+\mathrm{Err}(x)|}+\mathrm{Err}.
    \end{equation*}
    By triangle inequality, we obtain that the error term
    \begin{equation*}
        \Big|y-\bar r\frac{y-\bar x}{|y-\bar x|}\Big|\leq 10\delta,
    \end{equation*}
    which implies $y\in S_{10\delta}(\bar x,\bar r)$. That is \eqref{Eq:InclusionRelation}
\end{proof}

\subsection{ Discretization} 
Recall Proposition \ref{Prop:Nu=1Reduction} and we start the proof. 
\begin{proof}[Proof of Proposition \ref{Prop:Nu=1Reduction}] 
For any $X\subset B\medcap \delta\mathbb Z^n$ and its corresponding $X_\delta$, let 
$$\langle X\rangle_\alpha \coloneq \sup_{x\in \R^n,r\in [\delta,1]}\frac{\#B(x,r)\medcap X}{(r/\delta)^\alpha},$$ we introduce the following intermediate estimate:
\begin{equation}
      \norm{\AA_\delta f}_{L
        _x^q(X_\delta)(L^s_r(I))}\lessapprox \delta^{\frac{n-\alpha}{q}}\langle X\rangle_\alpha^{\frac{1}{q}}\norm{f}_{L^p(\R^n)}.\label{Eq:IntermediateEstimate}
\end{equation}
We will prove 
\begin{equation*}
    \eqref{Eq:Nu=1}\Rightarrow\eqref{Eq:IntermediateEstimate}\Rightarrow\eqref{Eq:MixedNormEstimate}.
\end{equation*}
\noindent \textbf{Step 1:$\eqref{Eq:IntermediateEstimate}\Rightarrow\eqref{Eq:MixedNormEstimate}$.} This reduction is used in the study of weighted restriction estimate by Du and Zhang \cite{Du_Zhang_WeightedRestriction}. Let $B\subseteq\medcup Q$, the union of all $\delta$-cubes with centers $x_Q\in B\medcap \delta\mathbb Z$ and $\mathcal Q =\{Q: \nu(Q)\in [\delta^{100n},1]\langle\nu \rangle_\alpha\}$.  Without loss of generality, we assume the left-hand side of \eqref{Eq:MixedNormEstimate} is $\geq \delta^{50n}\langle\nu \rangle_\alpha^{\frac{1}{q}}\norm{f}_{L^p(\R^n)}$. Otherwise \eqref{Eq:MixedNormEstimate} holds trivially. By Lemma \ref{Lem:LocallyConstantProperty}, note that the integral over $\{\medcup Q:Q\notin\mathcal Q\}$ is negligible,
\begin{equation}\label{Eq:IntegrationOverAllQs}
    \begin{aligned}
    \text{LHS of \eqref{Eq:MixedNormEstimate}}= \norm{\AA_\delta f}_{L
        _x^q(B,\nu)(L^s_r(I))}
        \lesssim \Big(\sum_{Q\in \mathcal Q}\norm{\AA_{10\delta} f(x_Q,r)}_{L^s_r(I)}^q\cdot \nu(Q)\Big)^{\frac{1}{q}}.
\end{aligned}
\end{equation}
We dyadically decompose $[\delta^{100n},1]$. 
\begin{equation}\label{Eq:ReplacedByDyadicNumbers}
    \text{RHS of \eqref{Eq:IntegrationOverAllQs}}\sim\Big(\sum_{s\in [\delta^{100n},1]\medcap 2^{-\mathbb N}}\sum_{Q\in \mathcal Q:\nu(Q)\sim s\langle\nu\rangle_{\alpha}}\norm{\AA_{10\delta} f(x_Q,r)}_{L^s_r(I)}^q\cdot s\langle\nu\rangle_{\alpha}\Big)^{\frac{1}{q}}.
\end{equation}
Since $\#[\delta^{100n},1]\medcap 2^{-\mathbb N}=O(\log \delta^{-1})$, by pigeonhole principle, there is $\nu_0$ such that 
\begin{equation}\label{Eq:AfterDP}
\begin{aligned}
     \text{RHS of \eqref{Eq:ReplacedByDyadicNumbers}}&\lessapprox\Big(\sum_{Q\in\mathcal Q;\nu(Q)\sim \nu_0}\norm{\AA_{10\delta} f(x_Q,r)}_{L^s_r(I)}^q\cdot \nu_0\Big)^{\frac{1}{q}}\\
     &=\Big(\nu_0\delta^{-n}\sum_{Q\subseteq X_\delta';\nu(Q)\sim \nu_0}\norm{\AA_{10\delta} f(x_Q,r)}_{L^s_r(I)}^q\cdot \delta^n\Big)^{\frac{1}{q}}
\end{aligned}
\end{equation}
If we define $X_\delta=\medcup_{Q\in \mathcal Q:\nu(Q)\sim \nu_0}Q$ and $X$ the corresponding centers, then
\begin{equation}\label{Eq:X_Delta}
    \text{RHS of \eqref{Eq:AfterDP}}\lesssim\big(\nu_0\delta^{-n}\big)^{\frac{1}{q}}\norm{\AA_{100\delta} f}_{L
        _x^q(X_\delta)(L^s_r(I))}
\end{equation}
Invoking \eqref{Eq:IntermediateEstimate}, we obtain
\begin{equation}
    \begin{aligned}
     \text{LHS of \eqref{Eq:MixedNormEstimate}}\lessapprox\text{RHS of \eqref{Eq:X_Delta}}\lessapprox \big(\nu_0\delta^{-n}\big)^{\frac{1}{q}}\delta^{\frac{n-\alpha}{q}}\langle X\rangle_\alpha^{\frac{1}{q}}\norm{f}_{L^p(\R^n)}\leq \langle\nu\rangle_\alpha^{\frac{1}{q}}\norm{f}_{L^p(\R^n)}=\text{RHS of \eqref{Eq:MixedNormEstimate}}.
    \end{aligned}
\end{equation}
Last inequality relies on the fact that $\nu_0\langle X\rangle_\alpha=\sup_{x,r}\frac{\nu_0\#B(x,r)\medcap X}{(r/\delta)^\alpha}=\sup_{x,r}\frac{\nu(B(x,r))}{(r/\delta)^\alpha}\leq \delta^{\alpha}\langle\nu\rangle_\alpha$.
\vspace{0.5em}\\
\noindent \textbf{Step 2:$\eqref{Eq:Nu=1}\Rightarrow\eqref{Eq:IntermediateEstimate}$.} This is a direct application of a factoring lemma for Katz-Tao set.
\begin{lemma}[Demeter and Wang \cite{Demeter_Wang}, Lemma 2.8]\label{Lem:FactoringLemmaForKatzTaoSet}
    Each $\eta$-uniform Katz-Tao $(\delta,s,K)$-set $S\subseteq B\medcap \delta\mathbb Z^n$ can be partitioned into $K\delta^{o_\eta(1))}$ many sets $S_i$ which are Katz-Tao $(\delta,s,O_\eta(1))$-sets. The quantities $o_\eta(1))$, $O_\eta(1)$ are independent of $K$.
\end{lemma}
Consult Demeter and Wang's paper for the definition of $\eta$-uniform set. For our purpose, we only need the following fact. 
\begin{lemma}[Demeter and Wang \cite{Demeter_Wang}, Lemma 2.7]\label{Lem:Uniformization}
    Let $\eta>0$ and $\delta<\delta(\eta)$. Then each set $P\subseteq B\medcap \delta \mathbb Z^n$has an $\eta$-uniform subset containing at least a $\sim \delta^{\eta}$-fraction of the original set.
\end{lemma}
A dyadic pigeonholing argument is required before applying their results. This is just technical. Similar to the process above, we claim there is $s_0\in [\delta^{100n},1]\cdot\norm{\AA_\delta f}^q_{L
        _x^q(X_\delta)(L^s_r(I))}$ such that 
\begin{equation}\label{Eq:DPForUsingUniform}
    \text{LHS of \eqref{Eq:IntermediateEstimate}}^q\lessapprox \sum_{\substack{Q\subseteq X_\delta;\\\norm{\AA_\delta f}^q_{L
        _x^q(Q)(L^s_r(I))}\sim s_0} } \norm{\AA_\delta f}^q_{L
        _x^q(Q)(L^s_r(I))}\sim s_0\#X',
\end{equation}
if we denote as $X'$ the centers of $Q$'s in the summation in \eqref{Eq:DPForUsingUniform}. Applying Lemma \ref{Lem:Uniformization} to $X'$ with $\eta\ll \epsilon$, we extract an $\eta$-uniform subset $X''\subseteq X'\subseteq X$ with $\#X''\gtrsim \delta^{\eta}\#X'$. Since $X$ is a Katz-Tao $(\delta,\alpha,\langle X\rangle_\alpha)$-set, it is easy to verify so are all of its subsets. This allows us to apply Lemma \ref{Lem:FactoringLemmaForKatzTaoSet} to factor $X''$ as a union of uniform Katz-Tao sets.
\begin{equation*}
    X''=\medcup_{i=1}^{\langle X\rangle_\alpha \delta^{o_\eta(1)}} {}^iX'',
\end{equation*}
where each ${^i}X''$ is a Katz-Tao $(\delta,\alpha,O_\eta(1))$-set. Then we have 
\begin{equation}\label{Eq:EstimateAfterFactoring}
    \text{RHS of \eqref{Eq:DPForUsingUniform}}\approx s_0\#X''\lessapprox \sum_{i=1}^{\langle X\rangle_\alpha \delta^{o_\eta(1)}}s_0\#{}^iX''=\sum_i\sum_{\substack{Q:x_Q\in {}^iX'';\\\norm{\AA_\delta f}^q_{L
        _x^q(Q)(L^s_r(I))}\sim s_0} } \norm{\AA_\delta f}^q_{L
        _x^q(Q)(L^s_r(I))}.
\end{equation}
For each $i$, the inner summation can be bounded using \eqref{Eq:Nu=1} if  $\eta$ is chosen so that $O_\eta(1)\lessapprox 1$.
\begin{equation}\label{Eq:ApplyNu=1}
    \sum_{\substack{Q:x_Q\in {}^iX''; \\\norm{\AA_\delta f}^q_{L
        _x^q(Q)(L^s_r(I))}\sim s_0}} \norm{\AA_\delta f}^q_{L
        _x^q(Q)(L^s_r(I))}= \norm{\AA_\delta f}^q_{L
        _x^q({}^iX_{\delta}'')(L^s_r(I))}\lessapprox \delta^{n-\alpha}\norm{f}^q_{L^p(\R^n)}.
\end{equation}
Plugging \ref{Eq:ApplyNu=1} back to \eqref{Eq:EstimateAfterFactoring}, we obtain 
\begin{equation*}
      \norm{\AA_\delta f}_{L
        _x^q(X_\delta)(L^s_r(I))}^q= \text{LHS of \eqref{Eq:IntermediateEstimate}}^q\lessapprox \sum_{i=1}^{\langle X\rangle_\alpha \delta^{o_\eta(1)}} \delta^{n-\alpha}\norm{f}^q_{L^p(\R^n)}\lessapprox \langle X\rangle_\alpha\delta^{n-\alpha}\norm{f}^q_{L^p(\R^n)}.
\end{equation*}
Taking the $q$-th root concludes the proof of Proposition \ref{Prop:Nu=1Reduction}.
\end{proof}

\subsection{ Duality argument}Now, we prove Proposition \ref{Prop:Discritization}, which is quite standard. 
\begin{proof}[Proof of Proposition \ref{Prop:Discritization}]Let $X\subseteq B\medcap \delta\mathbb Z^n$ be a Katz-Tao $(\delta,\alpha,C)$-set. By Lemma \ref{Lem:LocallyConstantProperty},
    \begin{equation}\label{Eq:DiscritizationOfTheEstimate}
          \norm{\AA_\delta f}_{L
        _x^q(X_\delta)(L^s_r(I))}\lesssim \Bigg(\sum_{x_i\in X}\delta^n\Big(\sum_{r_j\in R_{x_i}}\delta \cdot A_{10\delta
        }f(x_i,r_j)^s\Big)^{\frac{q}{s}} \Bigg)^{\frac{1}{q}}.
    \end{equation}
    After two rounds of duality argument, we have 
    \begin{equation}\label{Eq:Duality}
        \text{RHS of \eqref{Eq:DiscritizationOfTheEstimate}}=\delta^{\frac{n}{q}+\frac{1}{s}}\sup_{\{a_i\}_{i=1}^{\#X}:\sum_i a_i^{q'}=1}\sup_{\{b_{ij}\}_{j=1}^{\delta^{-1}}:\sum_j  b_{ij}^{s'}=1}\sum_{x_i}\sum_{r_j}a_ib_{ij}\AA_{10\delta}f(x_i,r_j).
    \end{equation}
    Unwinding the definition, separating $f$ and applying H\"older's inequality, we obtain
    \begin{equation*}
    \begin{aligned}
         \text{RHS of \eqref{Eq:Duality}}&=\delta^{\frac{n}{q}-\frac{1}{s'}}\sup_{\{a_i\}}\sup_{\{b_{ij}\}}\int_{\R^n}\Big(\sum_{x_i}\sum_{r_j}a_ib_{ij}\chi_{S_{10\delta}(x_i,r_j)}\Big)\cdot f\\
         &\leq \delta^{\frac{n}{q}-\frac{1}{s'}}\sup_{\{a_i\}}\sup_{\{b_{ij}\}} \norm{\sum_{x_i}\sum_{r_j}a_ib_{ij}\chi_{S_{10\delta}(x_i,r_j)}}_{L^{p'}(\R^n)}\cdot\norm{f}_{L^p(\R^n)}.
    \end{aligned}
    \end{equation*}
    To prove \eqref{Eq:Nu=1}, it suffices to prove 
    \begin{equation}\label{Eq:DiscritizedVersionWithAB}
       \sup_{\{a_i\}}\sup_{\{b_{ij}\}} \norm{\sum_{x_i}\sum_{r_j}a_ib_{ij}\chi_{S_{10\delta}(x_i,r_j)}}_{L^{p'}(\R^n)}\lessapprox \delta^{\frac{1}{s'}-\frac{\alpha}{q}}.
    \end{equation}
    Without loss of generality, we assume the left-hand side of \eqref{Eq:DiscritizedVersionWithAB} is $\geq \delta^{100}$. Otherwise the estimate holds trivially. Then similar as above, we can throw away those small summands.
    \begin{equation}\label{Eq:BeforeDPLowerBoundObtain}
          \sup_{\{a_i\}}\sup_{\{b_{ij}\}} \norm{\sum_{x_i}\sum_{r_j}a_ib_{ij}\chi_{S_{10\delta}(x_i,r_j)}}_{L^{p'}(\R^n)}\lessapprox  \sup_{\{a_i:a_i\geq \delta^{100}\}}\sup_{\{b_{ij}:b_{ij}\geq \delta^{100}\}} \norm{\sum_{x_i}\sum_{r_j}a_ib_{ij}\chi_{S_{10\delta}(x_i,r_j)}}_{L^{p'}(\R^n)}.
    \end{equation}
    After three rounds of dyadic pigeonholing argument, we obtain $a,b$ such that 
    \begin{itemize}
        \item  $\#\{i:a_i\sim a\}={N}$. For such $i$, $\#\{j:b_{ij}\sim b\}=m$;
        \item RHS of \eqref{Eq:BeforeDPLowerBoundObtain}$\lessapprox ab\norm{\sum_{i:a_i\sim a}\sum_{j:b_{ij}\sim b}\chi_{S_{10\delta}(x_i,r_j)}}_{L^{p'}(\R^n)}$
    \end{itemize}
    Note that ${N}a^{q'}\leq 1, mb^{s'}\leq 1$. Therefore, it suffice to prove 
    \begin{equation}\label{Eq:LastStepOfDiscritization}
       \norm{\sum_{i:a_i\sim a}\sum_{j:b_{ij}\sim b}\chi_{S_{10\delta}(x_i,r_j)}}_{L^{p'}(\R^n)}\lessapprox (m\delta)^{\frac{1}{s'}}\delta^{-\frac{\alpha}{q}}N^{\frac{1}{q'}}.
    \end{equation}
    Since $X'=\{x_i:a_i\sim a\}\subseteq X$ is also a Katz-Tao $(\delta,\alpha,C)$-set, we can apply \eqref{Eq:DiscretizedVersionOfMixedNormEstimate} with $X=X'$, which is exactly \eqref{Eq:LastStepOfDiscritization}.
\end{proof}

A corollary of \ref{Prop:Nu=1Reduction} and Proposition \ref{Prop:Discritization} is that to obtain Theorem \ref{Thm:MainTheorem}, it suffices to prove the following equivalent version.
\begin{proposition}\label{Prop:MainThmDiscretizedEndPointCase}
Let $X\subseteq B\medcap \delta\mathbb Z^n$ be a  {Katz-Tao $(\delta,\alpha,C)$-set} for some $C\lessapprox 1$. For each $x\in X$, Let $R_x$ be a subset of $\{r\in I\medcap \delta\mathbb Z\}$ with $\# R_x = m$. 
 \begin{enumerate}[label={\rm (\arabic*)}]
     \item When $\alpha\in (0,1],s \in [1, \frac{3}{1-\alpha}]$, \begin{equation}\label{Eq:KakeyaTypeEstiamteForAlphaIn0To1}
          \Big\|\sum_{x_i\in X}\sum_{r_j\in R_{x_i}}\chi_{S_{10\delta}(x_i,r_j)}\Big\|_{L^{3/2}(\R^n)}\lessapprox (m\delta)^{\frac{1}{s'}}\delta^{-\frac{\alpha}{3}}(\#X)^{\frac{2}{3}}.
     \end{equation}
      \item When $\alpha\in (1,2],s \in [1,\frac{2}{2-\alpha}]$, \begin{equation}\label{Eq:KakeyaTypeEstiamteForAlphaIn1To2}
          \Big\|\sum_{x_i\in X}\sum_{r_j\in R_{x_i}}\chi_{S_{10\delta}(x_i,r_j)}\Big\|_{L^{2}(\R^n)}\lessapprox (m\delta)^{\frac{1}{s'}}\delta^{-\frac{\alpha}{2}}(\#X)^{\frac{1}{2}}.
     \end{equation}
 \end{enumerate}
\end{proposition}
In what follows, our goal is to prove Proposition \ref{Prop:MainThmDiscretizedEndPointCase}.

\section{Case \texorpdfstring{$\alpha \in (0,1]$}{α∈(0,1]}: tangency bound and the two-scale argument}\label{sec:case_(0,1]}

The argument of Pramanik, Yang and Zahl \cite{PYZ24} is applied to tackle the  case $\alpha\in (0,1]$. 

\subsection{Geometric lemmas and incidence bound} We first record some well-known geometric facts, which can be found in \cite{Wolff_2003_BookWolffLectureNotes,Schlag98,PYZ24}.

For two circles $C_1=C(x,r)$ and $C_2=C(y,s)$ with $x,y\in B=B(0,1/4)$ and $r,s\in I=[1,2]$, define the \textit{tangency parameter} $\Delta(C_1,C_2)$ and \textit{distance parameter} $d(C_1,C_2)$ as 
\begin{align*}\label{Eq:DeltaAndD}
    \Delta(C_1,C_2)=\big||x-y|-|r-s|\big|,\quad
    d(C_1,C_2)=\big| |x-y|+|r-s| \big|.
\end{align*}
See the left picture of Figure \ref{fig:RelationOfTwoCIrcles}.  $\Delta(C_1,C_2)$ is the shorter distance between the intersection points of $C_1$, $C_2$ with the line joining their centers while $d(C_1,C_2)$ is the longer.

It can be checked that $C_1$ and $C_2$ are internally tangent if and only if $\Delta(C_1,C_2)=0$. If we view $C_1=C(x,r)$ and $C_2=C(y,s)$ as two points $(x,r)$ and $(y,s)$  in the parameter space $\R^3$, then $d(C_1,C_2)$ is also comparable to the distance between these two points. When the two circles $C_1$ and $C_2$ can be read from the context, we abbreviate $\Delta(C_1,C_2)$ and $d(C_1,C_2)$ as $\Delta$ and $d$. 

\begin{lemma}[ {\cite[Lemma 2.1]{Schlag98}} ]\label{Lem:GeometricLemma}
For small $\delta>0$, let $C_i^\delta,i=1,2$ be the $\delta$-neighborhood of $C_i$. 
\begin{enumerate}[label={\rm (\arabic*)}]
    \item If $C_1^\delta\medcap C_2^\delta\neq \emptyset
    $, then $|r-s|\leq |x-y|+2\delta$, which implies that $d\sim |x-y|$;
    \item $C_1^{\delta}\medcap C_2^\delta$ is contained in the union of $\delta$-neighborhoods of two arcs, each with length $\frac{\delta}{\sqrt{(\Delta+\delta)(d+\delta)}}$. This implies
    \begin{equation*}
        |C_1^{\delta}\medcap C_2^\delta|\lesssim \frac{\delta^2}{\sqrt{(\Delta+\delta)(d+\delta)}}.
    \end{equation*}
    \item $C_1^{\delta}\medcap C_2^\delta$ is contained in a $\delta$
-neighborhood of an arc centered at $x-r\mathrm{sgn}(r-s)\frac{x-y}{|x-y|}$, with length 
\begin{equation*}
    \lesssim \sqrt{\frac{\Delta+\delta}{d+\delta}}.
\end{equation*}
\end{enumerate}
    
\end{lemma}
See the right picture of Figure \ref{fig:RelationOfTwoCIrcles}.

\begin{figure}
    \centering

    \begin{minipage}{0.45\textwidth}
        \centering
        \tikzset{every picture/.style={line width=0.75pt}} 

\begin{tikzpicture}[x=0.75pt,y=0.75pt,yscale=-1,xscale=1]

\draw    (41,461.26) -- (339,459.81) ;
\draw   (55.77,461.74) .. controls (55.36,404.72) and (101.25,358.17) .. (158.26,357.77) .. controls (215.28,357.36) and (261.83,403.25) .. (262.23,460.26) .. controls (262.64,517.28) and (216.75,563.83) .. (159.74,564.23) .. controls (102.72,564.64) and (56.17,518.75) .. (55.77,461.74) -- cycle ;
\draw   (127.45,461.08) .. controls (127.12,414.57) and (164.55,376.59) .. (211.06,376.26) .. controls (257.58,375.93) and (295.55,413.36) .. (295.88,459.87) .. controls (296.22,506.39) and (258.78,544.36) .. (212.27,544.69) .. controls (165.76,545.03) and (127.78,507.59) .. (127.45,461.08) -- cycle ;
\draw [color={rgb, 255:red, 208; green, 2; blue, 27 }  ,draw opacity=1 ][line width=2.25]    (262.23,460.26) -- (295.88,459.87) ;
\draw    (159,461) -- (211.67,460.48) ;
\draw [shift={(211.67,460.48)}, rotate = 359.43] [color={rgb, 255:red, 0; green, 0; blue, 0 }  ][fill={rgb, 255:red, 0; green, 0; blue, 0 }  ][line width=0.75]      (0, 0) circle [x radius= 3.35, y radius= 3.35]   ;
\draw [shift={(159,461)}, rotate = 359.43] [color={rgb, 255:red, 0; green, 0; blue, 0 }  ][fill={rgb, 255:red, 0; green, 0; blue, 0 }  ][line width=0.75]      (0, 0) circle [x radius= 3.35, y radius= 3.35]   ;
\draw [color={rgb, 255:red, 74; green, 144; blue, 226 }  ,draw opacity=1 ][line width=2.25]    (55.77,461.74) -- (127.45,461.08) ;
\draw  [dash pattern={on 0.84pt off 2.51pt}]  (211.67,460.48) -- (271.67,400.81) ;
\draw  [dash pattern={on 0.84pt off 2.51pt}]  (159,461) -- (72.33,406.81) ;

\draw (153,470.73) node [anchor=north west][inner sep=0.75pt]    {$x$};
\draw (210,471.73) node [anchor=north west][inner sep=0.75pt]    {$y$};
\draw (44,392.4) node [anchor=north west][inner sep=0.75pt]    {$C_{1}$};
\draw (285,393.4) node [anchor=north west][inner sep=0.75pt]    {$C_{2}$};
\draw (266,472.88) node [anchor=north west][inner sep=0.75pt]  [color={rgb, 255:red, 208; green, 2; blue, 27 }  ,opacity=1 ]  {$\Delta ( C_{1} ,C_{2})$};
\draw (61.33,470.21) node [anchor=north west][inner sep=0.75pt]  [color={rgb, 255:red, 74; green, 144; blue, 226 }  ,opacity=1 ]  {$d( C_{1} ,C_{2})$};
\draw (109.33,412.21) node [anchor=north west][inner sep=0.75pt]    {$r$};
\draw (230,412.88) node [anchor=north west][inner sep=0.75pt]    {$s$};

\end{tikzpicture}
    \end{minipage}
    \hfill
    \begin{minipage}{0.45\textwidth}
        \centering

\tikzset{every picture/.style={line width=0.75pt}} 

\begin{tikzpicture}[x=0.75pt,y=0.75pt,yscale=-1,xscale=1]

\draw    (152,164.45) -- (450,163) ;
\draw   (166.77,164.92) .. controls (166.36,107.91) and (212.25,61.36) .. (269.26,60.95) .. controls (326.28,60.54) and (372.83,106.43) .. (373.23,163.45) .. controls (373.64,220.46) and (327.75,267.01) .. (270.74,267.42) .. controls (213.72,267.83) and (167.17,221.94) .. (166.77,164.92) -- cycle ;
\draw   (238.12,165.01) .. controls (237.79,118.23) and (275.44,80.04) .. (322.22,79.7) .. controls (369,79.37) and (407.19,117.02) .. (407.53,163.8) .. controls (407.86,210.58) and (370.21,248.78) .. (323.43,249.11) .. controls (276.65,249.45) and (238.45,211.79) .. (238.12,165.01) -- cycle ;
\draw    (270,164.19) -- (322.82,164.41) ;
\draw [shift={(322.82,164.41)}, rotate = 0.24] [color={rgb, 255:red, 0; green, 0; blue, 0 }  ][fill={rgb, 255:red, 0; green, 0; blue, 0 }  ][line width=0.75]      (0, 0) circle [x radius= 1.34, y radius= 1.34]   ;
\draw [shift={(270,164.19)}, rotate = 0.24] [color={rgb, 255:red, 0; green, 0; blue, 0 }  ][fill={rgb, 255:red, 0; green, 0; blue, 0 }  ][line width=0.75]      (0, 0) circle [x radius= 1.34, y radius= 1.34]   ;
\draw  [dash pattern={on 0.84pt off 2.51pt}]  (322.82,164.41) -- (270.7,229.88) ;
\draw  [dash pattern={on 0.84pt off 2.51pt}]  (270,164.19) -- (183.33,110) ;
\draw  [color={rgb, 255:red, 74; green, 144; blue, 226 }  ,draw opacity=1 ][fill={rgb, 255:red, 74; green, 144; blue, 226 }  ,fill opacity=0.75 ] (320.92,72.05) .. controls (321.14,71.75) and (321.57,71.68) .. (321.87,71.9) -- (337.02,83.09) .. controls (337.32,83.32) and (337.38,83.74) .. (337.16,84.04) -- (335.95,85.68) .. controls (335.73,85.98) and (335.3,86.05) .. (335,85.82) -- (319.85,74.64) .. controls (319.55,74.41) and (319.49,73.99) .. (319.71,73.69) -- cycle ;
\draw [color={rgb, 255:red, 208; green, 2; blue, 27 }  ,draw opacity=1 ]   (373.23,163.45) ;
\draw [shift={(373.23,163.45)}, rotate = 0] [color={rgb, 255:red, 208; green, 2; blue, 27 }  ,draw opacity=1 ][fill={rgb, 255:red, 208; green, 2; blue, 27 }  ,fill opacity=1 ][line width=0.75]      (0, 0) circle [x radius= 3.35, y radius= 3.35]   ;
\draw  [draw opacity=0][line width=1.5]  (317.94,71.04) .. controls (352.5,83.54) and (377.55,120.43) .. (377.55,164) .. controls (377.55,209.1) and (350.7,247.05) .. (314.25,258.19) -- (292.09,164) -- cycle ; \draw  [color={rgb, 255:red, 208; green, 2; blue, 27 }  ,draw opacity=1 ][line width=1.5]  (317.94,71.04) .. controls (352.5,83.54) and (377.55,120.43) .. (377.55,164) .. controls (377.55,209.1) and (350.7,247.05) .. (314.25,258.19) ;  
\draw  [draw opacity=0][line width=1.5]  (315.67,78.47) .. controls (346.58,91.65) and (368.55,125.2) .. (368.55,164.5) .. controls (368.55,205.43) and (344.73,240.11) .. (311.8,252.06) -- (286.55,164.5) -- cycle ; \draw  [color={rgb, 255:red, 208; green, 2; blue, 27 }  ,draw opacity=1 ][line width=1.5]  (315.67,78.47) .. controls (346.58,91.65) and (368.55,125.2) .. (368.55,164.5) .. controls (368.55,205.43) and (344.73,240.11) .. (311.8,252.06) ;  
\draw [color={rgb, 255:red, 208; green, 2; blue, 27 }  ,draw opacity=1 ][line width=1.5]    (311.8,252.06) -- (314.25,258.19) ;
\draw [color={rgb, 255:red, 208; green, 2; blue, 27 }  ,draw opacity=1 ][line width=1.5]    (315.67,78.47) -- (317.94,71.04) ;
\draw [color={rgb, 255:red, 208; green, 2; blue, 27 }  ,draw opacity=1 ]   (322,72) ;
\draw  [color={rgb, 255:red, 74; green, 144; blue, 226 }  ,draw opacity=1 ][fill={rgb, 255:red, 74; green, 144; blue, 226 }  ,fill opacity=0.75 ] (330.94,242.29) .. controls (330.72,241.99) and (330.3,241.93) .. (329.99,242.15) -- (314.85,253.34) .. controls (314.55,253.56) and (314.48,253.99) .. (314.7,254.29) -- (315.92,255.93) .. controls (316.14,256.23) and (316.56,256.29) .. (316.87,256.07) -- (332.01,244.88) .. controls (332.31,244.66) and (332.38,244.24) .. (332.16,243.93) -- cycle ;

\draw (253,166.92) node [anchor=north west][inner sep=0.75pt]    {$x$};
\draw (303,167.12) node [anchor=north west][inner sep=0.75pt]    {$y$};
\draw (173,72.59) node [anchor=north west][inner sep=0.75pt]    {$C_{1}$};
\draw (220,202.59) node [anchor=north west][inner sep=0.75pt]    {$C_{2}$};
\draw (194.33,127.4) node [anchor=north west][inner sep=0.75pt]    {$r$};
\draw (271.94,203.86) node [anchor=north west][inner sep=0.75pt]    {$s$};
\draw (341.16,240.33) node [anchor=north west][inner sep=0.75pt]  [font=\scriptsize,color={rgb, 255:red, 74; green, 144; blue, 226 }  ,opacity=1 ]  {$\frac{\delta }{\sqrt{( \Delta +\delta )( d+\delta )}}$};
\draw (372,92.4) node [anchor=north west][inner sep=0.75pt]  [font=\scriptsize,color={rgb, 255:red, 208; green, 2; blue, 27 }  ,opacity=1 ]  {$\sqrt{\frac{\Delta +\delta }{d+\delta }}$};

\end{tikzpicture}

    \end{minipage}

    \caption{circle relation}
    \label{fig:RelationOfTwoCIrcles}
\end{figure}
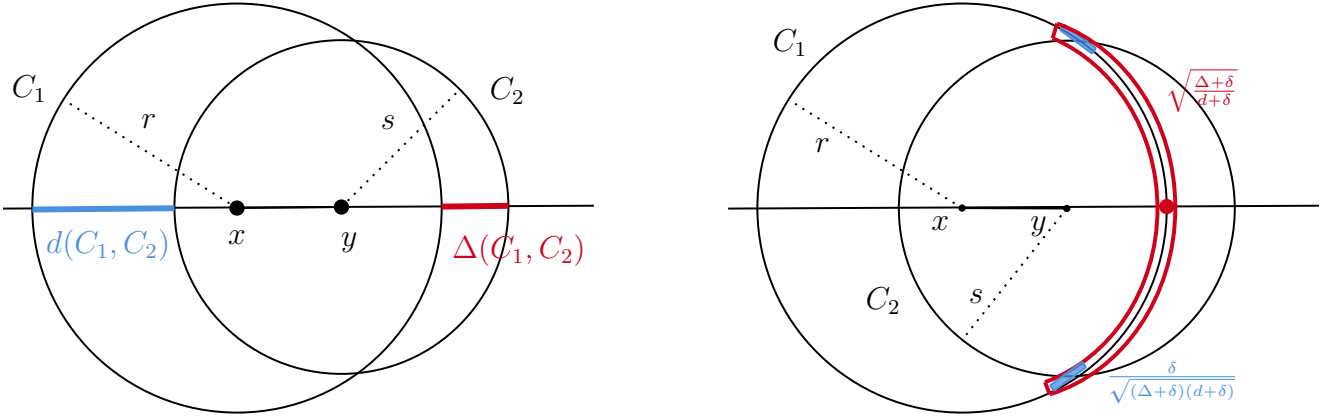

The next result of incidence bound is originally by Wolff \cite{Wolff97}, and is modified by Pramanik, Yang and Zahl \cite{PYZ24} to fit curves with low regularity. In our setting of circles, Wolff's result suffices, whereas for the cinematic curves, a more general version \cite{PYZ24} needs to be invoked.

\subsection{Proof Sketch of the item (1) of Proposition \ref{Prop:MainThmDiscretizedEndPointCase}}. 
Since our proof is a tiny modification of the argument of  Pramanik, Yang, and Zahl \cite{PYZ24}, we will give a sketch of the proof and intuitive explanations.

In what follows, we may identify the circles in $\R^2$ and the points in parameter space and denote by $\mathcal C=\{(x_i,r_j):x_i\in X, r_j\in R_{x_i}\}$ the collection of all circles we consider. Then to prove \eqref{Eq:KakeyaTypeEstiamteForAlphaIn0To1}, it suffices to prove 
\begin{equation}\label{Eq:AlphaIn0To1}
    \norm{\sum_{C\in \mathcal C}
    \chi_{C^\delta}}_{L^{3/2}(\R^2)}^{3/2}\leq C_\eps\delta^{-\eps} (m\delta)^{\frac{3}{2s'}}\delta^{-\frac{\alpha}{2}}\#X.
\end{equation}
\medskip

\noindent\textbf{Constant multiplicity reduction.} 
Let $\mu(x):=\sum_{C\in \mathcal C}
    \chi_{C^\delta}(x)$ be the multiplicity of $x\in \R^2$. Unify this quantity via dyadic pigeonholing.
\begin{equation*}
    \norm{\sum_{C\in \mathcal C}
    \chi_{C^\delta}}_{L^{3/2}(\R^2)}^{3/2}\sim \sum_{\mu:\text{dyadic}}\mu^{3/2}|E_\mu|,
\end{equation*}
where dyadic number $\mu\in [1,\delta^{-100}]$, $E_\mu=\{x\in \R^2:\mu(x)\sim \mu\}$. Therefore, there is $\mu_0$ and $E_{\mu_0}$, which slightly abusing the notations, we still denote as $\mu$ and $E$,  such that 
\begin{equation}\label{Eq:Collet1}
    \norm{\sum_{C\in \mathcal C}
    \chi_{C^\delta}}_{L^{3/2}(\R^2)}^{3/2}\approx \mu^{3/2} |E|
\end{equation}

 Now we endow the set of circles with weak uniformity via two-ends reduction. This helps us to localize the typical tangency parameter and distance parameter. 
\medskip

\noindent\textbf{Two-ends reduction for the distance parameter.} 
For each $x\in E$, consider the supremum 
\begin{equation}\label{Eq:MaximizerTwoEnds}
    \sup_{B(C,t)}\frac{\#\mathcal C(x)\medcap B(C,t)}{t^{\eps}},
\end{equation}
where $\mathcal C(x)\subseteq \mathcal C$ is the circles in $\mathcal C$ whose $\delta$-neighborhoods contain $x$ and $B(C,t)\subseteq B\times [1,2]\subset \R^3$ is the ball centered at $C$ with radius $t$ in the parameter space of circles. The supremum is taken over all $C\in B(0,1/4)\times [1,2]$ and $t\in [\delta,3]$. Let the maximizer be $B(C_x,t_x)$, then we have the correspondence:
\begin{equation*}
    E\ni x\mapsto B(C_x,t_x)\mapsto t_x.
\end{equation*}
Dyadic pigeonholing $\{t_x:x\in E\}$, we obtain a dyadic number $t\in [\delta,3]$, $E_t\subset E$ such that $|E_t|\approx |E|$ and $t_x\sim t$ for all $x\in E_t$. We still denote $E_t$ as $E$.

Cover $\mathcal C\subseteq B(0,1/4)\times [1,2]$ by finitely overlapping balls $B$ with radius $t$. For each $x\in E$, if $B(C_x,t)\medcap B\neq \emptyset$, then $\mathcal C(x)\subseteq 3B\medcap \mathcal C$, which we denote as $\mathcal C_B$. Since $\{B\}$ is a finitely overlapping cover of $\mathcal C$, each $B(C_x,t)$ intersects with at most $O(1)$ many $B$'s. If we denote as $E_B$ the points $x\in E$ whose $B(C_x,t)$ intersects $B$, then $\{E_B\}_B$ forms a finitely overlapping of $E$. 
From now on, we will focus on each such $B$. For each $x\in E_B$, \eqref{Eq:MaximizerTwoEnds} implies that  \begin{equation}\label{Eq:TwoEndsMaximizerForT}
    \#\mathcal C_B(x):=\#\mathcal C(x)\medcap 3B\geq \#\mathcal C(x)\medcap B(C_x,t)\gtrsim t^{\eps}\#\mathcal C(x)=t^{\eps}\mu.
\end{equation}
Dyadic pigeonholing $\{\#\mathcal C_B(x)\}_{x\in E_B}$, we obtain $\mu_1\approx \mu$ and $E_{\mu_1,B}\subseteq E_B$ with $|E_{\mu_1,B}|\approx |E_B|$ such that $\#\mathcal C_B(x)\sim \mu_1$ for all $x\in E_{\mu_1,B}$. As usual, we still denote it as $E_B$.
\medskip

\noindent\textbf{Two-ends reduction for the tangency parameter.}
For $x\in E_B$, consider the supremum 
\begin{equation*}
    \sup_{\substack{\bar C\in\mathcal C_B(x), 
    \Delta\in [\delta,6t]}}\frac{\#\{C\in \mathcal C_B(x):\Delta(C,\bar C)\leq \Delta\}}{\Delta^{\eps^2}}.
\end{equation*}
Note that by the first item of Lemma \ref{Lem:GeometricLemma}, it suffices to assume $\Delta\leq 6t.$
Let $\bar C_x, \Delta_x$ be the maximizer. Then we have the correspondence
\begin{equation*}
   E_B\ni x\mapsto \Delta_x.
\end{equation*}
Dyadic pigeonholing $\{\Delta_x:x\in E_B\}$, we obtain a dyadic number $\Delta\in [\delta,6t]$, $E_{\Delta,B}\subset E_B$ such that $|E_{\Delta,B}|\approx |E_B|$ and $\Delta_x\sim \Delta$ for all $x\in E_{\Delta,B}$. We still denote $E_{\Delta,B}$ as $E_B$.
This finishes our localization process. Note that 
\begin{equation}\label{Eq:Collect2}
    |E|\sim \sum_B|E_B|,\quad \#\mathcal C\sim \sum_B\#\mathcal C_B, \quad \mu^{3/2}|E|\approx \sum_B\mu_1^{3/2}|E_B|,
\end{equation}
and we will prove a similar \eqref{Eq:AlphaIn0To1} for each $B$. Note that $\mu_1$ in fact depends on $B$, but this is not relevant to our purpose. For each $B$, 
\begin{equation}\label{Eq:Collect3}
    \mu_1^{3/2}|E_B|=\norm{\sum_{C\in \mathcal C_B}\chi_{C^\delta}}_{L^{3/2}(E_B)}^{3/2}.
\end{equation}
\medskip

\noindent\textbf{Fine scale rectangle.}
As indicated by the item 2 of Lemma \ref{Lem:GeometricLemma}, $E_B$ is a union of disjoint fine scale (curvilinear) rectangles with dimensions $\delta\times {\delta}/{\sqrt{(\Delta+\delta)(t+\delta)}}$, which means that the typical intersection pattern of the circles in $\mathcal C_B$ is the fine scale rectangle. On each such rectangle, the multiplicity is $\mu_1$. It is remained to estimate how many fine scale rectangles are there in the plane. However, unfortunately, we cannot directly count the number of such fine scale rectangles via Wolff's incidence bound (or its variants). To explain what exactly the incidence bound is, we introduce the following definitions.
\medskip

\noindent\textbf{Tangency rectangle and incidence bound.}
 A \textit{$(\delta,t)$-rectangle} $R$ is a $\delta$-neighborhood of an arc of circles in $B(0,1/4)\times [1,2]$ with length $\sqrt{\delta/t}$. We say two rectangles $R,R'$ are \textit{$\lambda$-comparable} if there is another $(\lambda\delta,t)$-rectangle that contains both of $R$ and $R'$. Otherwise we say they are incomparable. Wolff's tangency bound is to estimate the number of such incomparable $(\delta,t)$-rectangles.
 \begin{proposition}[incidence bound, \cite{PYZ24}]\label{Prop:IncidenceBound}
     Assume $\mathcal C$ is a set of circles. Let $\delta\lesssim t$ and $\mathcal W,\mathcal B\subseteq \mathcal C$ be finite sets. Suppose $\mathcal W\medcup \mathcal B$ has diameter at most $6t$ and the any pair $(W,B)\in \mathcal W\times \mathcal B$ is $t/10$-separated.

     Let $\mathcal R$ be a collection of pairwise 100-incomparable $(\delta,t)$-rectangles. Suppose each rectangle in $\mathcal R$ is tangent to at least $\mu$ many elements of $\mathcal W$ and at least $\nu$ many elements of $\mathcal B$. Then 
     \begin{equation*}
         \#\mathcal R\lessapprox \Big(\frac{\#\mathcal W}{\mu}+\frac{\#\mathcal B}{\nu}\Big)^{3/2}.
     \end{equation*} 
 \end{proposition}
\medskip

\noindent\textbf{Coarse scale rectangle.}
With this proposition, we can bound the number of the fine scale rectangles formed by $\mathcal C_B$ only  when $\sqrt{\delta/t}=\delta/\sqrt{(\Delta+\delta)(t+\delta)}$, which is merely the case that $\Delta\sim \delta.$  To tackle arbitrary $\Delta\in [\delta,6t], \delta \leq t\leq 3$, in light of this fact and the item 3 of Lemma \ref{Lem:GeometricLemma}, we consider the $\Delta$-neighborhood of the circles in $\mathcal C_B$. Then all circles in $\mathcal C_B$ are $\Delta$-tangent and we can bound the number of $(\Delta,t)$-rectangles. These are the \textit{coarse scale rectangles}.

After some uniformization arguments for the quantities of rectangles, to count the number of fine scale rectangles, we first count the number of fine scale rectangles in each coarse scale rectangle via a bilinear $L^2$ argument and then count the number of coarse scale rectangles via incidence bound. 
\medskip

\noindent\textbf{Bilinear $L^2$ argument.}
Let $R\in \mathcal R$ be a typical coarse scale rectangle, then 
\begin{equation*}
     \norm{\sum_{C\in \mathcal C_B}\chi_{C^\delta}}_{L^{2}(E_B)}^{2}\approx \sum_{R\in \mathcal R} \norm{\sum_{C\in \mathcal C_B}\chi_{C^\delta}}_{L^{2}(E_B\medcap R)}^{2}.
\end{equation*}
For each term of the summations in the right-hand side,
\begin{equation*}
    \norm{\sum_{C\in \mathcal C_B}\chi_{C^\delta}}_{L^{2}(E_B\medcap R)}^{2}=\sum_{\substack{C,C'\in \mathcal C_B\\C,C'\sim R}}|C^{\delta}\medcap C'^\delta|\lesssim \mathcal (\#\mathcal C_B(R))^2\frac{\delta^2}{\sqrt{(\Delta+\delta)(t+\delta)}}.
\end{equation*}
Here we use $C\sim R$ to denote that $C$ is $\Delta$-tangent to $R$, $\mathcal C_B(R)$ is the set of all circles in $\mathcal C_B$ that are $\Delta$-tangent to $R$. In the last inequality, we uses the item 2 of Lemma \ref{Lem:GeometricLemma}.
\medskip

\noindent\textbf{$L^1$ argument.} Similarly, 
\begin{equation*}
     \norm{\sum_{C\in \mathcal C_B}\chi_{C^\delta}}_{L^{1}(E_B)}\approx \sum_{R\in \mathcal R} \norm{\sum_{C\in \mathcal C_B}\chi_{C^\delta}}_{L^{1}(E_B\medcap R)}
\end{equation*}
and
\begin{equation*}
    \norm{\sum_{C\in \mathcal C_B}\chi_{C^\delta}}_{L^{1}(E_B\medcap R)}\approx \norm{\sum_{C\in \mathcal C_B}\chi_{C^\delta}}_{L^{1}(R)}=\norm{\sum_{\substack{C\in \mathcal C_B\\C\sim R}}\chi_{C^\delta}}_{L^{1}(R)}\sim\#\mathcal C_B(R) \delta\cdot\sqrt{\frac{\Delta}{t}}.
\end{equation*}
\begin{remark}
   All facts above are not as trivial as we present here and can be rigorously proved by the language of ``shadings'', see \cite{PYZ24}.
\end{remark}
\medskip

\noindent\textbf{Interpolation.} By H\"older's inequality, we obtain that
\begin{equation*}
    \norm{\sum_{C\in \mathcal C_B}\chi_{C^\delta}}_{L^{3/2}(E_B)}^{3/2}\lessapprox  \#\mathcal R (\#\mathcal C_B(R)\delta)^{3/2} t^{-1/2}.
\end{equation*}
\medskip

\noindent\textbf{Applying incidence bound.}
Heuristically speaking, Proposition \ref{Prop:IncidenceBound} says that $$\#\mathcal R\lessapprox \Big(  \frac{\#\mathcal C_B}{\#\mathcal C_B(R)}\Big)^{3/2}.$$ Therefore, 
\begin{equation}\label{Eq:Collect4}
     \norm{\sum_{C\in \mathcal C_B}\chi_{C^\delta}}_{L^{3/2}(E_B)}^{3/2}\lessapprox  \#\mathcal R (\#\mathcal C_B(R)\delta)^{3/2} t^{-1/2}\lessapprox (\#\mathcal C_B\delta)^{3/2}t^{-1/2}.
\end{equation}
\begin{remark}
    In the rigorous proof, to apply the genuine bilinear incidence bound in Proposition \ref{Prop:IncidenceBound}, we need to divide $\mathcal C_B$ into two $t$-separated parts by pigeonhole principle. See \cite{PYZ24}. We omit all of the technicalities.
\end{remark}
\medskip

\noindent\textbf{Close the argument.} Thus far, our discussions have been quite general and do not rely on the mixed-norm setting. At this final stage, we need to make the parameter ``$m$'' come into play. Note that
\begin{equation}\label{Eq:Incid1}
    \#\mathcal C_B\leq \#X_B m_B, 
\end{equation}
where $X_B=X\medcap \pi(B)$, $\pi$ is the orthogonal projection to the first two coordinates and 
\begin{equation*}\label{Eq:Incid2}
    m_B=\sup_{x\in X_B}\#R_{x}.
\end{equation*}
Also, we have 
\begin{equation*}\label{Eq:Incid3}
    X_B\lessapprox (t/\delta)^\alpha,\quad m_B\lessapprox (t/\delta).
\end{equation*} 
Therefore,
\begin{equation*}
    (\#\mathcal C_B\delta)^{\frac{3}{2}}t^{-\frac{1}{2}}\overset{\eqref{Eq:Incid1}}{\lessapprox }{\color{red}\#X_B}{\color{blue}m_B} \underline{{\color{orange}(\#\mathcal C_B)^{\frac{1}{2}}}}(\delta/t)^{\frac{1}{2}}{\color{red}\delta}\lessapprox {\color{orange}m_B^{\frac{1}{2}-\frac{\alpha}{2}}}\underline{(t/\delta)^{\frac{\alpha}{2}}}(\delta/t)^{\frac{1}{2}}{\color{blue}m_B}^{{\color{blue}1}+{\color{orange}\frac{\alpha}{2}}}{\color{red}\#X_B\delta}\lessapprox {\color{blue}m_B}^{{\color{blue}1}+{\color{orange}\frac{\alpha}{2}}}{\color{red}\#X_B\delta}.
\end{equation*}
In the second inequality we bound $(\#\mathcal C_B)^{1/2}$ by $(\#X_Bm_B)^{1/2}\lessapprox (t/\delta)^{\alpha/2}m_B^{1/2}$. In the third inequality, we need $\alpha\leq 1$ to bound $m_B^{1/2-\alpha/2}$ by $(t/\delta)^{1/2-\alpha/2}$, which cancels $\underline{(t/\delta)^{\frac{\alpha}{2}}}(\delta/t)^{\frac{1}{2}}$.

Collect \eqref{Eq:Collet1}, \eqref{Eq:Collect2}, \eqref{Eq:Collect3}, \eqref{Eq:Collect4}, and the last inequality, sum over all $B$ and use $m_B\leq m$, $\sum_B\#X_B=\#X$ we obtain for $s=\frac{3}{1-\alpha}$, 
\begin{equation*}
    \norm{\sum_{C\in \mathcal C}
    \chi_{C^\delta}}_{L^{3/2}(\R^2)}^{3/2}\leq C_\eps\delta^{-\eps} (m\delta)^{\frac{3}{2s'}}\delta^{-\frac{\alpha}{2}}\#X.
\end{equation*}
By Fact \ref{Fact:Holder}, this concludes the proof.

\section{Case \texorpdfstring{$\alpha \in (1,2]$}{α∈(1,2]}: slicing method and refined \texorpdfstring{$L^2$}{L2} estimates}\label{sec:case_(1,2]}

\subsection{Prerequisites: slicing theorems}\label{subsec:Slicing}As we mentioned in Section \ref{subsec:SketchOfTheProof}, our main tool of coping this case is slicing theorems, which we introduce now.

Let $\mu\in \mathcal M(\R^n)$ denote $\mu$ is a non-zero finite Borel measure supported in $\R^n$ and $\mathcal H^m$ be the $m$-dimensional Hausdorff measure. For any $V\in G(m,n),1\leq m< n$ and $\mathcal H^m$ almost all $a\in V$, define the \textit{sliced measure} $\mu_{V,a}$ with the properties that 
\begin{equation*}
    \mathrm{spt}(\mu_{V,a})\subseteq \mathrm{spt}(\mu)\medcap V_a^\perp\quad \text{where } V_a^\perp=V^\perp+a,
\end{equation*}
and for $\varphi\in _0(\R^n)$,
\begin{equation*}
    \int_V\int\varphi\mathrm d\mathcal H^ma=\int \varphi \mathrm d\mu\quad \text{if } {P_V}_\sharp\mu \ll \mathcal H^m|_V,
\end{equation*}
where ${P_V}_\sharp\mu$ is the push-forward measure on $V$ under the projection map $P_V$ from $\R^n$ to $V$ and $\mathcal H^m|_V$ is the restricted measure to $V$.

\begin{theorem}[Mattila \cite{Mattila15}, Proposition 6.4] \label{Thm:MattilaSlicingTheorem}
    Let $m < s < n$, $\mu \in \mathcal{M}(\mathbb{R}^n)$ and let $\nu \in \mathcal{M}(G(m,n))$ be such that for some $t > m(n-m) + m - s$,
\[
\nu(B(V,r)) \le r^{t} \quad \text{for } V \in G(m,n),\ r > 0.
\]
Then
\begin{equation}\label{Eq:SlicingTheorem}
\int\int_{V} I_{s-m}(\mu_{V,a}) \, \mathrm d\mathcal{H}^m a \mathrm d\nu V\le C_{n,m,s}  I_s(\mu),
\end{equation}
where $I_s(\mu)=\iint |x-y|^{-s}\mathrm d\mu x\mathrm d \mu y$ is the $s$-energy of $\mu$.
\end{theorem}

For our purposes, it is technically important to turn Theorem \ref{Thm:MattilaSlicingTheorem} into a discretized version for Katz-Tao sets. This is because in for $L^2$ Córdoba-type arguments to work on each slice, it is often awkward to work with energy integrals. Also, it seems hard to directly sum up the information we get for all sliced measures $\mu_{V,a}$. On the other hand, counting $\delta$-covering numbers will never incur these annoyances. Now we state the main result in this subsection.
\begin{theorem}\label{Thm:DiscretizedSlicing}
    Let $\alpha\in (1,n]$. Assume $X\subseteq \delta\mathbb Z^n\medcap [-1,1]^n$ is a Katz-Tao $(\delta,\alpha,C)$-set for some $C\lessapprox 1$. Let $\eps_0$ be any fixed small positive number, $\nu$ is the normalized Haar measure on $G(m,n)$. Then for all $V\in G(m,n)$ but a subset with $\nu$-measure at most $\eps_0$, there is a set of $O(\delta^{-m})$ many essentially distinct $\delta$-neighborhood of translated copies of $V^{\perp}$, denoted as $\{{}^jV^\perp_{\delta}\}_{j=1}^{O(\delta^{-m})}$, perpendicular to $V_i$ such that the following holds.
    \begin{enumerate}[label={\rm (\arabic*)}]
    \item  $\medcup_{j}{}^jV^\perp_{\delta}$ covers $X'\subseteq X$ with $\#X'\gtrapprox \#X$.
    \item For all ${}^jV^\perp_{\delta}$, $
    X'\medcap {}^jV^\perp_{\delta}$
is a Katz-Tao $(\delta,\alpha-m,C)$-set.
\end{enumerate}
\end{theorem}

Theorem~\ref{Thm:DiscretizedSlicing} should be intuitively acceptable, but it does not seem to exist in the literature. One reason might be that people did not find reasonable applications (outside of pure geometric measure theory) for such an extension, and in this paper we provide one regarding the weighted mixed-norm estimates. Similar discretization procedures were previously done by Peres-Shmerkin \cite[Proposition~7]{PS09} and Shmerkin \cite[Proposition~5]{shmerkin08} for orthogonal projections on $\R^2$, which are the earliest version we can find and and are used to study certain resonance between generalized Cantor sets. We should also point out that for technical reasons one can only expect the conclusions of Theorem~\ref{Thm:DiscretizedSlicing} to hold for a suitable refinement of $X$, and have to abandon a positive proportional of directions in $G(m,n)$ (one can easily construct counterexamples). However, as we will see in the next subsection, these will not affect our purpose.

To prove Theorem~\ref{Thm:DiscretizedSlicing}, we need the following classical results from geometric measure theory.
\begin{lemma}\label{Lem:BasicGMTLemma}
    Let $\mu$ be a non-trivial Borel measure supported in $B(0,1)$ and $\norm{\mu}$ be the total variance of $\mu$. Then 
    \begin{enumerate}[label={\rm (\arabic*)}]
        \item\label{Item:Item1OfBasicGMTLemma} For all $0<t<s$, $ \mu(B_r)\leq C_\mu r^s \text{ for any } B_r \Longrightarrow  I_t(\mu)\leq C_{s,t}\norm{\mu}$.
        \item\label{Item:Item2OfBasicGMTLemma} $I_s(\mu)<\infty \Longleftrightarrow $ There is a non-empty set   $A\subseteq  \spt(\mu)$ such that 
        \begin{equation*}
         (\mu|_A)(B_r)\leq C_s\frac{I_s(\mu)}{\norm{\mu}}r^s   
        \end{equation*} for any $B_r$.
    \end{enumerate}
\end{lemma}
The first item is a direct application of spherical coordinates, see \cite[Page 19]{Mattila15}. Whereas the second is an application of popular principle, see \cite[Lemma 8.3]{Wolff_2003_BookWolffLectureNotes}.

Now we are ready to prove the theorem.
\begin{proof}[Proof of Theorem~\ref{Thm:DiscretizedSlicing}]
     \noindent\textbf{Apply Theorem \ref{Thm:MattilaSlicingTheorem}.} Define 
    \begin{equation*}
        \mu:=\delta^{n-\alpha}\chi_{X_\delta
        }.
    \end{equation*}
    It is straightforward to verify that $\mu$ is an $\alpha$-dimensional Frostman measure with ball condition constant $C$. Indeed, for any ball $B_r$ with radius $r\in (0,1)$, if $r\in (0,\delta]$, 
    \begin{equation*}
        \delta^{-n+\alpha}r^n\leq \delta^{-n+\alpha}\mathcal L^{n}(B_r\medcap X_\delta)=\mu(B_r)\leq C r^{\alpha};
    \end{equation*}
    if $r\in (\delta,1]$, 
    \begin{equation*}
        \mu(B_r)\leq \delta^{-n+\alpha}\#B_r\medcap X\cdot\delta^{n}\leq Cr^\alpha.
    \end{equation*}

     Apply Theorem \ref{Thm:MattilaSlicingTheorem} with $m<s<\alpha\leq n$ and $t=m(n-m)$. We have \eqref{Eq:SlicingTheorem} holds.
    \begin{equation*}
\int\int_{V} I_{s-m}(\mu_{V,a}) \, \mathrm d\mathcal{H}^m a \mathrm d\nu V\le C_{n,m,s}  I_s(\mu).
\end{equation*}

    \noindent\textbf{Remove outer integral.} We claim that for all $V\in G(m,n)$ but a subset with $\nu$-measure at most $\eps_0$, the inner integral
    \begin{equation}\label{Eq:InnerIntegral2}
        \int_V I_{s-m}(\mu_{V,a})d\mathcal H^m a\leq \eps_0^{-1}C_{n,m,s} I_s(\mu).
    \end{equation}
    In fact, it suffices to prove that 
    \begin{equation*}
       E:= \Big\{V\in G(m,n):\int_{V} I_{s-m}(\mu_{V,a}) \, \mathrm d\mathcal{H}^m a>\eps_0^{-1}C_{n,m,s} I_s(\mu)\Big\}.
    \end{equation*}
    has measure at most $\eps_0$. This is because 
    \begin{equation*}
        \nu(E)\eps_0^{-1}C_{n,m,s} I_s(\mu)\leq \int_E\int_{V} I_{s-m}(\mu_{V,a}) \, \mathrm d\mathcal{H}^m a \mathrm d\nu V\le C_{n,m,s}  I_s(\mu).
    \end{equation*}
    \noindent\textbf{Refine the set $X$.} From now on, we focus on each fixed such $V\in G(m,n)\setminus E$. Recall that $P_V$ is the orthogonal projection from $\R^n$ to $V$. For technical reason, instead of taking a $\delta$-net, we take a  $\delta/4$-net $V'$ of  $P_V(X_\delta)\subseteq V$. Then 
    \begin{itemize}
        \item $P_V(X_\delta)\subseteq \medcup_{v\in V'}B^{m}(v,\delta/4)$ and each point $a\in P_V(X_\delta)$ is covered by at most $C_n$ times. Therefore, \eqref{Eq:InnerIntegral2}implies that 
    \begin{equation}\label{Eq:EachTube}
        \sum_{v\in V'}\int_{B^m(v,\delta/4)}I_{s-m}(\mu_{V,a})d\mathcal H^m a\leq C_n\eps_0^{-1}C_{n,m,s} I_s(\mu).
    \end{equation}
    \item   $X\subseteq X_\delta\subseteq \medcup_{v\in V'}P_V^{-1}(B^m(v,\delta/4))$. Dyadic pigeonholing $\{\#X\medcap P_V^{-1}(B^m(v,\delta/4))\}_v$, we have a dyadic number $\tau$ and a refined subset of $V''\subseteq V'$, such that 
    \begin{enumerate}
        \item \label{Item:Item1OfCOuntingNumber}for each $v\in V''$, $\#X\medcap P_V^{-1}(B^m(v,\delta/4))\sim \tau$;
        \item \label{Item:Item2OfCOuntingNumber}$\#X\medcap \big(\medcup_{v\in V''}P_V^{-1}(B^m(v,\delta/4))\big)=\#V''\tau\geq (\log\delta^{-1})^{-1}\#X$.
    \end{enumerate} 
    \end{itemize}
\noindent\textbf{Analyze the slice measure $\mu_{V,a}$.} Denote as $V''_\delta$ the $\medcup_{v\in V''}B^m(v,\delta/4)$. We claim that for $a\in V''_\delta$, 
\begin{equation*}
    \norm{\mu_{V,a}}\mathcal L^{m}(V''_\delta)\approx \delta^\alpha\#X.
\end{equation*}
Indeed, note that
\begin{equation*}
    \mu_{V,a}=\delta^{-n+\alpha}\chi_{X_\delta\medcap (V^\perp+a)},\quad \mathrm{spt}(\mu_{V,a})=X_\delta\medcap (V^\perp+a).
\end{equation*}
For each $a\in V''_\delta$, there is $v_a\in V''$ such that $a\in B^m(v_a,\delta/4)$.  By Item \eqref{Item:Item1OfCOuntingNumber} $\#X\medcap P_V^{-1}(B^m(v_a,\delta/4))\sim \tau$ and $X$ are $\delta$-separated set of points, therefore, $X_\delta \medcap P_V^{-1}(a)$ is a union of $\delta$-balls in $\R^{n-m}$
\begin{equation*}
    \mathcal L^{n-m}(X_\delta \medcap a+V^\perp)=\mathcal L^{n-m}(X_\delta \medcap P_V^{-1}(a))\sim \tau \delta^{n-m}.
\end{equation*}
Therefore, $\norm{\mu_{V,a}}\sim \delta^{-n+\alpha}\delta^{n-m}\tau\sim \tau\delta^{\alpha-m}$. By Item \eqref{Item:Item2OfCOuntingNumber}, we obtain
\begin{equation*}
    \norm{\mu_{V,a}}\mathcal L^m(V''_\delta)\sim \tau \delta^{\alpha-m}\delta^m\#V''\approx \delta^\alpha\#X.
\end{equation*}
\noindent\textbf{Analyze the inner integral.} Form \eqref{Eq:EachTube}, we have for each $V\in G(m,n)\setminus E, \nu(E)\leq \eps_0$,
\begin{equation*}
    \int_{V''_\delta}I_{s-m}(\mu_{V,a})d\mathcal H^m(a)\leq C_{n,m,s} I_s(\mu).
\end{equation*}
By a similar argument analyzing $\nu(E)$, we have 
\begin{equation*}
    \mathcal H^m\Big(\big\{ a\in V_\delta'':I_{s-m}(\mu_{V,a})\leq \frac{2C_{n,m,s} I_s(\mu)}{\mathcal L^{m}(V''_\delta)}\big \}\Big)>\frac{1}{2}\mathcal L^{m}(V''_\delta).
\end{equation*}
For every $a\in\big\{ a\in V_\delta'':I_{s-m}(\mu_{V,a})\leq \frac{2C_{n,m,s} I_s(\mu)}{\mathcal L^{m}(V''_\delta)}\big \}$ , 
\begin{equation*}
    I_{s-m}(\mu_{V,a})\leq \frac{2C_{n,m,s} I_s(\mu)}{\mathcal L^{m}(V''_\delta)}<\infty.
\end{equation*}
By Item \ref{Item:Item2OfBasicGMTLemma} of Lemma \ref{Lem:BasicGMTLemma}, there is $A_a\subseteq \mathrm{spt}(\mu_{V,a})=X_\delta\cap V^\perp+a$, in this case, it is a union sub-collections of $\delta$-balls in $\R^{n-m}$ that lose at most half of the $\delta$-balls in $\mathrm{spt}(\mu_{V,a})$, see Remark \ref{Remark:LoseAtMostHalf} below. Furthermore, we have the ball condition $\mu_{V,a}|_{A_a}(B_r)\leq C_a r^{s-m}$, where 
\begin{align*}
    C_a\leq C_{n,m,s} \frac{I_{s-m}(\mu_{V,a})}{\norm{\mu_{V,a}}} &\leq C_{n,m,s} \frac{I_{s}(\mu)}{\norm{\mu_{V,a}}\mathcal L^{m}(V''_\delta)}\\
   \text{(By Item \ref{Item:Item1OfBasicGMTLemma} of Lemma~\ref{Lem:BasicGMTLemma})} &\leq C_{n,m,s} \frac{C\norm{\mu}}{\norm{\mu_{V,a}}\mathcal L^{m}(V''_\delta)}\\
   &\lessapprox C_{n,m,s} C\frac{\delta^\alpha \#X}{\delta^\alpha \#X}\\
   &\leq CC_{n,m,s} .
\end{align*}

   This implies that 
  \begin{equation}\label{Eq:Non-ConcentrationOnEachSlice}
        \delta^{-n+\alpha}\chi_{X'_\delta\medcap A_a}(B^{n-m}_r)\leq \delta^{-n+\alpha}\chi_{X_\delta\medcap A_a}(B^{n-m}_r)\leq CC_{n,m,s} r^{s-m},
  \end{equation}
  which gives 
  \begin{equation}\label{Eq:Eq:Non-ConcentrationOnEachSliceVersion2}
      \delta^{-n+\alpha}|\big( X_\delta'\medcap A_a\big)\medcap B_r^{n-m}|_\delta\delta^{n-m}\leq CC_{n,m,s} r^{s-m}.
  \end{equation}
  \noindent\textbf{Verify the desired properties.} Since each $\delta$-ball in $\R^{n-m}$ is the intersection of some $B^n(x,\delta)\subseteq X_\delta'$ with $ a+V^\perp$ for some $x\in P_V^{-1}(B^m(v_a,\delta/4))\medcap X'$. Note that passing from $X'_\delta$ to $A_a \cap X'_\delta$ removes at most $\tau/2$ many $\delta$-balls. Since this is only a minor refinement, we continue to denote the resulting collection by $X'_\delta$. Moreover, the non-concentration conditions \eqref{Eq:Non-ConcentrationOnEachSlice} and \eqref{Eq:Eq:Non-ConcentrationOnEachSliceVersion2} continue to hold for this refined collection of $\delta$-balls. Define ${}^jV^\perp_\delta$ as $P_V^{-1}(B^m(v,\delta/4)),v\in V''$, obviously $\#\{j\}=O(\delta^{-m})$, For each $B_r^n$, 
  \begin{equation*}
      \#X'\medcap {}^jV^\perp_\delta\medcap B_r^n\sim |\big( X_\delta'\medcap A_a\big)\medcap B_r^{n-m}|_\delta\leq CC_{n,m,s} \frac{r^{s-m}}{\delta^{\alpha-m}}.
  \end{equation*}
  Let $s\to \alpha$, the error term can be absorbed by $C$ and we conclude the proof.
\end{proof}
\begin{remark}\label{Remark:LoseAtMostHalf}
    In the proof, we claimed that $A_a$ is a union of $\delta$-balls in $\R^{n-m}$ that loses at most half of the $\delta$-balls in $\mathrm{spt}(\mu_{V,a})$. This follows for the following reasons.
\begin{enumerate}
    \item In the proof of the item 2 of Lemma \ref{Lem:BasicGMTLemma}, $A_a$ is defined as the set of all $y \in \mathrm{spt}(\mu_{V,a})$ such that the potential (a function of $y$)
    \[
    \int |x-y|^{-(s-m)} \, d\mu_{V,a}(x)
    \]
    of $\mu_{V,a}$ at $y$ is bounded by
    \[
    \frac{2I_{s-m}(\mu_{V,a})}{\|\mu_{V,a}\|}.
    \]
    By the construction of popular principle, the set of such points carries at least $\frac{1}{2}\|\mu_{V,a}\|$ of the total mass.

    \item Moreover, thanks to the specific form of $\mu$, this potential is essentially constant on each $\delta$-ball. Hence if a $\delta$-ball contains one point of $A_a$, then the whole ball may be regarded as lying in $A_a$.
\end{enumerate}
Recall that $\mu_{V,a}$ is essentially the characteristic function of a union of $\delta$-balls. Therefore, the two observations above imply that $A_a$ contains at least half of the $\delta$-balls in $\mathrm{spt}(\mu_{V,a})$. We omit the routine technical details.
\end{remark}
A direct corollary of this theorem is the following special case in $\R^2$.
\begin{corollary}\label{Coro:SlicingForCentersOnThePlane}
    Let $\alpha\in (1,2]$. Assume $X\subseteq \delta\mathbb Z^2\medcap [0,1]^2$ is a Katz-Tao $(\delta,\alpha,C)$-set for some $C\lessapprox 1$ and  $D_i=[\frac{2\pi i}{100},\frac{2\pi (i+1)}{100}]\medcap [0,2\pi],i=0,1,...,99$. Then for all $i\in \{0,...,99\}$, there is $d_i\in D_i$ {\rm(}view it as a direction{\rm)} and a set of essentially distinct $\delta$-tubes $\mathbb T_i=\{T_{i,j}\}_{j=1}^{O(\delta^{-1})}$ perpendicular to $d_i$, such that the following holds. 
\begin{enumerate}[label={\rm (\arabic*)}]
    \item $\medcup_{j}T_{i,j}$ covers (most of) $X$.
    \item For all $T_{i,j}$, $
    X\medcap T_{i,j}$
is a Katz-Tao $(\delta,\alpha-1,C)$-set.
\end{enumerate}
\end{corollary}

\subsection{Proof of the item (2) of Proposition \ref{Prop:MainThmDiscretizedEndPointCase}}
In light of Corollary \ref{Coro:SlicingForCentersOnThePlane}, we can decompose each circle into 100 many small arcs with equal-length and prove a uniform estimate for each arc. This trick was used in \cite{HJ25,Zahl25}. Therefore, we may assume that all circles are the upper arc with length $\sim 2\pi/100$. To prove \eqref{Eq:KakeyaTypeEstiamteForAlphaIn1To2}, it suffices to prove 

\begin{equation*}
    \Big\|\sum_{x_i\in X}\sum_{r_j\in R_{x_i}}\chi_{C^*_{\delta}(x_i,r_j)}\Big\|_{L^{2}(\R^n)}\lessapprox (m\delta)^{\frac{1}{s'}}\delta^{-\frac{\alpha}{2}}(\#X)^{\frac{1}{2}}
\end{equation*}
where $C^*(x,r)$ is the an arc with length $\sim 2\pi/100$ of circle $C(x,r)$ and similar as above, we use $C_\delta^*(x,r)$ to denote its $\delta$-neighborhood. The proof is based on a Cord\'oba type argument. 

Note that the set of centers $X$ is a Katz-Tao $(\delta,\alpha,C)$-set. It is tempting to directly apply Corollary~\ref{Coro:SlicingForCentersOnThePlane} to find a set of parallel $\delta$-tubes $\{T_k\}_{k}$ (slices) that cover most of $X$. In each $T_k$ (slice), the centers form a Katz-Tao $(\delta,\alpha-1,C)$-set. However, the problem is that the resulting refinement of $X$ does not necessarily preserve the $L^2$-mass above, since we have no a priori knowledge about the refinement. To fix this technical issue, we need to apply Corollary~\ref{Coro:SlicingForCentersOnThePlane} before the duality step in Section~\ref{sec:discretization}. A more precise sketch is as follows:
\begin{enumerate}[label={\rm (\alph*)}]
    \item Given a function $f\in L^p(\R^2)$, we start from (\ref{Eq:IntermediateEstimate}) with $\langle X \rangle_\alpha \lessapprox 1$ (by the reduction in Section~\ref{sec:discretization}), and do the same pigeonholing argument as in (\ref{Eq:DPForUsingUniform}), which yields a refinement $X'$ of $X$;

    \item Next, we apply Corollary~\ref{Coro:SlicingForCentersOnThePlane} to this $X'$ and get a further refinement $X''$. Note that by the initial uniformization refinement, the mixed-norm of $\AA_\delta f$ over $X''$ will still dominate that over the original $X$. Our $X''$ has all the slicing structures as in the output of Corollary~\ref{Coro:SlicingForCentersOnThePlane}.

    \item As long as we could prove uniform (discretized) weighted mix-norm estimates over $X''$, then by applying it to our given function $f$, we get the desired bound for it. Since the bound is uniform, i.e., the constants are independent of $f$ (though the refinements $X'$ and $X''$ may depend on $f$), we are done with the proof of Case (2) of Theorem~\ref{Thm:MainTheorem}.

    \item The only thing that remains to show is how to get uniform weighted mix-norm estimates over $X''$. Now we perform the duality argument in Section~\ref{sec:discretization} with respect to $X''$, which yields some $X''' \subseteq X''$. Since Katz-Tao properties are inherited downwards, we know that not only is $X'''$ a Katz-Tao $(\delta,\alpha,C)$-set ($C \lessapprox1$), but it also enjoys the slicing structures of $X''$ with the same choice of $\mathbb{T}_i$.

    \item By taking $X'''$ as the new $X$, we can safely continue the proof. Therefore, in later arguments, we can pretend that Corollary~\ref{Coro:SlicingForCentersOnThePlane} is directly applied to the $X$ under examination.
\end{enumerate}

After an appropriate rotation, we can assume all $T_k$'s are parallel to $x$-axis. In this case, $C^*$'s are upper arcs. By dyadic pigeonholing, we assume each slice contains an equal number of centers from $X$. In the computations below, we may frequently use the following facts. 
\begin{itemize}
    \item[(i)] Assume the total number of the center is $N$, then $N\approx \#X$. And $m\leq \delta^{-1}$
    \item[(ii)] We also assume that each slice has $K$ points then $K\lessapprox \delta^{-(\alpha-1)}$.
    \item[(iii)] $k=1,...,N/K\leq \min\{\delta^{-1},N\}$, which implies $K\geq \delta N$ if $N \geq \delta^{-1}$; $K\geq 1$ if $N\leq \delta^{-1}$.
\end{itemize}
Now we are ready to estimate the $L^2$ norm. For circles on each slices, we use item 2 and item 3 in Lemma \ref{Lem:GeometricLemma}; whereas for the interaction of different slices, we simply use triangle inequality. By the triangle inequality, we have
\begin{equation*}
    \Big\|\sum_{x_i\in X}\sum_{r_j\in R_{x_i}}\chi_{C^*_{\delta}(x_i,r_j)}\Big\|_{L^{2}(\R^n)}\leq\frac{N}{K} \max_{k}\Big\|\sum_{x_i\in T_k}\sum_{r_j\in R_{x_i}}\chi_{C^*_{\delta}(x_i,r_j)}\Big\|_{L^{2}(\R^n)}.
\end{equation*}
Expand the $L^2$ norm
\begin{equation}\label{Eq:ExpandL2Norm}
    \frac{N}{K}\Big(\sum_{i_1,i_2}\sum_{j_1\in R_{x_{i_1}}}\sum_{j_2\in R_{x_{i_2}}}|C_\delta^*(x_{i_1},r_{j_1})\medcap C^*_\delta(x_{i_2},r_{j_2})|\Big)^{1/2}.
\end{equation}
Since $C^*$ is the upper arc, if two arcs (or their $\delta$-neighborhoods) whose centers have the same $y$-axis, intersect, then by item (3) of Lemma \ref{Lem:GeometricLemma}, $\sqrt{(\Delta+\delta)/(d+\delta)}\sim 1$. This upgrades item (2) of Lemma \ref{Lem:GeometricLemma} into 
\begin{equation*}
    |C^*_\delta(x_1,y;r)\medcap C_\delta^*(x_2,y,r)|\lesssim \frac{\delta^2}{|x_1-x_2|+\delta}.
\end{equation*}
 Plugging this improved estimate to \eqref{Eq:ExpandL2Norm} and introduce dyadic summation, we get
 \begin{equation*}
     \eqref{Eq:ExpandL2Norm}\lesssim \frac{N}{K}\Bigg(Km\sum_{\theta}\min\Big\{\Big(\frac{\theta}{\delta}\Big)^{\alpha-1},K\Big\}\min\Big\{\frac{\theta}{\delta},m\Big\}\frac{\delta^{2}}{\theta}\Bigg)^{1/2},
 \end{equation*}
 where $\theta$ is dyadic numbers between $\delta$ to 1. Here we first fix the center $i_1$ and radius $j_1$. This gives the term $Km$ where $K$ is the number of centers on each slice and $m$ the number of radii for the same center. Then we consider the geometric relation of other circles with this fixed circle. We dyadically decompose the distances of other centers with the fixed circle: $|x_1-x_2|\sim \theta$. This gives the summation term. Finally, for each center, the number of radii is bounded by $\theta/\delta$ and $m$. This gives the last term.

 Next, we need to discuss different cases of the relation between $K,\theta/\delta,m$ for the ``$\min$''. We use $(*)$ to represent the part in the largest parenthesis above. 
 \begin{itemize}
     \item Suppose that $K^{\frac{1}{\alpha-1}} \le m$, i.e. $K \le m^{\alpha-1}$.

     If $\left(\frac{\theta}{\delta}\right)^{\alpha-1} \le K$ (thus $\frac{\theta}{\delta} \le m$), then $(*)\lesssim Km \sum_{\theta \lesssim \delta K^{\frac{1}{\alpha-1}}} \left(\frac{\theta}{\delta}\right)^{\alpha-1}\delta \lesssim Km \cdot K\delta$.

     If $\left(\frac{\theta}{\delta}\right)^{\alpha-1} \ge K$ and $\frac{\theta}{\delta} \le m$, then $(*)\lesssim Km \sum_{\delta K^{\frac{1}{d-1}} \le \theta \le \delta m} K \cdot \frac{\theta}{\delta} \cdot \frac{\delta^2}{\theta} \lesssim Km \cdot K\delta$.

     It $\frac{\theta}{\delta} \ge m$ (thus $\left(\frac{\theta}{\delta}\right)^{\alpha-1} \ge K$), then $(*)\lesssim Km \sum_{\theta \text{ dyadic},\ \theta \ge m\delta} K \cdot m \cdot \frac{\delta^2}{\theta} \lesssim Km \cdot K\delta$.

 Hence in this case, we have 
 \begin{equation*}
     \norm{\sum_{i=1}^N\sum_{r_j\in R_{x_i}}\chi_{C^*_\delta(x_i,r_j)}}_{L^2(\R^2)}\lessapprox \frac{N}{K}(KmK\delta) \lessapprox Nm^{\frac{1}{2}}\delta^{\frac{1}{2}}.
 \end{equation*}
 \item Suppose $K^{\frac{1}{\alpha-1}} \ge m$, i.e. $K \ge m^{\alpha-1}$.

 If $\frac{\theta}{\delta} \le m$ (thus $\left(\frac{\theta}{\delta}\right)^{\alpha-1} \le K$), then $(*)\lesssim Km \sum_{\theta \le \delta m} \left(\frac{\theta}{\delta}\right)^{\alpha-1} \cdot \delta \le K \cdot m^\alpha \cdot \delta$.

 If $\frac{\theta}{\delta} \ge m$ and $\left(\frac{\theta}{\delta}\right)^{\alpha-1} \le K$, then $(*)\lesssim Km \sum_{\delta m \le \theta \le \delta K^{\frac{1}{\alpha-1}}} \left(\frac{\theta}{\delta}\right)^{\alpha-1} \cdot m \cdot \frac{\delta^2}{\theta} \le K \cdot m^\alpha \cdot \delta$.

 If $\left(\frac{\theta}{\delta}\right)^{\alpha-1} \ge K$ (thus $\frac{\theta}{\delta} \ge m$), then 
 \begin{equation*}
     (*)\lesssim Km \sum_{\theta \ge \delta K^{\frac{1}{\alpha-1}}} Km \cdot \frac{\delta^2}{\theta} \le K^2 m^2 \cdot K^{-\frac{1}{\alpha-1}} \cdot \delta \le K m^\alpha \delta.
 \end{equation*}

 Hence in this case, we have 
 \begin{equation*}
     \norm{\sum_{i=1}^N\sum_{r_j\in R_{x_i}}\chi_{C^*_\delta(x_i,r_j)}}_{L^2(\R^2)}\lessapprox \frac{N}{K}(Km^\alpha\delta)^{1/2}\lessapprox \frac{N}{K^{1/2}}m^{\alpha/2}\delta^{1/2}
 \end{equation*}
 \end{itemize}
We summarize the above computations as follows, 
\begin{equation}\label{Eq:AfterCordoba}
\norm{\sum_{i=1}^N\sum_{r_j\in R_{x_i}}\chi_{C^*_\delta(x_i,r_j)}}_{L^2(\R^2)}
\lesssim
\begin{cases}
N \cdot m^{\frac{1}{2}} \cdot \delta^{\frac{1}{2}}, & \text{if } K \le m^{\alpha-1},\\
\frac{N}{K^{\frac{1}{2}}} \cdot m^{\frac{\alpha}{2}} \cdot \delta^{\frac{1}{2}}, & \text{if } K \ge m^{\alpha-1}.
\end{cases}
\end{equation}
Therefore, our goal comes down to finding the Lebesgue exponents $q,s$ such that 
\begin{equation}\label{Eq:GoalAfterL2Argument}
    \text{RHS of }\eqref{Eq:AfterCordoba}\lessapprox (m\delta)^{\frac{1}{s'}}\delta^{-\frac{\alpha}{2}}(\#X)^{\frac{1}{2}}
\end{equation}
 for all $K,m,N$. Now we start to discuss the different cases of $q$ and $s$.
\begin{enumerate}
    \item If $\delta^{-\alpha} \ge N \geq  \delta^{-1}, N\delta \le k \le m^{\alpha-1},$ \eqref{Eq:GoalAfterL2Argument} is reduced to
    \begin{equation}\label{Eq:GoalAfterL2Argument2}
        N \cdot m^{\frac12} \cdot \delta^{\frac12}
\le
(m\delta)^{\frac1{s'}} \cdot \delta^{-\frac{\alpha}{q}} \cdot N^{\frac{1}{q'}}\Longleftrightarrow
(m\delta)^{\frac12-\frac1{s'}} \cdot N^{\frac{1}{q}}
\le
\delta^{-\frac{\alpha}{q}}.
    \end{equation}

    When $s= 2$, \eqref{Eq:GoalAfterL2Argument2} holds for all $q\in[1,\infty]$ and $ \alpha\in(1,2].$

    When $s>2$ $(\Rightarrow \tfrac12-\tfrac1{s'}<0),$ after computations, \eqref{Eq:GoalAfterL2Argument2} holds if and only if $\frac{\alpha-1}{q}+\frac1s\geq\frac12.$

    When $s<2$ $(\Rightarrow \frac{1}{2}-\frac{1}{s'}>0)$, \eqref{Eq:GoalAfterL2Argument2} holds for all $q\in [1,\infty]$ and $\alpha\in (1,2]$. 
    \item If $1 \le N \le \delta^{-1}, 1 \le K \le m^{\alpha-1}$, \eqref{Eq:GoalAfterL2Argument} is reduced to   
        \begin{equation}\label{Eq:GoalAfterL2Argument3}
            N\,m^{\frac12}\,\delta^{\frac12} \le (m\delta)^{\frac1{s'}}\,\delta^{-\frac{\alpha}{q}}\,N^{\frac1{q'}} \Longleftrightarrow
(m\delta)^{\frac12-\frac1{s'}}\,N^{\frac1q} \le \delta^{-\frac{\alpha}{q}}
        \end{equation}
        
    When $s=2$, \eqref{Eq:GoalAfterL2Argument3} holds for all $q\in[1,\infty]$ and $ \alpha\in(1,2].$

    When $s>2$ $(\Rightarrow \frac12-\frac1{s'}<0)$, after computations, \eqref{Eq:GoalAfterL2Argument3} holds if and only if $\frac{\alpha-1}{q}+\frac{1}{s}\geq \frac{1}{2}$.

    When $s<2$ $(\Rightarrow \frac{1}{2}-\frac{1}{s'}>0)$, \eqref{Eq:GoalAfterL2Argument3} holds for all $q\in [1,\infty]$ and $\alpha\in (1,2]$. 

    \item $\delta^{-\alpha} \ge N \ge \delta^{-1},  K \ge \max\{N\delta,\; m^{\alpha-1}\}$, \eqref{Eq:GoalAfterL2Argument} is reduced to 
    \begin{equation*}\label{Eq:GoalAfterL2Argument4}
        \frac{N}{K^{1/2}} \cdot m^{\alpha/2} \cdot \delta^{1/2}
\le
(m\delta)^{1/s'} \cdot \delta^{-\alpha/q} \cdot N^{1/q'}
    \end{equation*}
    \begin{enumerate}
        \item[(3.1)] $N\delta \ge m^{\alpha-1}$, plugging $K\geq N\delta$, after computations, we need 
        \begin{equation*}
  m^{\alpha/2-1/s'}
\lesssim
\delta^{\,1/s'-\alpha/q}\cdot N^{1/q'-1/2}
        \end{equation*}

        If $q\geq 2$ and $\frac{\alpha-1}{q}+\frac{1}{s}\geq \frac{1}{2}$ then this holds trivially.

        If $q\geq 2$ and $\frac{\alpha-1}{q}+\frac{1}{s}<\frac{1}{2}$, then this never holds.

        If $q\leq 2$ and $\frac{\alpha}{2}-\frac{1}{s'}\geq 0$ $(s\geq \frac{2}{2-\alpha})$, then this holds trivially.

         If $q\leq 2$ and $\frac{\alpha}{2}-\frac{1}{s'}< 0$ $(s\geq \frac{2}{2-\alpha})$, then this never holds.

         \item[(3.2)] $N\delta \le m^{\alpha-1}$, plugging $K\geq m^{\alpha-1}$ reduces to the discussion of item (1).
    \end{enumerate}
\item If $1 \le N \le \delta^{-1}, K \ge m^{\alpha-1}$, then this is reduced to item (2).
\end{enumerate}

To summarize, we need the following conditions to simultaneously hold:
\begin{enumerate}
    \item [(I)] $s\in[1,2], q\in[1,\infty]$ or $s\in[2,\infty], \frac{\alpha-1}{q}+\frac{1}{s}\ge \frac{1}{2}$

    \item [(II)] $q\in[1,2], s\in\left[1,\frac{2}{2-\alpha}\right]$ or $q\in[2,\infty], \frac{\alpha-1}{q}+\frac{1}{s}\ge \frac{1}{2}$.
\end{enumerate}
This is exactly the range we obtained in the item (2) of Proposition \ref{Prop:MainThmDiscretizedEndPointCase}. Note that the endpoint corresponds to $q=2$, which matches our $p=2$.

\begin{remark}
    It seems to us that the initial application of the triangle inequality accounts for the main loss of the entire arguments. For example, if $\alpha$ is close to $1$, then we should only expect a very low dimensional fractal set on each slice, and the triangle inequality may well be far from sharp. This somehow explains why Case (2) in Theorem~\ref{Thm:MainTheorem_0} only improves the previous results when $\alpha$ is not too small. However, we currently do not know how to reduce this loss. Naive ways such as choosing a suitable thickened scale to apply the triangle inequality do not seem to help.
\end{remark}

\section{High frequency decay and the proof of Theorem~\ref{Thm:MainTheorem_0}}\label{sec:high_freq_decay}

We have proved Theorem~\ref{Thm:MainTheorem} in the previous sections. In this section, we will prove Theorem~\ref{Thm:MainTheorem_0} from Theorem~\ref{Thm:MainTheorem}. The main difference between these two theorems is that Theorem~\ref{Thm:MainTheorem} involves averages over $\delta$-annuli and with an additional $\delta^{-\eps}$-loss, while Theorem~\ref{Thm:MainTheorem_0} involves genuine circular averages with no loss of scale. For application to various exceptional set estimates, it is important to get the scale-free estimate (\ref{Eq:MixedNormEstimate_0}) in Theorem~\ref{Thm:MainTheorem_0}.

Via H\"older's inequality in the radial direction of the $\delta$-spherical shell, it is easy to see that (\ref{Eq:MixedNormEstimate_0}) implies (\ref{Eq:MixedNormEstimate}). However, the reverse implication is more subtle and requires the so-called ``high frequency decay'' estimates. Such type of estimates are established via local smoothing estimates, and are used to remove the $\delta^{-\eps}$-loss from purely geometric arguments via interpolation. One may consult \cite[Section~1.3]{Zahl25} for references and backgrounds in the setting of maximal functions along planar curves, or \cite[Section~3.2]{HKL22} for formulations with respect to fractal weights. 

The framework works in general $\R^n$. Let us first set-up a few basic facts. Let $\psi$ be a smooth radial function on $\R^n$ such that $\spt\,\widehat{\psi} \subseteq \{\xi: 1/2 \leq |\xi| \leq 2\}$, $0\leq \widehat{\psi} \leq 1$, and $\sum_{j\in\Z}\widehat{\psi_j}(\xi)=1$ for any $\xi\neq 0$, where $\widehat{\psi_j}(\xi) \coloneqq \widehat{\psi}(2^{-j}\xi)$. We also let $\widehat{\varphi_0} \coloneqq 1 - \sum_{j\geq 1}\widehat{\psi_j}$. Let $\dd \sigma_r$ be the normalized surface measure on $S(0,r)$, then $\AA^n f(x,r) = |f|*\dd\sigma_r(x)$. For $t\in[1,2]$, an elementary computation shows that $\psi_j * \sigma_t$ is $O(2^j)$ and decays rapidly away from the $O(2^{-j})$ neighborhood of $S(0,r)$. To be specific, for all $j\geq 1$ and $N > 0$, we have
\begin{align}
    |\psi_j| * \sigma_r (x) &\lesssim_N 2^j (1+2^j||x|-r|)^{-N},\label{Eq:basic_1}\\
    |\varphi_0| * \sigma_r (x) &\lesssim_N (1+|x|)^{-N}.\label{Eq:basic_2}
\end{align}

\begin{proposition}\label{Prop:smooth_freq_cut}
    Suppose that for all $\eps>0$, there is a constant $C_\eps>0$ such that \begin{equation*}
        \norm{\AA_\delta^n f}_{L
        _x^q(B,\nu)(L^s_r(I))}\leq C_\eps\delta^{-\eps}\langle\nu\rangle_\alpha^{\frac{1}{q}}\norm{f}_{L^p(\R^n)},
    \end{equation*}
    holds for all $0<\delta<1.$ Then for all $\eps'>0$, there is a constant $C_{\eps'}'>0$ such that 
    \begin{equation*}
        \norm{\AA^n (f*\psi_j)}_{L
        _x^q(B,\nu)(L^s_r(I))}\leq C_{\eps'}'2^{\eps'j}\langle\nu\rangle_\alpha^{\frac{1}{q}}\norm{f}_{L^p(\R^n)},
    \end{equation*}
    holds for all $j\geq 1$ and $f$ with $\spt\,f\subseteq B(0,5)$.
\end{proposition}
\begin{proof}
    Note that by (\ref{Eq:basic_1}), for any $0<\eta<1/10$ and $j \geq 1$, we have
    \begin{align*}
        \AA^n(f*\psi_j)(x,r) &= |f*\psi_j|*\dd\sigma_r(x) 
        \leq         |f|*(|\psi_j|*\dd\sigma_r)(x)\\
        &\lesssim_N \int_{\R^n} |f(x-y)|\cdot 2^j (1+2^j||y|-r|)^{-N} \dd y\\
        &\leq 
        2^j\int_{||y|-r|\leq 2^{-j+\eta j}} |f(x-y)|\dd y + 2^{(1-\eta N) j}\int_{||y|-r|\geq 2^{-j+\eta j}} |f(x-y)|\dd y\\
        &\leq
        2^{\eta j} \AA_{2^{-j+\eta j}}^n f(x,r) + 2^{(1-\eta N)j} \norm{f}_{L^1(\R^n)}
    \end{align*}
    for any $x\in\R^n$ and $r\in I$. Therefore, 
    \begin{align*}
        \norm{\AA^n (f*\psi_j)}_{L
        _x^q(B,\nu)(L^s_r(I))}
        &\lesssim_N 
        \norm{\AA_{2^{-j+\eta j}}^n f}_{L
        _x^q(B,\nu)(L^s_r(I))} + 2^{(1-\eta N)j} \norm{\nu\chi_{B}}^{\frac{1}{q}} \norm{f}_{L^1(\R^n)}\\
        &\leq C_\eps 2^{\eta j + \eps j - \eta\eps j}\langle\nu\rangle_\alpha^{\frac{1}{q}}\norm{f}_{L^p(\R^n)} + 10^n 2^{(1-\eta N)j} \langle\nu\rangle_\alpha^{\frac{1}{q}} \norm{f}_{L^p(\R^n)},
    \end{align*}
    where we applied the assumed estimate for $\AA_\delta^n f$ with $\delta = 2^{-j+\eta j}$, and used $\norm{\nu\chi_{B}} \leq \langle\nu\rangle_\alpha$ as well as $\norm{f}_{L^1(\R^n)} \leq 6^n \norm{f}_{L^p(\R^n)}$ due to $\spt\, f \subseteq B(0,5)$. Finally, for any given $\eps'>0$, by taking $\eta = \eps = \eps'/2$ and $N > 1/\eta$, we get the desired bound.
\end{proof}
Now by interpolation, we can conclude that
\begin{proposition}\label{Prop:interpolate}
    Suppose that for all $\eps>0$, there is a constant $C_\eps>0$ such that \begin{equation*}
        \norm{\AA_\delta^n f}_{L
        _x^{q_1}(B,\nu)(L^{s_1}_r(I))}\leq C_\eps\delta^{-\eps}\langle\nu\rangle_\alpha^{\frac{1}{q_1}}\norm{f}_{L^{p_1}(\R^n)},
    \end{equation*}
    holds for all $0<\delta<1.$ Suppose also that for some $\eps_0>0$, the inequality
    \begin{equation}\label{Eq:High_freq_decay}
        \norm{\AA^n (f*\psi_j)}_{L
        _x^{q_2}(B,\nu)(L^{s_2}_r(I))}\leq C_{\eps_0}2^{-\eps_0 j}\langle\nu\rangle_\alpha^{\frac{1}{q_2}}\norm{f}_{L^{p_2}(\R^n)}
    \end{equation}
    holds for all $j \geq 1$. Then for any $(\frac{1}{p}, \frac{1}{q}, \frac{1}{s}) = (1-\theta)(\frac{1}{p_1}, \frac{1}{q_1}, \frac{1}{s_1}) + \theta(\frac{1}{p_2}, \frac{1}{q_2}, \frac{1}{s_2})$ with $\theta \in (0,1]$, there is a constant $C>0$ depending only on $\{C_\eps\}_\eps$, $C_{\eps_0}$, $n$, $\eps_0$, and $\theta$, such that
    \begin{equation*}
        \norm{\AA^n f}_{L
        _x^{q}(B,\nu)(L^{s}_r(I))}\leq C\langle\nu\rangle_\alpha^{\frac{1}{q}}\norm{f}_{L^{p}(\R^n)}.
    \end{equation*}
\end{proposition}
\begin{proof}
    By Proposition~\ref{Prop:smooth_freq_cut}, we know that for all $\eps'>0$, there is a constant $C_{\eps'}'>0$ depending only on $\{C_\eps\}_\eps$, $n$, and $\eps'$, such that 
    \begin{equation}\label{Eq:Geom_loss}
        \norm{\AA^n (f*\psi_j)}_{L
        _x^{q_1}(B,\nu)(L^{s_1}_r(I))}\leq C_{\eps'}'2^{\eps'j}\langle\nu\rangle_\alpha^{\frac{1}{q_1}}\norm{f}_{L^{p_1}(\R^n)},
    \end{equation}
    holds for all $j\geq 1$ and $f$ with $\spt\,f\subseteq B(0,5)$. By complex interpolation in the mixed-norm setting (\cite[Section~7, Theorem~2]{BP61}\footnote{See \cite[Section~4.1~and~5.1]{BL76} or \cite[Section~2.2]{HNVW16} for expositions in more general vector-valued settings.}), from (\ref{Eq:High_freq_decay}) and (\ref{Eq:Geom_loss}) we immediately obtain
    \begin{align*}
        \norm{\AA^n (f*\psi_j)}_{L
        _x^{q}(B,\nu)(L^{s}_r(I))}
        &\leq \big(C_{\eps'}'2^{\eps' j}\langle\nu\rangle_\alpha^{\frac{1}{q_1}}\big)^{1-\theta} \big(C_{\eps_0}2^{-\eps_0 j} \langle\nu\rangle_\alpha^{\frac{1}{q_2}} \big)^{\theta} \norm{f}_{L^{p}(\R^n)}\\
        &= C_{\eps'}'^{1-\theta}C_{\eps_0}^\theta 2^{(\eps'(1-\theta) - \eps_0\theta)j} \langle\nu\rangle_\alpha^{\frac{1}{q}} \norm{f}_{L^{p}(\R^n)}
    \end{align*}
    for any $\eps'>0$ and $f$ with $\spt\,f\subseteq B(0,5)$. Here for the interpolation, we are considering the linear operator $Tf \coloneqq \AA^n (f*\psi_j)$ restricted to those $f$ with $\spt\,f \subseteq B(0,5)$. 
    
    Now by taking $\eps' = \eps_0\theta/2$ and summing over all $j \geq 1$, we get the high frequency bound
    \begin{align*}
        \norm{\AA^n (f-f*\varphi_0)}_{L
        _x^{q}(B,\nu)(L^{s}_r(I))} 
        = 
        \norm{\AA^n (f*\sum_{j\geq1}\psi_j)}_{L
        _x^{q}(B,\nu)(L^{s}_r(I))}
        \leq
        C \langle\nu\rangle_\alpha^{\frac{1}{q}} \norm{f}_{L^{p}(\R^n)}
    \end{align*}
    for all $f$ with $\spt\,f\subseteq B(0,5)$, where $C>0$ depends only on $\{C_\eps\}_\eps$, $C_{\eps_0}$, $n$, $\eps_0$, and $\theta$. On the other hand, by H\"older's inequality and (\ref{Eq:basic_2}) with $N = n+1$, we have
    \begin{align*}
        \AA^n(f * \varphi_0)(x,r) \leq |f| * (|\varphi_0| * \dd\sigma_r)(x) \leq \norm{f}_{L^{p}(\R^n)} \norm{|\varphi_0| * \dd\sigma_r}_{L^{p'}(\R^n)} \lesssim_n \norm{f}_{L^{p}(\R^n)}
    \end{align*}
    for any $x\in\R^n$ and $r\in I$. So we get the low frequency bound
    \begin{align*}
        \norm{\AA^n (f*\varphi_0)}_{L
        _x^{q}(B,\nu)(L^{s}_r(I))} \lesssim_n \norm{\nu\chi_B}^{\frac{1}{q}} \norm{f}_{L^{p}(\R^n)} \leq \langle\nu\rangle_\alpha^{\frac{1}{q}} \norm{f}_{L^{p}(\R^n)}. 
    \end{align*}
    Combining the two bounds, we conclude that
    \begin{equation*}
        \norm{\AA^n f}_{L
        _x^{q}(B,\nu)(L^{s}_r(I))}\leq C\langle\nu\rangle_\alpha^{\frac{1}{q}} \norm{f}_{L^{p}(\R^n)}
    \end{equation*}
    for some slightly modified $C>0$ and all $f$ with $\spt\,f \subseteq B(0,5)$. 
    
    Finally, note that $\norm{\AA^n f}_{L_x^{q}(B,\nu)(L^{s}_r(I))} = \norm{\AA^n (f\chi_{B(0,5)})}_{L_x^{q}(B,\nu)(L^{s}_r(I))}$ by the locality of $\AA^n$, so the proof is completed.
\end{proof}

Inequalities of the form (\ref{Eq:High_freq_decay}) are the ``high frequency decay'' estimates mentioned before. For the proof of Theorem~\ref{Thm:MainTheorem_0}, we will only need the following result in $\R^2$:
\begin{proposition}[\rm \cite{CDK25}, Proposition~4.12]\label{Prop:CDK_decay}
    Let $\overline{\nu}$ be a non-zero Borel measure on $\R^2 \times I$. If $1 < \alpha \leq 3$, then for any $p>3$, there exists $\eps = \eps(\alpha,p)>0$ such that
    \begin{align*}
        \norm{\AA(f * \psi_j)}_{L^p(\overline{\nu})} \lesssim 2^{-\eps j} \langle\overline{\nu}\rangle_\alpha^{\frac{1}{q}} \norm{f}_{L^{p}(\R^2)}.
    \end{align*}
    If $2 < \alpha \leq 3$, then for $\eps = (\alpha - 2)/2 >0$, we have
    \begin{align*}
        \norm{\AA(f * \psi_j)}_{L^2(\overline{\nu})} \lesssim 2^{-\eps j} \langle\overline{\nu}\rangle_\alpha^{\frac{1}{2}} \norm{f}_{L^{2}(\R^2)}.
    \end{align*}
\end{proposition}
\begin{proof}[Proof of Theorem~\ref{Thm:MainTheorem_0}]
    By taking $\overline{\nu} = \nu \otimes \mathcal{L}^1|_I$ and replacing $\alpha$ with $\alpha+1$ in Proposition~\ref{Prop:CDK_decay}, we get     \begin{equation*}
        \norm{\AA (f*\psi_j)}_{L
        _x^{p}(B,\nu)(L^{p}_r(I))}\lesssim 2^{-\eps j}\langle\nu\rangle_\alpha^{\frac{1}{p}}\norm{f}_{L^{p}(\R^2)}
    \end{equation*}
    for some $\eps = \eps(\alpha,p)>0$, whenever $0 < \alpha \leq 2$ and $p>3$, or $1 < \alpha \leq 2$ and $p=2$. Here we are using the fact that $\langle\overline{\nu}\rangle_{\alpha+1} \lesssim \langle\nu\rangle_\alpha$.

    Let us first prove Case (2). By Proposition~\ref{Prop:interpolate}, we can interpolate the high frequency decay estimate with Case (2) in Theorem~\ref{Thm:MainTheorem} to obtain (\ref{Eq:MixedNormEstimate_0}) for any $(\frac{1}{p}, \frac{1}{q}, \frac{1}{s}) = (1-\theta)(\frac{1}{2}, \frac{1}{2}, \frac{2-\alpha}{2}) + \theta(\frac{1}{2}, \frac{1}{2}, \frac{1}{2})$, where $\theta \in (0,1]$. In particular, by taking $\theta$ arbitrarily small, we immediately see that (\ref{Eq:MixedNormEstimate_0}) holds for $p=q=2$ and $1 \leq s < \frac{2}{2-\alpha}$.
    
    Now we turn to Case (1), which is a bit subtle as the high frequency decay estimate only holds for $p>3$ instead of $p=3$. Fortunately, this endpoint issue can be handled by an additional interpolation step. As before, by Proposition~\ref{Prop:interpolate} and Case (1) in Theorem~\ref{Thm:MainTheorem}, we obtain (\ref{Eq:MixedNormEstimate_0}) for any $(\frac{1}{p}, \frac{1}{q}, \frac{1}{s}) = (1-\theta)(\frac{1}{3}, \frac{1}{3}, \frac{1-\alpha}{3}) + \theta(\frac{1}{4}, \frac{1}{4}, \frac{1}{4}) = (\frac{1}{3}-\frac{\theta}{12}, \frac{1}{3}-\frac{\theta}{12}, \frac{1-\alpha}{3} + \frac{4\alpha-1}{12}\theta)$, where $\theta \in (0,1]$. On the other hand, by Fubini's theorem, one can easily see that (\ref{Eq:MixedNormEstimate_0}) trivially holds for $(\frac{1}{p}, \frac{1}{q}, \frac{1}{s}) = (1,0,1)$. Therefore, by a further complex interpolation in the mixed-normed setting with weight $u = \theta/(8+\theta)$, we get (\ref{Eq:MixedNormEstimate_0}) for any
    \begin{align*}
        \Big(\frac{1}{p}, \frac{1}{q}, \frac{1}{s}\Big) &= (1-u)\Big(\frac{1}{3}-\frac{\theta}{12}, \frac{1}{3}-\frac{\theta}{12}, \frac{1-\alpha}{3} + \frac{4\alpha-1}{12}\theta\Big) 
        +
        u (1,0,1)\\
        &= \Big( \frac{1}{3}, \frac{8}{8+\theta}\Big(\frac{1}{3} - \frac{\theta}{12}\Big), \frac{8}{8+\theta}\Big(\frac{1-\alpha}{3} + \frac{4\alpha-1}{12}\theta\Big) + \frac{\theta}{8+\theta}\Big).
    \end{align*}
    By Fact~\ref{Fact:Holder}, since $\frac{8}{8+\theta}(\frac{1}{3} - \frac{\theta}{12}) < \frac{1}{3}$, we must have (\ref{Eq:MixedNormEstimate_0}) for $p=q=3$ with the same $s$. Again by Fact~\ref{Fact:Holder}, since $\frac{1}{s} \rightarrow \frac{1-\alpha}{3}$ as $\theta \rightarrow 0$, we must have (\ref{Eq:MixedNormEstimate_0}) for $p=q=3$ and $1\leq s < \frac{3}{1-\alpha}$, as desired.
\end{proof}

Before ending this section, we point out that because of the existence of measure zero Besicovitch–
Rado–Kinney sets \cite{BR68,Kinney68}, when $0 < \alpha \leq 1$, one cannot expect the space-time weighted estimates of the form 
\begin{align*}
    \norm{\AA f}_{L^q(\overline{\nu})} \lesssim \langle\overline{\nu}\rangle_\alpha^{\frac{1}{q}} \norm{f}_{L^{p}(\R^2)}
\end{align*}
to hold, and any estimate for $\norm{\AA_\delta f}_{L^q(\overline{\nu})}$ must involve some $\delta^{-\eps}$-loss when $p<\infty$. This indicates that there is no high frequency decay when $0 < \alpha \leq 1$, and is one of the reasons why Ham-Ko-Lee \cite{HKL22} only focus on $1 < \alpha \leq 3$. In contrast, our Theorem~\ref{Thm:MainTheorem_0} indeed contains estimates for low-dimensional $\nu$ with no loss of scale when $p=q=3$. This is because by giving up a little $L^s$ integrability in $r$, we are able to take advantage of the high frequency decay estimates of $\overline{\nu} = \nu \otimes \mathcal{L}^1|_I$, whose dimension is increased by $1$. In other words, in the regime of mixed-norm estimates, high frequency decay is always no big issue if we are willing to abandon the endpoints.
\section{Exceptional set estimates for the wave equation}\label{sec:wave_eqn}
\subsection{Preliminaries}\label{subsec:prelim} When proving exceptional set estimates for partial differential equations, it is often the case that one needs to be careful about whether or not it makes sense to consider the behavior over a lower-dimensional set. In this subsection, we will set up some basic estimates to argue that our target is properly formulated. More precisely, we will show that the solution of the wave equation on $\R^2$ is continuous whenever $h$ has local integrability slightly better than $L^2$. The discussions are expected to be standard, but since we could not find a suitable reference in the literature, we present everything in detail for the reader's convenience.

Recall that in Theorem~\ref{Thm:wave_eqn}, we are considering the following linear wave equation:
\[
\begin{cases}
u_{tt}-\Delta_x u=0, \qquad x\in \mathbb{R}^2,\ t\in (0,\infty),\\
u(x,0)=0,\\
u_t(x,0)=h(x).
\end{cases}
\]
By the well-known Poisson's formula (see \cite[Section~2.4]{Evans10}), the (formal) solution is given by 
\begin{align}\label{Eq:Poisson}
    u(x,t) = \frac{1}{2\pi}\int_{B(x,t)} \frac{h(y)}{(t^2 - |y-x|^2)^{1/2}} \dd y.
\end{align}
We can rewrite (\ref{Eq:Poisson}) formula in terms of $\AA$ in the polar coordinate system:
\begin{align*}
    u(x,t) = \int_0^t \AA h(x,r) \frac{r}{\sqrt{t^2-r^2}} \dd r,
\end{align*}
Let $F_x(r) \coloneqq \AA h(x,r)r$ and $\J$ is an Abel-type transform (see \cite{GV91} for an exposition) defined by $$\J F(t) \coloneqq \int_0^t \frac{F(r)}{\sqrt{t^2-r^2}} \dd r,$$ then we can further rewrite the formal solution (\ref{Eq:Poisson}) as 
\begin{align*}
    u(x,t) = \J F_x (t).
\end{align*}

Intuitively, since the singularity of the kernel resembles $|\cdot|^{-1/2}$ (Riesz potential of order $\frac{1}{2}$), if $F_x \in L^s[0,t]$ for some $s>2$ and $t>0$, then we should expect $u(x,\cdot) = \J F_x(\cdot) \in C^{\frac{1}{2}-\frac{1}{s}}[0,t]$ by the Sobolev embedding theorem:
\begin{equation*}
    F\in L^s \xrightarrow[\text{via ``Riesz potential''}]{\text{``gain'' $\frac{1}{2}$ derivative}} \mathcal JF\in L^{s}_{1/2}\hookrightarrow C^{1/2-1/s}.
\end{equation*}
Such a smoothing property for Abel-type transforms has been rigorously established (see \cite[Theorem~4.1.7]{GV91}), and here we only record a special case for our purpose:
\begin{proposition}\label{Prop:Abel_transform}
   For $F \in L^s[0,t]$ with $s > 2$, we have $\J F \in C^{\frac{1}{2}-\frac{1}{s}}[0,t]$. In fact, for any $0 < t \leq T<\infty$ we have
    \begin{align*}
    \norm{\J F}_{C^{\frac{1}{2}-\frac{1}{s}}[0,t]} \leq C_{s,T} \norm{F}_{L^s[0,t]},
    \end{align*}
    where $C_{s,T}$ is a finite constant depending only on $s$ and $T$. 
\end{proposition}

Let us first assume that $h \in L^p(\R^2)$ for some $p>2$. The general case where $h \in L_{\rm loc}^{2+}(\R^2)$ then follows immediately. Note that the integral in (\ref{Eq:Poisson})  converges absolutely by H\"older's inequality, so the solution $u$ is pointwise well-defined. Also, for each $x\in\R^2$, we have the trivial estimate
\begin{align*}
    \int_0^t F_x(r)^p \dd r = \int_0^t \AA h(x,r)^p r^p \dd r \leq t^{p-1} \int_0^t \AA |h|^p(x,r) r \dd r = \frac{1}{2\pi} t^{p-1} \int_{B_{x,t}} |h|^p \leq t^{p-1}\norm{h}_{L^p(\R^2)}^p.
\end{align*}
Thus by applying Proposition~\ref{Prop:Abel_transform} with $F = F_x(r)$, $s=p$, and $t=T$, we know that 
\begin{align*}
    \norm{u(x,\cdot)}_{C^{\frac{1}{2}-\frac{1}{p}}[0,T]} \leq C_{p,T} \norm{h}_{L^p(\R^2)}
\end{align*}
uniformly in $x$, where $C_{p,T}$ is some finite constant depending only on $p$ and $T$. This process says that improving the Lebesgue regularity of the initial velocity $h$ gives better H\"older regularity to the solution $u$.

Firstly, let us check in what sense does (\ref{Eq:Poisson}) satisfies the two initial conditions. Firstly, by H\"older's inequality followed by a change of variables $y \mapsto t\Tilde{y}$, for every $x \in R^2$, we have
\begin{align}\label{Eq:u(x,0)}
     |u(x,t)| &\leq \frac{1}{2\pi} \int_{B(0,t)} \frac{|h|(x+y)}{(t^2 - |y|^2)^{1/2}} \dd y \nonumber\\
     &\leq \frac{1}{2\pi} \left[ \int_{B(0,t)} \frac{1}{(t^2-|y|^2)^{p'/2}} \dd y \right]^{1/p'} \norm{h}_{L^p(\R^2)}\nonumber\\
     &= \frac{1}{2\pi} \left[ t^{2-p'} \int_{B(0,1)} \frac{1}{(1-|\Tilde{y}|^2)^{p'/2}} \dd y \right]^{1/p'} \norm{h}_{L^p(\R^2)}\\
     &\rightarrow 0 \nonumber
\end{align}
as $t \mapsto 0^+$, where we used $p'<2$ in the last step. In particular, by extending the definition of $u(x,t)$ in (\ref{Eq:Poisson}) to $t=0$ by the initial condition $u(x,0)=0$, we know that it is uniformly H\"older continuous in $t\in[0,T]$. Secondly, 
by a change of variables $r \mapsto t\Tilde{r}$, we have 
\begin{align*}
    u(x,t) = t\int_0^1 \AA h(x,t\Tilde{r}) \frac{1}{\sqrt{1-\Tilde{r}^2}}\dd \Tilde{r}
\end{align*}
Combining the $L^p$ ($p>2$) boundedness of Bourgain's circular maximal function \cite{Bourgain86} with arguments from real analysis \cite[Theorem~5.6]{Tao_Lecture}, we know that for almost every $x \in \R^2$, the averages $\AA h (x,\cdot)$ are well-defined and finite for all $r>0$, continuous in $r$ (see also \cite{Marstrand87}), and $\lim_{r\rightarrow 0^+} \AA h (x,r) = h(x)$. Here we do not identify $L^p$ functions if they agree almost everywhere. Therefore, for almost every $x\in \R^2$, we have
\begin{align*}
    u_t(x,0) &= \lim_{t\rightarrow0^+} \frac{t\int_0^1 \AA h(x,t\Tilde{r}) \frac{1}{\sqrt{1-\Tilde{r}^2}}\dd \Tilde{r} - 0}{t} \\ 
    &= \lim_{t\rightarrow0^+} \int_0^1 \AA h(x,t\Tilde{r}) \frac{1}{\sqrt{1-\Tilde{r}^2}}\dd \Tilde{r} - 0  \\
    &= h(x) \int_0^1 \frac{1}{\sqrt{1-\Tilde{r}^2}}\dd \Tilde{r}\\
    &= h(x),
\end{align*}
which coincides with our initial condition $u_t(x,0) = h(x)$. Note that $h\in L^p(\R^2)$ is only well-defined almost everywhere, so this is the best sense that we could hope for.

Secondly, let us check the regularity of $u(x,t)$ in $x$. For any fixed $t>0$, we can write (\ref{Eq:Poisson}) as the convolution of $h \in L^p(\R^2)$ and $(2\pi)^{-1}(t^2 - |\cdot|^2)^{-1/2} \chi_{B(0,t)} \in L^{p'}(\R^2)$, so $u(\cdot,t)$ is a uniformly continuous function on $\R^2$. In fact, via similar computation as in (\ref{Eq:u(x,0)}), one can easily see that the family $\{u(x,t)\}_{t\in[0,T]}$ is uniformly equicontinuous (Recall $u(\cdot,t) \equiv 0$ for $t=0$.). To conclude, on any $\R^2\times [0,T]$ ($0<T<\infty$), the solution $u(x,t)$ is a multivariate function uniformly equicontinuous in $x \in \R^2$ and uniformly H\"older continuous in $t\in [0,T]$, and so must be jointly uniformly continuous for in $(x,t)$. 

Now by a standard approximation procedure, one can easily see that $u(x,t)$ satisfies the linear wave equation $u_{tt} - \Delta_x u = 0$ on $\R^2 \times (0,\infty)$ in the weak sense (distributional sense). Moreover, as we previously showed, the initial data enjoy very nice interpretations. As a side remark, by clasical theory of linear wave equations, our $u(x,t)$ is also the unique continuous function on $\R^2\times[0,\infty)$ that satisfies the equation in the sense of distributions. 

Finally, by the local nature of the truncated kernel in the Poisson formula (\ref{Eq:Poisson}), it is easy to see that even if we only assume $h \in L_{\rm loc}^{2+}(\R^2)$ at the beginning, the same continuity property of solution $u$ (as in the $h \in L^p(\R^2)$ case) would still hold true, except that we only have local uniformity. In this case, (a priori) we can explicitly tell whether or not $u(x,\cdot) \in C_{\rm loc}^\beta[0,\infty)$ for each $x\in\R^2$ and $\beta\in(0,1]$, and so the problem of bounding $\dim_H E(u,\beta) = \dim_H\{x\in\mathbb{R}^2: u(x,\cdot) \not\in C_{\rm loc}^\beta[0,\infty)\}$ is properly formulated.

\subsection{From \texorpdfstring{$L^s_r(I)$}{Lsr(I)} to \texorpdfstring{$L^s_r[0,2]$}{Lsr(0,2]}}\label{subsec:extend_0}
In the previous subsection, we have seen that if $h\in L^{2+}(\R^2)$, then for every $x\in\R^2$, we always have $u(x,\cdot) \in C_{\rm loc}^{\frac{1}{2} - \frac{1}{p}}[0,\infty)$, and so $E(u,\beta) = \varnothing$ for $\beta = \frac{1}{2}-\frac{1}{p}>0$. The question is: Can we raise the exponent $\beta$ upon sacrificing a ``small'' set of $x$ in the sense of $\dim_H$? This should intuitively be achievable, because in Theorem~\ref{Thm:MainTheorem_00} we have estimates of the general form
\begin{equation}\label{Eq:main_0}
    \norm{\AA f}_{L
    _x^q(B,\nu)(L^s_r(I))}\leq C \langle\nu\rangle_\alpha^{\frac{1}{q}}\norm{f}_{L^p(\R^n)},
\end{equation}
and if $s>p$, then we could convert better integrability into better H\"older continuity by Proposition~\ref{Prop:Abel_transform}. The main technical difficulty is that we only have $L^s_r(I)$ in (\ref{Eq:main_0}) instead of $L^s_r[0,2]$ (More general $L^s_r[0,T]$ can be handled similarly.). Also, we need to consider $F_x(r) = \AA h(x,r)r$ instead of $\AA f(x,r)$. In this subsection we will adopt a rescaling argument to deal with these issues, which also works in $\R^n$. For our purpose, we can mainly assume $q \geq p$.

Suppose that (\ref{Eq:main_0}) holds true for any $f$ and $\nu$, then for $\lambda \in (0,1]$ and $h\in L^p(\R^n)$, by change of variable, we have
\begin{align*}
    \norm{\AA h}_{L
    _x^q(B,\nu)(L^s_r(\lambda I))} 
    &= \Big(\int_{B}\Big(\int_{\lambda I}|\AA h(x,r)|^s\mathrm dr\Big)^\frac{q}{s}\mathrm d\nu(x)\Big)^{\frac{1}{q}}\\
    &= \Big(\int_{B}\Big( \lambda \int_{I} |\AA h(x,\lambda\Tilde{r})|^s \dd \Tilde{r} \Big)^\frac{q}{s}\dd\nu(x)\Big)^{\frac{1}{q}}.
\end{align*}
If we define $\Tilde{h}(\cdot) \coloneqq h(\lambda\cdot)$, then 
\begin{equation*}
    \Big(\int_{B}\Big( \lambda \int_{I} |\AA h(x,\lambda\Tilde{r})|^s \dd \Tilde{r} \Big)^\frac{q}{s}\dd\nu(x)\Big)^{\frac{1}{q}} = \Big(\int_{B}\Big( \lambda \int_{I} |\AA \Tilde{h}(x/\lambda,\Tilde{r})|^s \dd \Tilde{r} \Big)^\frac{q}{s}\dd\nu(x)\Big)^{\frac{1}{q}}.
\end{equation*}
Let $\Tilde{\nu}(\cdot) \coloneqq \lambda^{-\alpha} \nu(\lambda \cdot)$ and apply another change of variables, 
\begin{align*}
    \Big(\int_{B}\Big( \lambda \int_{I} |\AA \Tilde{h}(x/\lambda,\Tilde{r})|^s \dd \Tilde{r} \Big)^\frac{q}{s}\dd\nu(x)\Big)^{\frac{1}{q}} 
    &=\Big(\int_{\frac{1}{\lambda} B} \lambda^\alpha \Big( \lambda \int_{I} |\AA \Tilde{h}(x,\Tilde{r})|^s \dd \Tilde{r} \Big)^\frac{q}{s}\dd\Tilde{\nu}(x)\Big)^{\frac{1}{q}}\\
    &= \lambda^{\frac{\alpha}{q}+\frac{1}{s}} \Big(\int_{\frac{1}{\lambda} B} \Big( \int_{I} |\AA \Tilde{h}(x,\Tilde{r})|^s \dd \Tilde{r} \Big)^\frac{q}{s}\dd\Tilde{\nu}(x)\Big)^{\frac{1}{q}}.
\end{align*}

In particular, we have
\begin{align*}
    \langle \Tilde{\nu} \rangle_\alpha 
    = \sup_{x\in\R^n, \rho\in(0,1]} \frac{\Tilde{\nu}(B(x,\rho))}{\rho^\alpha} 
    = \sup_{x\in\R^n, \rho\in(0,1]} \frac{\nu(B(\lambda x, \lambda\rho))}{(\lambda\rho)^\alpha} 
    = \sup_{x\in\R^n, \rho\in(0,\lambda]} \frac{\nu(B(x, \rho))}{\rho^\alpha} 
    \leq \langle\nu \rangle_\alpha.
\end{align*}
Now the key observation is that we can cover $\frac{1}{\lambda} B$ by a finitely overlapping family of translates of $B$, which we denote by $\{B(x_i,1/4)\}_{i\in \Lambda}$. By the support property, (\ref{Eq:main_0}) actually means that
\begin{equation*}
    \norm{\AA f}_{L
    _x^q(B,\nu)(L^s_r(I))}\leq C \langle\nu\rangle_\alpha^{\frac{1}{q}}\norm{f}_{L^p(B(0,5))}
\end{equation*}
for any $f$ and $\nu$, which by translation invariance implies that
\begin{equation*}
    \big\Vert\AA \Tilde{h}\big\Vert_{L_x^q(B(x_i,1/4),\Tilde{\nu})(L^s_r(I))} 
    \leq C \langle \Tilde{\nu} \rangle_\alpha^{\frac{1}{q}} \big\Vert\Tilde{h}\big\Vert_{L^p(B(x_i,5))}.
\end{equation*}
Note also that the family $\{B(x_i,5)\}_{i\in\Lambda}$ is still finitely overlapping. Thus we can continue the above computations as follows: By triangle inequality 
\begin{align*}
    \norm{\AA h}_{L
    _x^q(B,\nu)(L^s_r(\lambda I))}
    &= \lambda^{\frac{\alpha}{q}+\frac{1}{s}} \Big(\int_{\frac{1}{\lambda} B} \Big( \int_{I} |\AA \Tilde{h}(x,\Tilde{r})|^s \dd \Tilde{r} \Big)^\frac{q}{s}\dd\Tilde{\nu}(x)\Big)^{\frac{1}{q}}\\
    &\leq \lambda^{\frac{\alpha}{q}+\frac{1}{s}} 
    \Big( \sum_{i\in\Lambda} \big\Vert\AA \Tilde{h}\big\Vert_{L_x^q(B(x_i,1/4),\Tilde{\nu})(L^s_r(I))}^q  \Big)^{\frac{1}{q}}.
\end{align*}
Apply \eqref{Eq:main_0} to each term, 
\begin{equation*}
    \lambda^{\frac{\alpha}{q}+\frac{1}{s}} 
    \Big( \sum_{i\in\Lambda} \big\Vert\AA \Tilde{h}\big\Vert_{L_x^q(B(x_i,1/4),\Tilde{\nu})(L^s_r(I))}^q  \Big)^{\frac{1}{q}}\leq  C \langle \Tilde{\nu} \rangle_\alpha^{\frac{1}{q}} \lambda^{\frac{\alpha}{q}+\frac{1}{s}} 
    \Big( \sum_{i\in\Lambda} \big\Vert\Tilde{h}\big\Vert_{L^p(B(x_i,5))}^q  \Big)^{\frac{1}{q}}.
\end{equation*}
Since $q\geq p$, $\ell^p\hookrightarrow\ell^q$, 
\begin{equation*}
    C \langle \Tilde{\nu} \rangle_\alpha^{\frac{1}{q}} \lambda^{\frac{\alpha}{q}+\frac{1}{s}} 
    \Big( \sum_{i\in\Lambda} \big\Vert\Tilde{h}\big\Vert_{L^p(B(x_i,5))}^q  \Big)^{\frac{1}{q}}\leq C \langle \nu \rangle_\alpha^{\frac{1}{q}} \lambda^{\frac{\alpha}{q}+\frac{1}{s}}
    \Big( \sum_{i\in\Lambda} \big\Vert\Tilde{h}\big\Vert_{L^p(B(x_i,5))}^p  \Big)^{\frac{1}{p}}.
\end{equation*}
(If $q < p$, then we have to apply H\"older's inequality instead, and there will be an additional loss $\lesssim_{n} \lambda^{-n(\frac{1}{q}-\frac{1}{p})}$.) Recall that $\{B(x_i,5)\}_{i\in\Lambda}$ is finitely overlapping, 
\begin{equation}\label{Eq:AhForLambdaI}
    \begin{aligned}
    \norm{\AA h}_{L
    _x^q(B,\nu)(L^s_r(\lambda I))}
    &\leq C \langle \nu \rangle_\alpha^{\frac{1}{q}} \lambda^{\frac{\alpha}{q}+\frac{1}{s}}
    \Big( \sum_{i\in\Lambda} \big\Vert\Tilde{h}\big\Vert_{L^p(B(x_i,5))}^p  \Big)^{\frac{1}{p}} \\
    & \lesssim_n 
    C \langle \nu \rangle_\alpha^{\frac{1}{q}} \lambda^{\frac{\alpha}{q}+\frac{1}{s}}    \big\Vert\Tilde{h}\big\Vert_{L^p(\R^n)}\\
    &= 
    C \langle \nu \rangle_\alpha^{\frac{1}{q}} \lambda^{\frac{\alpha}{q} + \frac{1}{s} - \frac{n}{p}}
    \norm{h}_{L^p(\R^n)}
\end{aligned}
\end{equation}

If $s \geq q$, then by taking $\lambda = 2^{-k}$ and summing over $k \in \mathbb{N}$, we get
\begin{align*}
    \norm{F_x(r)}_{L
    _x^q(B,\nu)(L^s_r[0,2])} 
    &= \norm{\AA h(x,r)r}_{L
    _x^q(B,\nu)(L^s_r[0,2])}\\
    &\leq 
    \Big( \int_{B}\Big( \sum_{k\in\mathbb{N}} 2^{-(k-1)s} \int_{2^{-k} I}|\AA h(x,r)|^s \dd r \Big)^\frac{q}{s} \dd\nu(x)\Big)^{\frac{1}{q}}.
\end{align*}
Since $s\geq q$, $\ell^q\hookrightarrow\ell^s$, 
\begin{align*} 
   \Big( \int_{B}\Big( \sum_{k\in\mathbb{N}} 2^{-(k-1)s} \int_{2^{-k} I}|\AA h(x,r)|^s \dd r \Big)^\frac{q}{s} \dd\nu(x)\Big)^{\frac{1}{q}}
    &\leq
    \Big( \int_{B}\Big( \sum_{k\in\mathbb{N}} 2^{-(k-1)q} \norm{\AA h}_{L_r^s(2^{-k}I)}^q \Big) \dd\nu(x)\Big)^{\frac{1}{q}}\\
    &=
    \Big( \sum_{k\in\mathbb{N}} 2^{-(k-1)q} \norm{\AA h}_{L_x^q(B,\nu)(L_r^s(2^{-k}I))}^q  \Big)^{\frac{1}{q}}.
\end{align*}
Plugging \eqref{Eq:AhForLambdaI}, we obtain 
\begin{align*}
     \norm{F_x(r)}_{L
    _x^q(B,\nu)(L^s_r[0,2])}  
    &\leq  \Big( \sum_{k\in\mathbb{N}} 2^{-(k-1)q} \norm{\AA h}_{L_x^q(B,\nu)(L_r^s(2^{-k}I))}^q  \Big)^{\frac{1}{q}}\\
    &\lesssim_n
    C \langle \nu \rangle_\alpha^{\frac{1}{q}}
    \Big( \sum_{k\in\mathbb{N}} 2^{-(k-1)q} 2^{-k(\frac{\alpha}{q}+\frac{1}{s}-\frac{n}{p})q}  \Big)^{\frac{1}{q}}
    \norm{h}_{L^p(\R^n)}. 
\end{align*}

If $s<q$, then by using the Minkowski inequality instead of $\ell^q \hookrightarrow \ell^s$, we get
\begin{align*}
    \norm{F_x(r)}_{L
    _x^q(B,\nu)(L^s_r[0,2])} 
    \lesssim_n 
    C \langle \nu \rangle_\alpha^{\frac{1}{q}}
    \Big( \sum_{k\in\mathbb{N}} 2^{-(k-1)s} 2^{-k(\frac{\alpha}{q}+\frac{1}{s}-\frac{n}{p})s}  \Big)^{\frac{1}{s}}
    \norm{h}_{L^p(\R^n)}.
\end{align*}
In both cases, as long as $\frac{\alpha}{q}+\frac{1}{s}-\frac{n}{p} + 1 > 0$, we would get
\begin{align*}
    \norm{F_x(r)}_{L
    _x^q(B,\nu)(L^s_r[0,2])} \lesssim_{p,q,s,n} 
    \langle \nu \rangle_\alpha^{\frac{1}{q}} \norm{h}_{L^p(\R^n)}.
\end{align*}
(If $q<p$, then what we need becomes $\frac{\alpha}{q}+\frac{1}{s}-\frac{n}{q} + 1 > 0$.) 

Now let us briefly comment on what will happen if we take $\lambda > 1$ at the beginning. In this case, since $\frac{1}{\lambda}B \subseteq B$, no covering procedure is required, and we have
\begin{align*}
    \norm{\AA h}_{L
    _x^q(B,\nu)(L^s_r(\lambda I))}
    \leq
    \lambda^{\frac{\alpha}{q}+\frac{1}{s}} \Big(\int_{B} \Big( \int_{I} |\AA \Tilde{h}(x,\Tilde{r})|^s \dd \Tilde{r} \Big)^\frac{q}{s}\dd\Tilde{\nu}(x)\Big)^{\frac{1}{q}}
\end{align*}
in the first step. On the other hand, by covering $B(x,\rho)$ with $\lesssim_n \rho^n\leq \lambda^n$ unit balls whenever $\rho > 1$, we have
\begin{align*}
    \langle \Tilde{\nu} \rangle_\alpha 
    = \sup_{x\in\R^n, \rho\in(0,\lambda]} \frac{\nu(B(x, \rho))}{\rho^\alpha}
    \lesssim_n \lambda^n \langle\nu \rangle_\alpha.
\end{align*}
So we immediately get
\begin{align*}
    \norm{\AA h}_{L
    _x^q(B,\nu)(L^s_r(\lambda I))} 
    \lesssim_n 
    C \langle\nu \rangle_\alpha^{\frac{1}{q}} \lambda^{\frac{\alpha}{q} + \frac{1}{s} - \frac{n}{p} + \frac{n}{q}}
    \norm{h}_{L^p(\R^n)}.
\end{align*}
In general, the infinite sum over $\lambda = 2^k$ ($k\in\mathbb{N}$) is divergent. But for our purpose, the important thing is that for any $1 < T <\infty$, we have
\begin{align*}
    \norm{F_x(r)}_{L
    _x^q(B,\nu)(L^s_r[1,T])} 
    &= \norm{\AA h(x,r)r}_{L
    _x^q(B,\nu)(L^s_r[1,T])}\\
    &\leq \sum_{k=0}^{\lceil \log T \rceil} 
    \norm{\AA h(x,r)r}_{L
    _x^q(B,\nu)(L^s_r(2^k I))}\\
    &\lesssim_n 
    C \langle\nu \rangle_\alpha^{\frac{1}{q}} 
    \sum_{k=0}^{\lceil \log T \rceil}
    2^{k(\frac{\alpha}{q} + \frac{1}{s} - \frac{n}{p} + \frac{n}{q}) + k}
    \norm{h}_{L^p(\R^n)}\\
    &\lesssim_n 
    C T^{2n+2} \langle\nu \rangle_\alpha^{\frac{1}{q}} \norm{h}_{L^p(\R^n)}
\end{align*}
by the triangle inequality.

Finally, we summarize what we have got so far in the following proposition:
\begin{proposition}\label{Prop:local_[0,T]}
    Suppose that {\rm (\ref{Eq:main_0})} holds true for any $f$ and $\nu$, then we have
    \begin{align*}        \norm{F_x(r)}_{L_x^q(B,\nu)(L^s_r[0,T])} 
    \lesssim_{p,q,s,n}
    \max\{1,T\}^{2n+2} 
    \langle\nu \rangle_\alpha^{\frac{1}{q}} \norm{h}_{L^p(\R^n)}
    \end{align*}
    as long as $\frac{\alpha}{q}+\frac{1}{s} + 1 > \frac{n}{\min(p,q)}$.
\end{proposition}
In fact, by further exploiting the support property as before, we can bootstrap Proposition~\ref{Prop:local_[0,T]} to the following global version, whose proof we omit.
\begin{proposition}\label{Prop:global_[0,T]}
    Suppose that {\rm (\ref{Eq:main_0})} holds true for any $f$ and $\nu$, then we have
    \begin{align*}        \norm{F_x(r)}_{L_x^q(\R^n,\nu)(L^s_r[0,T])} 
    \lesssim_{p,q,s,n}
    \max\{1,T\}^{3n+2} 
    \langle\nu \rangle_\alpha^{\frac{1}{q}} \norm{h}_{L^p(\R^n)}
    \end{align*}
    as long as $q\geq p$ and $\frac{\alpha}{q}+\frac{1}{s} + 1 > \frac{n}{p}$.
\end{proposition}

\subsection{Proof of Theorem~\ref{Thm:wave_eqn}}\label{subsec:Proof_wave_eqn}
With tools and results established in the previous two subsections, we can now start proving Theorem~\ref{Thm:wave_eqn}. Fix $n=2$. Proposition~\ref{Prop:local_[0,T]} and \ref{Prop:global_[0,T]} together with Proposition~\ref{Prop:Abel_transform} ($F = F_x(r)$, $t=T$) immediately yield
\begin{theorem}\label{Thm:Holder}
Suppose that {\rm (\ref{Eq:main_0})} holds true for any $f$ and $\nu$, then we have
\begin{align}\label{Eq:local_Holder}
\norm{u(x,t)}_{L_x^q(B,\nu)(C_{\rm loc}^{\frac{1}{2} - \frac{1}{s}}[0,T])} 
\lesssim_{p,q,s,T}
\langle\nu \rangle_\alpha^{\frac{1}{q}} \norm{h}_{L^p(\R^2)}
\end{align}
as long as $\frac{\alpha}{q}+\frac{1}{s} + 1 > \frac{n}{\min(p,q)}$, and
\begin{align}\label{Eq:global_Holder}
\norm{u(x,t)}_{L_x^q(\R^2,\nu)(C_{\rm loc}^{\frac{1}{2} - \frac{1}{s}}[0,T])} 
\lesssim_{p,q,s,T}
\langle\nu \rangle_\alpha^{\frac{1}{q}} \norm{h}_{L^p(\R^2)}
\end{align}
as long as $q\geq p$ and $\frac{\alpha}{q}+\frac{1}{s} + 1 > \frac{n}{p}$.
\end{theorem}

For our purpose, we claim that (\ref{Eq:main_0}) holds for 
\begin{itemize}
    \item $p=q>3$ and $s = \frac{3}{1-\alpha}$ when $\alpha\in(0,1]$;
    \item $p=q>2$ and $s = \frac{2}{2-\alpha}$ when $\alpha \in (\frac{3}{2},2]$.
\end{itemize}
By interpolating Theorem~\ref{Thm:MainTheorem_0} with the trivial bound for $(\frac{1}{p}, \frac{1}{q}, \frac{1}{s}) = (0,0,0)$, we immediately verify the cases $\alpha \in (0,1)$ and $\alpha \in (\frac{3}{2}, 2)$.  When $\alpha=1$, the claim follows from \cite[Theorem~1.13]{CDK25} and \cite[Lemma~3.2]{HKL22}. When $\alpha=2$, the claim follows directly from \cite[Corollary~1.4]{HKL22}.

Since we always have $p=q>2$ and $\frac{n}{p} = \frac{2}{p} < 1 < \frac{\alpha}{q} + \frac{1}{s} + 1$, we can apply (\ref{Eq:global_Holder}) in Theorem~\ref{Thm:Holder}:
\begin{align}\label{Eq:R2_global_Holder}
\norm{u(x,t)}_{L_x^q(\R^2,\nu)(C_{\rm loc}^{\beta}[0,T])} 
\lesssim_{\alpha,T}
\langle\nu \rangle_\alpha^{\frac{1}{q}} \norm{h}_{L^p(\R^2)}
\end{align}
for 
$\beta = \frac{2\alpha + 1}{6}$ if $p=q>3$ and $\alpha \in (0,1]$, and for $\beta = \frac{\alpha-1}{2}$ if $p=q>2$ and $\alpha \in (\frac{3}{2},2]$. 

Let us only show the first part of Theorem~\ref{Thm:wave_eqn}: If $h \in L_{\rm loc}^{3+}(\R^2)$, then $\dim_H E(u,\beta) \leq 3\beta-\frac{1}{2}$, $\forall\,\beta\in(\frac{1}{6},\frac{1}{2}]$. The second part follows the same line of reasoning. We argue by contradiction. Suppose that $\dim_H E(u,\beta) > 3\beta-\frac{1}{2}$ for some $\beta\in(\frac{1}{6},\frac{1}{2}]$, then there exists some $0<T<\infty$ such that $\dim_H E(u,\beta,T) > 3\beta-\frac{1}{2}$, where $E(u,\beta,T) \coloneqq \{x\in\mathbb{R}^2: u(x,\cdot) \not\in C^\beta[0,T]\}$ (since $E(u,\beta) = \medcup_{T\in \mathbb{N}} E(u,\beta,T)$ by definition). By Frostman's lemma, there exists a Borel measure $\nu_{*}$ satisfying $\spt\, \nu_* \subseteq E(u,\beta,T)$, $0 < \nu_*(E(u,\beta,T)) < \infty$, and $\langle \nu_* \rangle_{\alpha} \leq 1$ ($\alpha \coloneqq 3\beta - \frac{1}{2} \in (0,1]$). By covering $\R^2$ with countably many translates of $B$, there must exist some disc $B_* = B(x_*, 1/4)$ such that $0 < \nu_*(B_*) < \infty$. Note that in Poisson's formula (\ref{Eq:Poisson}), the convolution kernel is localized in a ball of radius $t$. So if we consider $L_x^q(B_*,\nu_*)$ on the left-hand side of (\ref{Eq:R2_global_Holder}), then only the portion of $h$ supported on $B(x_*,4T+1)$ will be relevant. To be more precise, since $h\in L_{\rm loc}^{3+}(\R^2)$, $\beta = \frac{2\alpha+1}{6}$, and $\alpha \in (0,1]$, we can apply (\ref{Eq:R2_global_Holder}) with $\nu = \nu_*|_{B_*}$ to get
\begin{align*}
\norm{u(x,t)}_{L_x^q(B_*),\nu_*)(C^{\beta}[0,T])} 
\lesssim_{\alpha,T}
\norm{h}_{L^p(B(x_*,4T+1))} < \infty
\end{align*}
for any $p=q>3$. However, $\spt\, \nu_* \subseteq E(u,\beta,T)$ implies that $\norm{u(x,\cdot)}_{C^{\beta}[0,T]}$ = $\infty$ for any $x \in \spt\,\nu_*$. This together with $\nu(B_*) > 0$ forces the left-hand side of the above inequality to be infinity, which contradicts the finiteness of the right-hand side. And so the proof of Theorem~\ref{Thm:wave_eqn} is completed.

\begin{remark}
    Looking back at the proof, it is the ``localization'' property of $u(x,t)$ that allows us to deal with $h \in L_{\rm loc}^{p}$, not just $h \in L^p$. This ultimately comes down to the finite propagation speed of disturbances: local singularity only depends on nearby information. This is a universal feature of wave equations that also holds in higher dimensions. In particular, it is easy to see that it suffices to only use the local estimate (\ref{Eq:local_Holder}) in Theorem~\ref{Thm:Holder} for the above arguments to work, if we only care about the Hausdorff dimension of exceptional set estimates. This modified approach loosens the requirements for $q$ as we no longer need $q\geq p$, although it does not matter in our cases. Here we choose to use (\ref{Eq:global_Holder}) simply because it seems good to state a global estimate like (\ref{Eq:R2_global_Holder}). 
\end{remark}

\subsection{Historical remarks}\label{subsec:further_direction}
In this final subsection, we will provide additional historical remarks and discuss possible further directions regarding the regularity of the wave equations in time $t$, or the closely related $\AA^n$ in radius $r$ for general $n\geq2$. 

We have mentioned that our Theorem~\ref{Thm:wave_eqn} on H\"older regularity, although only for $n=2$, is the first of its kind in the literature. Possible reasons are that classical Fourier analytic techniques do not work very well in the weighted mixed-norm setting, and if $n\geq 3$, then even the well-posedness of the problem might fail. For example, when $n=3$, by the well-known Kirchhoff's formula, the solution of the wave equation is represented exactly via $\AA^3$ due to the Huygens property, and so Abel-type regularity-improving results like Proposition~\ref{Prop:Abel_transform} are no longer available. One cannot expect $u(x,t)$ to be jointly continuous in both $x$ and $t$ if we only have $h \in L^p(\R^3)$, and so it is subtle to even formulate the problem of bounding exceptional sets for H\"older regularity in $t$. 

Historically, similar issues also happens in the study of divergence sets of dispersive equations mentioned in Section~\ref{subsec:History}. A typical way \cite{BBCR11} of remedying this is to fix one multiplier representation of $u$, and interpret pointwise convergence to the initial data as up to a null set. In Theorem~\ref{Thm:wave_eqn}, we do not need to worry about choosing a specific interpretation of weak solutions due to the continuity in Section~\ref{subsec:prelim}, which is kind of lucky as our main results are in the admissible regime $p>2$. In other words, everything is unique and deterministic at least in $\R^2$, which is physically satisfactory to see. 

Nevertheless, one can still ask what would get unambiguous regularity results in higher dimensions if we allow $h$ to lie in Sobolev spaces $L_\gamma^p$ for reasonable $\gamma>0$. As mentioned in Item (a) of Section~\ref{subsec:add_background}, this requires a different framework from that for $L^p$. More precisely, our discretization and duality arguments should be modified accordingly, and it is possibly a brand new problem. Also, when $\gamma$ is large, probably it would no longer be appropriate to put mixed-norms like $L_x^q(\nu)(L_r^s(I))$ on the left-hand side, since the expected regularity can exceed the best integrability $s=\infty$. For these reasons, we only focus on $h \in L^p$ in this paper. Similarly, we also do not consider wave equations with nonzero initial displacement, which would canonically involve the Sobolev regime. 

Anyway, the discussions so far provide motivation for bounding
\begin{align*}
    \norm{\AA^n f}_{L_x^q(\nu)(\tnorm{\cdot})},
\end{align*}
where $\tnorm{\cdot}$ can be a general norm in $r$. Unfortunately, there seems to be very few results towards this direction. When $\nu = \mathcal{L}^n$ (unweighted case), there are some noteworthy early results: 
\begin{itemize}
    \item Peyrière and Sjölin \cite{PS78} obtained estimates where $q=2$ and $\tnorm{\cdot}$ is an $L^2$-based Besov space. Via embedding inequalities, this implies H\"older regularity in $r$ outside of a Lebesgue null set of $x$ when $n \geq 3$. Endpoint results were later obtained by Oberlin and Stein \cite{OS82} via complex interpolation of analytic families of operators.

    \item Sjölin \cite{sjolin83} proved sharp estimates where $q=2$ and $\tnorm{\cdot}$ is an $L^2$-based Sobolev space (the right-hand side is $\norm{f}_{L^p(\R^n)}$). The proof relies heavily on properties of Bessel functions. Via embedding inequalities, this implies bounds for $\tnorm{\cdot} = L^s, {\rm BMO}, C^\beta$ according to different ranges of $p$. As an application, results on convergence of almost everywhere convergence of Fourier integrals on $L^2(\R^n)$ ($n\geq 2$) were obtained. 

    \item Sjölin's result \cite{sjolin83} were later extended to general $p \leq q \leq p'$ ($p\geq 2$) cases \cite{sjolin85}, which were in turn extended to more general cases where $\AA^n$ is replaced with generalized spherical means \cite{borjeson88}. These include the solutions to linear wave equations as a special case.
\end{itemize} 

All the above works relies heavily on Fourier analytic methods. In particular, $L^2$-based estimates plays a central role for the use of Plancherel's identity. It is unclear to us to what extent can these classical methods be adapted to $\nu$-weighted settings. In contrast, our geometric framework seems more flexible to such extensions. Even when $\nu=\mathcal{L}^n$, there are two major limitations of these works. Firstly, almost all the results require $p\leq 2$, which is exactly outside of the admissible range $p>2$ for continuity in Section~\ref{subsec:prelim}. Nevertheless, we should point out that when $p=n=2$, the $\tnorm{\cdot}={\rm BMO}$ results matches Case (2) in Theorem~\ref{Thm:MainTheorem_0} when $\alpha=2$. Secondly, the regularity results are always only local in time, i.e., involving a bump $\varphi \in C_c^\infty(0,\infty)$ in $\tnorm{\cdot}$. In contrast, H\"older estimates in our Theorem~\ref{Thm:wave_eqn} include $0$ as a closed endpoint. 
\section{Necessary conditions}\label{sec:necessary}
We now test several examples to prove Proposition~\ref{Prop:necessary} ($n=2$). Let $C(x,r)$ be the circle centered at $x$ with radius $r$ and $C_\delta(x,r)$ be the corresponding $\delta$-neighborhood. For simplicity, we only focus on $\AA_\delta$, and the examples and arguments for $\AA$ are exactly the same. Examples~1,2,3 are of typical Knapp type, and Example~4,5 are refined self‑replicating versions of Example~2,1, respectively.
\begin{example}
    Take $f=\chi_{C_\delta(0,1)}$, then $\norm{f}_{L^p(\R^2)}\sim \delta^{\frac{1}{p}}$.
    \begin{equation}\label{Eq:ADelta1}
        \AA_\delta f(x,r)\sim \frac{1}{\delta}|C_\delta(x,r)\medcap C_\delta(0,1)|\sim 1, \text{ for $x\in B(0,\delta), r\in [1,1+\delta]$}.
    \end{equation}
    Let $\nu$ supported in $B=B(0,1/4)$ satisfying $\langle\nu\rangle_\alpha \sim 1$ and $\nu(B(0,\delta))=\delta^{\alpha}$. Then \eqref{Eq:ADelta1} implies that 
    \begin{equation*}
         \norm{\AA_\delta f}_{L
        _x^q(B,\nu)(L^s_r(I))}\sim \delta^{\frac{1}{s}}\nu(B(0,\delta))^{\frac{1}{q}}=\delta^{\frac{1}{s}+\frac{\alpha}{q}}.
    \end{equation*}
    This gives 
    \begin{equation*}
        \delta^{\frac{1}{s}+\frac{\alpha}{q}}\lessapprox\delta^{\frac{1}{p}}\Longleftrightarrow \frac{1}{s}+\frac{\alpha}{q}\geq \frac{1}{p}.
    \end{equation*}
\end{example}

\begin{example}
   Let $R$ be a $\sqrt \delta\times \delta$-rectangle centered at $0$ whose long side is parallel to $x$-axis. Take $f=\chi_{R}$, then $\norm{f}_{L^p(\R^2)}\sim \delta^{\frac{3}{2p}}$.
    \begin{equation}\label{Eq:ADelta2}
        \AA_\delta f(x,r)\sim \frac{1}{\delta}|C_\delta(x,r)\medcap R|\sim \delta^{\frac{1}{2}}, \text{ for $x\in \tilde R, r\in I_\delta$},
    \end{equation}
    where $\tilde R$ is a $1\times \sqrt \delta$-rectangle dual to $R$ and $I_\delta$ is some interval with length $\delta$. When $\alpha\in [1,2]$, 
    take $\nu(x)=\delta^{\frac{\alpha}{2}-1}\chi_{\tilde R}(x)$. Then $\langle\nu\rangle_\alpha \sim 1$, and 
    \eqref{Eq:ADelta2} implies that 
    \begin{equation*}
         \norm{\AA_\delta f}_{L
        _x^q(B,\nu)(L^s_r(I_\delta))}\sim \delta^{\frac{1}{2}}\delta^{\frac{1}{s}}\nu(\tilde R)^{\frac{1}{q}}=\delta^{\frac{1}{2}}\delta^{\frac{1}{s}+\frac{\alpha-1}{2q}}.
    \end{equation*}
    This gives 
    \begin{equation*}
        \delta^{\frac{1}{2}+\frac{1}{s}+\frac{\alpha-1}{2q}}\lessapprox\delta^{\frac{3}{2p}}\Longleftrightarrow 1+\frac{2}{s}+\frac{\alpha-1}{q}\geq \frac{3}{p}.
    \end{equation*}
    When $\alpha\in (0,1]$, 
    take $\nu(x)=\delta^{-\frac{1}{2}}\chi_{\tilde R}(x)$. Then $\langle\nu\rangle_\alpha \sim 1$, and \eqref{Eq:ADelta2} implies that 
    \begin{equation*}
         \norm{\AA_\delta f}_{L
        _x^q(B,\nu)(L^s_r(I))}\sim \delta^{\frac{1}{2}}\delta^{\frac{1}{s}}\nu(\tilde R)^{\frac{1}{q}}=\delta^{\frac{1}{2}}\delta^{\frac{1}{s}}.
    \end{equation*}
    This gives 
    \begin{equation*}
        \delta^{\frac{1}{2}+\frac{1}{s}} \lessapprox
        \delta^{\frac{3}{2p}}\Longleftrightarrow 1+\frac{2}{s} \geq \frac{3}{p}.
    \end{equation*}
\end{example}

\begin{example}
  Take $f=\chi_{B(0,\delta)}$, then $\norm{f}_{L^p(\R^2)}\sim \delta^{\frac{2}{p}}$.
    \begin{equation}\label{Eq:ADelta3}
        \AA_\delta f(x,r)\sim \frac{1}{\delta}|C_\delta(x,r)\medcap B(0,\delta)|\sim \delta, \text{ for $x\in B, r\in I_\delta$},
    \end{equation}
   
    Take $\nu$ such that $\nu(B)=1$ and $\langle\nu\rangle_\alpha \sim 1$. Then
    \eqref{Eq:ADelta3} implies that 
    \begin{equation*}
         \norm{\AA_\delta f}_{L
        _x^q(B,\nu)(L^s_r(I))}\sim \delta\delta^{\frac{1}{s}}\nu(B)^{\frac{1}{q}}=\delta^{1+\frac{1}{s}}.
    \end{equation*}
    This gives 
    \begin{equation*}
        \delta^{1+\frac{1}{s}}\lessapprox\delta^{\frac{2}{p}}\Longleftrightarrow 1+\frac{1}{s}\geq \frac{2}{p}.
    \end{equation*}
\end{example}
\begin{example}
    Let $\alpha\in [1,2]$ and $R$ be a disjoint union of $\sim\delta^{\frac{1-\alpha}{2}}$ many $\delta\times \sqrt{\delta}$ rectangles, with the long side parallel to $x$-axis and spacing $\delta^{\frac{\alpha-1}{2}}$. Take $f=\chi_R$, then $\norm{f}_{L^p(\R^2)}\sim\delta^{\frac{4-\alpha}{2p}}$. 
    \begin{equation}\label{Eq:ADelta4}
        \AA_\delta f(x,r)\gtrsim \frac{1}{\delta}\delta^{3/2}\sim \sqrt \delta, \text{ for $x\in B, r\in I_\delta$}.
    \end{equation}
    Take $\nu(x)=\delta^{\frac{\alpha}{2}-1}\chi_{ \Tilde{R}}(x)$, where $\Tilde{R}$ is a disjoint union of parallel $1 \times \sqrt{\delta}$ rectangles dual to those in $R$. Hence $\nu(B)\sim 1$ and $\langle\nu\rangle_\alpha \sim 1$. Then \eqref{Eq:ADelta4} implies that 
     \begin{equation*}
         \norm{\AA_\delta f}_{L
        _x^q(B,\nu)(L^s_r(I))}\sim \sqrt \delta\delta^{\frac{1}{s}}\nu(B)^{\frac{1}{q}}
        \sim \delta^{\frac{1}{2}+\frac{1}{s}}.
    \end{equation*}
    This gives 
    \begin{equation*}
        \delta^{\frac{1}2{}+\frac{1}{s}}\lessapprox\delta^{\frac{4-\alpha}{2p}}\Longleftrightarrow 1+\frac{2}{s}\geq \frac{4-\alpha}{p}.
    \end{equation*}
\end{example}

\begin{example}
    Let $\alpha\in (0,1]$, and $\Tilde{R}$ be the disjoint union of $\sim\delta^{-\alpha}$ many $\delta$-balls $\{B(c_i, \delta)\}_i$ aligned along the $x$-axis with $c_i = (i\delta^\alpha, 0)$. Take $\nu(x) = \delta^{\alpha-2}\chi_{\Tilde{R}}(x)$, then $\nu(B) \sim 1$ and $\langle\nu\rangle_\alpha \sim 1$. Take $f = \chi_{R}$, where $R = \medcup_i C_\delta(c_i,1)$, then $\norm{f}_{L^p(\R^2)}= |R|^{\frac{1}{p}} \lesssim \delta^{\frac{1-\alpha}{p}}$. Now \eqref{Eq:ADelta1} implies that 
     \begin{equation*}
         \norm{\AA_\delta f}_{L_x^q(B,\nu)(L^s_r(I))}\sim \delta^{\frac{1}{s}}\nu(B)^{\frac{1}{q}}
        \sim \delta^{\frac{1}{s}}.
    \end{equation*}
    This gives 
    \begin{equation*}
        \delta^{\frac{1}{s}}\lessapprox\delta^{\frac{1-\alpha}{p}}\Longleftrightarrow \frac{1}{s}\geq \frac{1-\alpha}{p}.
    \end{equation*}
\end{example}

\bibliographystyle{plainnat}
\bibliography{References}

\end{document}